\documentclass{amsart}

\usepackage{amsmath, amssymb}
\usepackage{amsfonts}
\usepackage{latexsym}
\usepackage{amsthm}
\usepackage{mathrsfs}
\usepackage{color}
\usepackage[all]{xy}
\usepackage[foot]{amsaddr}

\makeatletter
\@addtoreset{equation}{section}
\makeatother

\theoremstyle{definition}
\newtheorem{defn}[equation]{Definition}
\theoremstyle{plain}
\newtheorem{thm}[equation]{Theorem}
\newtheorem{prob}[equation]{Problem}
\newtheorem{prop}[equation]{Proposition}
\newtheorem{fact}[equation]{Fact}
\newtheorem{cor}[equation]{Corollary}
\newtheorem{lem}[equation]{Lemma}
\theoremstyle{remark}
\newtheorem{rem}[equation]{Remark}
\newtheorem{ex}[equation]{Example}

\newcommand{\pullbackcorner}[1][dr]{\save*!/#1-1.2pc/#1:(-1,1)@^{|-}\restore}

\title[Indices on manifolds with fibered boundaries]{A topological approach to indices of geometric operators on manifolds with fibered boundaries}
\author[M. Yamashita]{Mayuko Yamashita}
\address{Graduate School of Mathematical Sciences, the University of Tokyo, 3-8-1 Komaba Meguro-ku Tokyo 153-8914, Japan}
\email{mayuko@ms.u-tokyo.ac.jp}

\begin{document}
\begin{abstract}
In this paper, we investigate topological aspects of indices of twisted geometric operators on manifolds equipped with fibered boundaries. 
We define $K$-groups relative to the pushforward for boundary fibration, and show that indices of twisted geometric operators, defined by complete $\Phi$ or edge metrics, can be regarded as the index pairing over these $K$-groups. 
We also prove various properties of these indices using groupoid deformation techniques.  
Using these properties, we give an application to the localization problem of signature operators for singular fiber bundles. 
\end{abstract}
\maketitle
\tableofcontents
\section{Introduction}
In this paper, we consider pairs of the form $(M, \pi : \partial M \to Y)$, where $M$ is a compact manifold with boundary $\partial M$ which is a closed manifold, and $\pi : \partial M \to Y$ is a smooth submersion, equivalently a fiber bundle structure, to a closed manifold $Y$. 
We call such pairs {\it manifolds with fibered boundaries}. 
We investigate topological aspects of indices of geometric operators, namely $spin^c$-Dirac operators, signature operators and their twisted versions, on such manifolds. 
There are two purposes of this paper. 
The first one is to formulate the index pairing on such manifolds. 
We define $K$-groups relative to the pushforward for boundary fibration, and show that indices of twisted geometric operators, defined by complete metrics of the form (\ref{intro_phi_metric}), can be regarded as the index pairing over these $K$-groups. 
The second one is to 
prove properties of these indices using groupoid deformation techniques.  
Using these properties, we give an application to the localization problem of signature operators for singular fiber bundles. 

Singular spaces arise in various areas in mathematics. 
In particular, stratified pseudomanifolds include many important examples of singular spaces, such as manifolds with corners and algebraic varieties. 
Manifolds with fibered corners arise as resolutions of stratified pseudomanifolds \cite{ALMP}, and the simplest case, stratified manifolds of depth $1$, corresponds to manifolds with fibered boundaries. 
There are some classes of metrics which is suited to encode the singularities of such spaces, including (complete) $\Phi$-metrics and edge metrics.  
To study pseudodifferential operators with respect to such metrics, the corresponding pseudodifferential calculi, called $\Phi$-calculus and $e$-calculus, were introduced by \cite{MM} and \cite{Maz}. 
Since then, analysis of elliptic operators in these calculi, in particular Fredholm theory and spectral theory of geometric operators, has been developed by many authors and there have been many applications to geometry of singular spaces, for example see \cite{ALMP}, \cite{DLR} and \cite{LMP}. 

Most of those works are analytic in nature, in the sense that they analyze individual operators under these calculi. 
On the other hand, it is natural to expect more topological description of Fredholm indices of these operators, as in the case of closed manifolds. 
One of related works in this direction is \cite{MR}, in which they formulate the index theorem for fully elliptic operators, as an equality of analytic and topological indices defined on abelian groups of stable homotopy classes of full symbols $K_{\Phi-cu}(\phi)$, which corresponds to $K_1(\Sigma^{\mathring{M}}(G_\Phi))$ in our paper. 
We go in this direction further, and show that, once we fix a class $\star$ of geometric operators we are concerned with (for example $\star$ can be $spin^c$ or $sign$), the indices of twisted operators can be formulated in terms of the pairing on more primary $K$-groups, $K_*(\mathcal{A}^\star_\pi)$, ``$K$-groups relative to the $\star$ pushforward for boundary fibration''. 
This paper is considered as a step to understand elliptic theory on singular spaces from a more topological, or $K$-theoretical viewpoint. 

In order to explain our index pairing on manifolds with fibered boundaries, first we recall the index pairing on closed manifolds. 
Let $M$ be a closed even dimensional smooth manifold. 
Suppose we are given a complex vector bundle $E$ over $M$, and a Clifford module bundle $W$ over $M$, with a Dirac type operator $D_M$. 
Then we have the corresponding classes $[E] \in K^0(M)$ and $[D_M] \in K_0(M)$ in the $K$-theory and the $K$-homology of $M$, and the index pairing of $K^0(M)$ and $K_0(M)$ sends this pair to the index of the twisted Dirac operator, $D_M^E$, acting on the Clifford module bundle $W \hat{\otimes} E$, 
\begin{align}\label{intro_closed_eq}
\langle {\cdot}, {\cdot}\rangle : K^0(M) \otimes K_0(M) &\to \mathbb{Z}  \notag \\
\langle [E], [D_M]\rangle &= \mathrm{Ind}(D_M^E) . 
\end{align} 
Many examples of such an operator $D_M$ arise as ``geometric operators'' on compact manifolds, such as $spin^c$-Dirac operators and signature operators. 
Fixing a class of geometric operator $\star$ and a corresponding geometric structure on $M$ that determines the operator $D_M$, the above pairing is equivalently described as the pushforward $p_!^\star : K^0(M) \to \mathbb{Z}$ in $K$-theory. 

The first main purpose of this paper is to generalize this index pairing to the case of compact manifolds with fibered boundaries. 
It is stated in terms of $K$-theory for some $C^*$-algebras. 
The $C^*$-algebras depend on the ``geometric structure'' we choose to deal with, so here we explain the case of $spin^c$-structures. 

Assume we are given a compact even dimensional manifold with fibered boundary $(M, \pi : \partial M \to Y)$, and assume that the fibers of $\pi$ has dimension $n$. 
If $\pi$ is equipped with a $spin^c$-structure, 
we associate a $C^*$-algebra $\mathcal{A}_\pi$, whose $K$-groups fit into the long exact sequence
\[
\cdots\to K^{*}(M) \xrightarrow{\pi_! \circ i^*} {K}^{*-n}(Y) \to {K}_{*-n}(\mathcal{A}_\pi) \to {K}^{*+1}(M) \xrightarrow{\pi_! \circ i^*} \cdots
\]
where $\pi_!$ is the Gysin map in $K$-theory $\pi_! : K^*(\partial M) \to K^{*-n}(Y)$ defined by the fiberwise $spin^c$-structure, and $i$ is the inclusion $i : \partial M \to M$ (Definition \ref{def_A_pi} and Proposition \ref{A_pi_longexact}). 
Thus $K$-groups of this $C^*$-algebra, $K_*(\mathcal{A}_\pi)$, can be regarded as the $K$-groups relative to the $spin^c$-pushforward of the boundary fibration. 

From now on, we assume $n$, the dimension of $M$, is even. 
A pair of the form $(E, \tilde{D}_\pi^E)$, where $E$ is a complex vector bundle over $M$ satisfying $\pi_![E] = 0\in K^0(Y)$, and the operator $\tilde{D}_\pi^E$ is an invertible perturbation of the fiberwise $spin^c$-Dirac operators by lower order odd self-adjoint operators, gives a class $[(E, [\tilde{D}_\pi^E])] \in K_1(\mathcal{A}_\pi)$ (Lemma \ref{lem_K_class}; the bracket in $[\tilde{D}_\pi^E]$ means that we actually only have to consider the homotopy equivalence class of invertible perturbations). 
Furthermore, a pair $(P'_M, P'_Y)$ of (equivalence classes of) $spin^c$-structures on $M$ and $Y$ which is compatible with the one on $\pi$ at the boundary, gives a class in $K^1(\mathcal{A}_\pi)$. 
This is the element $[(P'_M, P'_Y)] \otimes_{\Sigma^{\mathring{M}}(G_\Phi)} \partial^{\mathring{M}}(G_\Phi) \in KK^1(\mathcal{A}_\pi, \mathbb{C})$ appearing in Theorem \ref{main_thm_K}. 

On the other hand, from the data $(E, \tilde{D}_\pi^E, P'_M, P'_Y)$ we can construct a Fredholm operator by using $\Phi$ or edge metrics, as we now explain. 
For a manifold with fibered boundary $(M, \pi : \partial M \to Y)$, natural classes of complete riemannian metrics on the interior arise as follows. 
First fix a splitting $T\partial M = \pi^*TY \oplus T^V\partial M$ and a collar structure near the boundary. 
Consider metrics on $\mathring{M}$ which are on the collar of the form
\begin{equation}\label{intro_phi_metric}
g_\Phi = \frac{dx^2}{x^4} \oplus \frac{\pi^*g_Y}{x^2} \oplus g_\pi \mbox{ and }
g_e = \frac{dx^2}{x^2} \oplus \frac{\pi^*g_Y}{x^2} \oplus g_\pi. 
\end{equation}
Here $g_Y$ and $g_\pi$ are some riemannian metrics on $TY$ and $T^V\partial M$, respectively, and $x$ is a normal coordinate of the collar. 
These are called rigid $\Phi$-metrics and rigid edge metrics in the literature, respectively. 
In this paper, we adopt pseudodifferential calculus on Lie groupoids, which is due to \cite{NWX}. 
As constructed in \cite{N}, the groupoids corresponding to $\Phi$ and edge metrics are of the form 
\begin{align*}
G_\Phi &= \mathring{M} \times \mathring{M} \sqcup \partial M \times_\pi \partial M \times_\pi TY \times \mathbb{R} \rightrightarrows M \\
G_e &= \mathring{M} \times \mathring{M} \sqcup \partial M \times_\pi \partial M \times_\pi (TY \rtimes \mathbb{R}^*_+) \rightrightarrows M. 
\end{align*}
Using the given (equivalence classes of) $spin^c$-structures, we can consider the $spin^c$-Dirac operator twisted by $E$ under these metrics, denoted by $D_\Phi^E$ and $D_e^E$. 
For $\Phi$-case, the operator $D_\Phi^E$ restricts to the boundary operator of the form $
D^E_\pi \hat{\otimes}1 + 1 \hat{\otimes}D_{TY \times \mathbb{R}}$. 
If we perturb this boundary operator to the invertible operator $\tilde{D}^E_\pi \hat{\otimes}1 + 1 \hat{\otimes}D_{TY \times \mathbb{R}}$, we get the Fredholmness of the operator on the interior and get the index, denoted by $\mathrm{Ind}_\Phi(P'_M,P'_Y,E,[\tilde{D}^E_\pi])$ (Definition \ref{def_index_perturbation_spinc}). 
The $e$-case is analogous, and we define $\mathrm{Ind}_e(P'_M,P'_Y,E, [\tilde{D}^E_\pi])$. 

Our main theorem, Theorem \ref{main_thm_K}, proves the equality
\[
\mathrm{Ind}_\Phi(P'_M, P'_Y, E, [\tilde{D}^E_\pi]) = [(E, [\tilde{D}^E_\pi])] \otimes_{\mathcal{A}_\pi} [(P'_M, P'_Y)] \otimes_{\Sigma^{\mathring{M}}(G_\Phi)} \partial^{\mathring{M}}(G_\Phi), 
\]
which can be regarded as a generalization of the index pairing (\ref{intro_closed_eq}). 

In the case of signature operators, the arguments proceed in parallel. 
If we are given a pair $(M, \pi : \partial M \to Y)$ such that $T^VM$ are oriented, 
then we associate a $C^*$-algebra $\mathcal{A}^{\mathrm{sign}}_\pi$, whose $K$-groups can be regarded as the $K$-groups relative to the signature pushforward of the boundary fibration. 
Given an orientation on $M$, we get the corresponding index pairing formula, Theorem \ref{main_thm_K_sign}, 
\[
\mathrm{Sign}_\Phi(M, E, [\tilde{D}^{\mathrm{sign}, E}_\pi])= [(E, [\tilde{D}^{\mathrm{sign}, E}_\pi])] \otimes_{\mathcal{A}_\pi^{\mathrm{sign}}} [M^{\mathrm{sign}}] \otimes_{\Sigma^{\mathring{M}}(G_\Phi)} \partial^{\mathring{M}}(G_\Phi) . 
\]

These indices, being defined as indices of operators on Lie groupoids, can be analyzed in terms of groupoids. 
We call {\it groupoid deformation technique} the following type of arguments. 
Suppose we are given a compact manifold $M$ and two Lie groupoids $G_0 \rightrightarrows M$ and $G_1 \rightrightarrows M$, equipped with geometric structures that determine the geometric operators $D_i \in \mathrm{Diff}^*(G_i; E_i)$. 
If one can define a Lie groupoid structure on $\mathcal{G} = G_0 \times \{0\} \sqcup G_1\times (0, 1] \rightrightarrows M \times [0, 1]$, and a geometric structure on $\mathcal{G}$ that restricts to ones on $G_0$ and $G_1$ by evaluation at $0$ and $1$ respectively, 
then the associated geometric operator $\mathcal{D}$ satisfies $\mathcal{D}|_{M \times \{i\}} = D_i$. 
Under some nice assumption on the groupoids, the element $[ev_0] \in KK(C^*(\mathcal{G}), C^*(G_0))$ is a $KK$-equivalence. 
Then, for example if the operators are elliptic, their index classes, $\mathrm{Ind}(D_i) \in K_0(C^*(G_i))$, are related as 
\[
\mathrm{Ind}(D_0)\otimes_{C^*(G_0)} [ev_0]^{-1} \otimes_{C^*(\mathcal{G})} [ev_1] = \mathrm{Ind}(D_1) \in K_0(C^*(G_1)).
\]
Furthermore, assuming we have a closed saturated subset $V \subset M$ such that $D_i|_V$ are invertible for $i = 0, 1$, we get the index classes $\mathrm{Ind}_{M \setminus V}(D_i) \in K_0(C^*(G_i|_{M \setminus V}))$. 
If we can give $\mathcal{G}$ a geometric structure such that the associated operator $\mathcal{D}$ is invertible when restricted to $V \times [0, 1]$, then we can relate these indices as 
\[
\mathrm{Ind}_{M \setminus V}(D_0)\otimes_{C^*(G_0|_{M \setminus V})} [ev_0]^{-1} \otimes_{C^*(\mathcal{G}|_{M \setminus V \times [0,1]})} [ev_1] = \mathrm{Ind}_{M \setminus V}(D_1) \in K_0(C^*(G_1|_{M \setminus V})).
\]
This argument, though very simple, turns out to be useful in proving various properties of indices considered here, without any difficult analysis involved. 

For example, in Proposition \ref{equality} (also see (\ref{equality_perturbation})), we prove that the indices defined by $\Phi$-metrics and $e$-metrics actually coincide for our settings, 
\[
\mathrm{Ind}_\Phi(P'_M, P'_Y, E, [\tilde{D}^E_\pi]) = \mathrm{Ind}_e(P'_M, P'_Y, E, [\tilde{D}^E_\pi]) . 
\]
The proof is an application of the groupoid deformation technique, by considering a groupoid of the form $G_\Phi \times \{0\} \sqcup G_e \times (0, 1]$. 
Also we prove the gluing formula (Proposition \ref{gluing}), 
as well as a condition for vanishing of the indices (Proposition \ref{vanishing}), using the such arguments.  

As an application of properties obtained in this way, in Section \ref{sec_localsignature}, we explain that the $\Phi$ (and $e$) indices of signature operators defined by the fiberwise invertible perturbations, can be used to solve the localization problem of signature for the singular fiber bundles. 
Suppose we are given a smooth map $\pi : M^{4k} \to X^{\mathrm{ev}}$ between closed oriented manifolds and $X$ is partitioned into compact manifolds with closed boundaries as $X = U \cup \cup_{i=1}^{m}V_i$. 
Suppose that each $V_i$ are disjoint, and the restriction of $\pi$ to $U$ is a fiber bundle structure with structure group contained in some nice subgroup $G \subset \mathrm{Diff}^+(F)$ of the orientation-preserving diffeomorphism group of the typical fiber $F$.  
The submanifold $U$ is regarded as the ``regular part'' of this singular fiber bundle structure $\pi$. 
The localization problem is to 
define a real number $\sigma (M_i, V_i, \pi|_{M_i}) \in \mathbb{R}$, which only depends on the data $(M_i, V_i, \pi|_{M_i})$, and write
\[
\mbox{Sign}(M) = \sum_{i = 1}^m \sigma (M_i, V_i, \pi|_{M_i}). 
\]
This problem originates from algebraic geometry for the case where the typical fiber is two dimensional, and the local signatures are constructed and calculated in various areas of mathematics, including topology, algebraic geometry and complex analysis. 
For example see \cite{Mat}, \cite{E} for topological approaches and \cite{F} for a differential geometric approaches. 
Also see \cite{AK} and the introduction of \cite{S} for more survey on this problem. 

In this situation, for each $i$ the pair $(\pi^{-1}(V_i), \pi :\pi^{-1}(\partial V_i)\to \partial V_i )$ is a compact manifold with fibered boundary, and $\pi$ has the structure group $G$. 
The idea is to fix an invertible perturbation of universal family of signature operators defined on the classifying space, and pullback the perturbation to define the $\Phi$-indices of signature operators for each $V_i$. 
We verify this idea and construct functions $\sigma$ with the desired properties in the main theorem of Section \ref{sec_localsignature}, Theorem \ref{mainthm}. 
In Section \ref{sec_example}, we give a particular example of this localization problem, where the typical fiber is the two dimensional oriented closed manifold with genus $g\geq 2$, and the group $G$ is the hyperelliptic diffeomorphism group.  
This is similar to the situation considered in \cite{E} for the case where $M$ is four dimensional, but we consider a more general situation where the dimension of $M$ can be higher. 

This paper is organized as follows. 
In Section \ref{sec_preliminaries}, we give preliminaries on representable $K$-theory, Lie groupoids and $\Phi$, $e$-calculi. 
In Section \ref{sec_index_without_perturbation} and Section \ref{sec_index_perturbation}, we define the indices of twisted geometric operators in $\Phi$ and edge metrics, and prove various properties, using the groupoid deformation technique. 
Section \ref{sec_index_without_perturbation} is about the case without the invertible perturbations, and Section \ref{sec_index_perturbation} is for the case with invertible perturbations. 
Although the results in Section \ref{sec_index_without_perturbation} are covered by those in Section \ref{sec_index_perturbation}, we separate this primitive case, because the author believes it makes it easier to understand what is going on. 
We note that the properties proved in these sections are not used in Section \ref{sec_K}, so the readers who are only interested in the index pairing need only to check the definitions of indices given in Definition \ref{def_index_perturbation_spinc} and Definition \ref{twisted_sign_perturbation_def}, and proceed to Section \ref{sec_K}. 
In Section \ref{sec_K}, we give the formulation of the indices as index pairings over the $K$-groups relative to the boundary pushforward. 
In Section \ref{sec_localsignature}, we give the application of the those indices to the localization problem of signature for singular fiber bundles, and in Section \ref{sec_example}, we apply this to the case of singular hyperelliptic fiber bundles. 
\section{Preliminaries}\label{sec_preliminaries}
\subsection{Representable $K$-theory}\label{subsec_repK}

In this subsection, we recall the definitions for representable $K$-theory in [AS]. 
We only work with complex coefficients. 
 
Let $H$ be a separable infinite dimensional Hilbert space. 
Let $\hat{H} := H \oplus H$ be the $\mathbb{Z}_2$-graded separable infinite dimensional Hilbert space. 
Let $B(H)$ and $K(H)$ denote the spaces of bounded operators and compact operators on $H$, respectively. 
For two topological spaces $X$ and $Y$, let $[X, Y]$ denote the set of homotopy classes of continuous maps from $X$ to $Y$.  

\begin{defn}\label{repKdef}
\begin{enumerate}
\item[(0)]
Let $\mathrm{Fred}^{(0)}(\hat{H})$ denote the space of self-adjoint odd bounded Fredholm operators $\tilde{A}$ on $\hat{H}$ such that $\tilde{A}^2 - I \in K(\hat{H})$, 
with the topology coming from its embedding
\[
\mathrm{Fred}^{(0)}(\hat{H}) \to B(\hat{H})_{\mathrm{c.o.}} \times K(\hat{H})_{\mathrm{norm}} \ , \ \tilde{A} \mapsto (\tilde{A}, \tilde{A}^2-1). 
\]
Here we denoted by $B(\hat{H})_{\mathrm{c.o.}}$ the space of bounded operators equipped with compact open topology, and by $K(\hat{H})_{\mathrm{norm}}$ the space of compact operators equipped with norm topology. 
\item[(1)]
Let $\mathrm{Fred}^{(1)}(H)$ denote the space of self-adjoint bounded Fredholm operators $A$ on $H$ such that $A^2-I \in K(H)$, 
with the topology coming from its embedding
\[
\mathrm{Fred}^{(1)}({H}) \to B({H})_{\mathrm{c.o.}} \times K({H})_{\mathrm{norm}} \ , \ {A} \mapsto ({A}, {A}^2-1). 
\]
\end{enumerate}
\end{defn}

\begin{fact}[{\cite[Section3]{AS}}]\label{repKfact1}

$\mathrm{Fred}^{(0)}(\hat{H})$ and $\mathrm{Fred}^{(1)}({H})$ are classifying spaces of the functors $K^0$ and $K^1$, respectively, 
i.e., we have for any space $X$, 
\[
K^0(X) = [X, \mathrm{Fred}^{(0)}(\hat{H})] \ \mbox{and} \ K^1(X) = [X, \mathrm{Fred}^{(1)}({H})]. 
\]
\end{fact}

\begin{fact}[{\cite[Proposition A2.1]{AS}}]\label{repKfact2}

The space of unitary operators on $H$ equipped with compact open topology, denoted by $U(H)_{\mathrm{c.o.}}$, is contractible. 
\end{fact}

Define the following spaces as
\begin{align}
GL^{(0)}(\hat{H}) := GL(\hat{H}) \cap \mathrm{Fred}^{(0)}(\hat{H}) &\mbox{ and } 
U^{(0)}(\hat{H}) := U(\hat{H}) \cap \mathrm{Fred}^{(0)}(\hat{H}) \label{repKeq1} \\
GL^{(1)}({H}) := GL({H}) \cap \mathrm{Fred}^{(1)}({H}) &\mbox{ and } 
U^{(1)}({H}) := U({H}) \cap \mathrm{Fred}^{(1)}({H}), \notag
\end{align}
equipped with the topology induced by the ones on $\mathrm{Fred}^{(0)}(\hat{H})$ and $\mathrm{Fred}^{(1)}({H})$. 

\begin{cor}\label{repKcor}
The spaces $GL^{(0)}(\hat{H})$, $U^{(0)}(\hat{H})$, $GL^{(1)}({H})$ and $U^{(1)}({H}) $ are contractible. 
\end{cor}
\begin{proof}
By Fact \ref{repKfact2}, the spaces $U^{(0)}(\hat{H})$ and $U^{(1)}({H}) $ are contractible. 
The map $(A, t) \to A|A|^{-t}$ for $t \in [0, 1]$ gives a retraction from $GL^{(0)}(\hat{H})$ to
$U^{(0)}(\hat{H})$ and from $GL^{(1)}({H})$ to $U^{(1)}({H}) $, respectively. 
So we get the result. 
\end{proof}

The definition of Hilbert bundles, which is suitable for our purposes, is as follows. 
\begin{defn}[Hilbert bundles]\label{def_hilbertbundle}
Let $X$ be a space. 
A {\it separable infinite dimensional Hilbert bundle} $\mathcal{H} \to X$ is a fiber bundle whose typical fibers are separable infinite dimensional Hilbert space $H$, with structure group $U(H)_{\mathrm{c.o.}}$. 

A {\it $\mathbb{Z}_2$-graded separable infinite dimensional Hilbert bundle} $\hat{\mathcal{H}} \to X$ is a fiber bundle whose typical fibers are $\mathbb{Z}_2$-graded separable infinite dimensional Hilbert space $H$, with structure group $U(\hat{H})_{\mathrm{c.o.}}$. 
\end{defn}

By \cite[Proposition 3.1]{AS}, the action of $U(\hat{H})_{\mathrm{c.o.}}$ on $\mathrm{Fred}^{(0)}(\hat{H})$ is continuous. 
Thus given a $\mathbb{Z}_2$-graded Hilbert bundle $\hat{\mathcal{H}} \to X$, we also get the associated $\mathrm{Fred}^{(0)}(\hat{H})$-bundle $\mathrm{Fred}^{(0)}(\hat{\mathcal{H}})\to X$. 
The analogous construction applies to the ungraded case. 

By Fact \ref{repKfact2}, we have the following. 
\begin{cor}\label{hilbert_bundle_trivial}
Any separable infinite dimensional Hilbert bundle is trivial, and any choices of trivialization are homotopic. 
\end{cor}

\subsection{$\mathbb{C}l_1$-invertible perturbations}\label{subsec_invertibleperturbation}

In this subsection, we discuss $\mathbb{C}l_1$-invertible perturbations for a family of $\mathbb{Z}_2$-graded Fredholm operators parametrized by a possibly noncompact space. 
The symbol $K^i$ denotes the representable $K$-theory. 
The setting is as follows. 
\begin{itemize}
\item Let $X$ be a topological space.  
\item Let $\hat{\mathcal{H}} = \{\hat{\mathcal{H}}_x\}_{x \in X} \to X$ be a $\mathbb{Z}_2$-graded separable Hilbert bundle (see Definition \ref{def_hilbertbundle}). 
\item Let $\gamma$ be the involution on $\hat{\mathcal{H}}$ defining the $\mathbb{Z}_2$-grading. 
\item Let $\mathrm{Fred}^{(0)}(\hat{\mathcal{H}}) = \{\mathrm{Fred}^{(0)}(\hat{\mathcal{H}}_x)\}_{x \in X} \to X$ be the $\mathrm{Fred}^{(0)}(\hat{H})$-fiber bundle associated to $\hat{\mathcal{H}}$. 
\item Assume we are given an element $F \in \Gamma(X; \mathrm{Fred}^{(0)}(\hat{\mathcal{H}}) )$. 
\end{itemize}

Let $pr : X \times [0, 1] \to X$ be the canonical projection, and consider the Hilbert bundle $pr^*\hat{\mathcal{H}} \to X \times [0, 1]$. 
For simplicity, we also denote this bundle by $\hat{\mathcal{H}} \to X \times [0, 1]$. 
A $\mathbb{C}l_1$-invertible perturbation for $F$ is defined to be a homotopy from $F$ to an invertible family, as follows. 

\begin{defn}[$\mathbb{C}l_1$-invertible perturbations]
Let $(X, \hat{\mathcal{H}}, F)$ as above. 
An operator $\tilde{\mathbb{F}} : \Gamma(X \times [0, 1]; \hat{\mathcal{H}}) \to \Gamma(X \times [0,1]; \hat{\mathcal{H}})$ is called a $\mathbb{C}l_1$-invertible perturbation for $F$ if 
\begin{itemize}
\item $\tilde{\mathbb{F}}\in \Gamma(X\times[0,1]; \mathrm{Fred}^{(0)}(\hat{\mathcal{H}}) )$
\item $\tilde{\mathbb{F}}|_{X \times \{0\}} = F$. 
\item $\tilde{\mathbb{F}}|_{X \times \{1\}} $ is a family of invertible operators.  
\end{itemize}
Let us denote the set of $\mathbb{C}l_1$-invertible perturbations for $F$ by $\tilde{\mathcal{I}}(F)$. 
\end{defn}

We introduce a natural homotopy equivalence relation on $\tilde{\mathcal{I}}(F)$, 

\begin{defn}
Let $\tilde{\mathbb{F}}$ and $\tilde{\mathbb{F}}'$ be two elements in $\tilde{\mathcal{I}}(F)$. 
We say $\tilde{\mathbb{F}}$ and $\tilde{\mathbb{F}}'$ are homotopic if 
there exists an operator $\tilde{\mathbb{F}}'' : \Gamma(X \times [0, 1] \times [0, 1]; \hat{\mathcal{H}}) \to \Gamma(X\times[0,1] \times [0, 1]; \hat{\mathcal{H}})$ such that
\begin{itemize}
\item $\tilde{\mathbb{F}}'' \in \Gamma(X\times[0,1]\times[0,1]; \mathrm{Fred}^{(0)}(\hat{\mathcal{H}}))$. 
\item $\tilde{\mathbb{F}}''|_{X \times [0, 1] \times \{0\}} = \tilde{\mathbb{F}}$ and $\tilde{\mathbb{F}}''|_{X \times [0,1] \times \{1\}} = \tilde{F}'$. 
\item $\tilde{\mathbb{F}}''|_{X \times [0,1] \times \{u\}} \in \tilde{\mathcal{I}}(F)$ for all $u \in [0,1]$. 
\end{itemize}
Let us denote $\mathcal{I}(F)$ the set of homotopy classes of elements in $\tilde{\mathcal{I}}(F)$. 
\end{defn}

The following lemma follows directly from Fact \ref{repKfact1}. 
 
\begin{lem}\label{perturbationexistence}
The element $F$ admits a $\mathbb{C}l_1$-invertible perturbation if and only if $[F] = 0 \in K^0(X)$. 
\end{lem}

\begin{lem}\label{htpyclass_spectralsection}
Suppose $F$ satisfies $[F] = 0 \in K^0(X)$. 
Then $\mathcal{I}(F)$ has a natural structure of affine space over $K^{-1}(X) (:= [X, \Omega\mathrm{Fred}^{(0)}(\hat{H})])$. 
\end{lem}

\begin{proof}
Assume we are given two elements in $\mathcal{I}(F)$. 
By Corollary \ref{hilbert_bundle_trivial}, we choose a trivialization of the Hilbert bundle $\hat{\mathcal{H}} \simeq \hat{H}\times X$, which is unique up to homotopy.  
Take any representative of these elements and denote them by $\tilde{\mathbb{F}}^0, \tilde{\mathbb{F}}^1 \in \tilde{\mathcal{I}}(F)$, respectively. 
We explain the definition of the difference class $[\tilde{\mathbb{F}}^1 - \tilde{\mathbb{F}}^0] \in K^{-1}(X)$. 

Define the continuous map $\mathcal{F}$ as follows. 
\begin{align*}
\mathcal{F} &: X \times [0, 1] \to \mathrm{Fred}^{(0)}(\hat{H}) \\
\mathcal{F}|_{X \times [0, 1]} &= \begin{cases}
\tilde{\mathbb{F}}^1|_{X \times \{1-2t\}} & \mbox{for } t \in [0, 1/2]\\
\tilde{\mathbb{F}}^0|_{X \times \{2t-1\}} & \mbox{for } t \in [1/2, 1]. 
\end{cases}
\end{align*}
The image of $\mathcal{F}|_{X \times \{0, 1\}}$ is contained in $GL^{(0)}(\hat{H})$. 
Since $GL^{(0)}(\hat{H})$ is contractible by Corollary \ref{repKcor}, the map $\mathcal{F}$ gives the desired element 
\[
[\tilde{\mathbb{F}}^1 - \tilde{\mathbb{F}}^0] := [\mathcal{F}] \in [X, \Omega\mathrm{Fred}^{(0)}(\hat{H})] = K^{-1}(X). 
\]
The well-definedness is obvious. 

Conversely, if we are given an element $[\tilde{\mathbb{F}}^0] \in \mathcal{I}(F)$ and an element $[\mathcal{F}] \in K^{-1}(X)$, it is easy to construct the unique element $[\tilde{\mathbb{F}}^1] \in \mathcal{I}(F)$ such that $[\tilde{\mathbb{F}}^1 - \tilde{\mathbb{F}}^0] = [\mathcal{F}] $. 
Also it is easy to see that this defines an affine structure of $\mathcal{I}(F)$ over $K^{-1}(X)$. 
\end{proof}

Let us turn to the case where the parameter space $X$ is a smooth compact manifold (possibly with boundaries or corners), the Hilbert bundle $\hat{\mathcal{H}}$ and the family of operators $D$ come from a fiber bundle over $X$, and the family $D$ is unbounded. 
More precisely, we consider the following situations. 
\begin{itemize}
\item Let $\pi : M \to X$ be a smooth fiber bundle with closed fibers, equipped with a smooth fiberwise riemannian metric $g_\pi$. 
\item Let $E \to M$ be a smooth hermitian $\mathbb{Z}_2$-graded vector bundle. 
\item Let $D= \{D_x\}_{x \in X}$, $D_x : C^\infty(\pi^{-1}(x) ; E|_{\pi^{-1}(x)}) \to C^\infty(\pi^{-1}(x) ; E|_{\pi^{-1}(x)})$ be a smooth family of odd formally self-adjoint elliptic operators of positive order. 
\item Let us denote $\hat{\mathcal{H}} = \{\hat{\mathcal{H}}_x = L^2(\pi^{-1}(x) ; E|_{\pi^{-1}(x)})\}_{x \in X}$ with the natural Hilbert bundle structure over $X$. 
The operator $D$ also denotes the closed extension to $D : \Gamma(X; \hat{\mathcal{H}}) \to\Gamma(X; \hat{\mathcal{H}})$. 
\end{itemize}
For such a family $D$, the bounded transform $\psi(D) :=D / \sqrt{1 + D^2}$ is a smooth family pseudodifferential operators of order $0$, and defines an element $[\psi(D)] \in K^0(X)$. 
We call this class the {\it family index class} of $D$, and abuse the notation to write $[D]:=[\psi(D)] \in K^0(X)$. 

\begin{defn}[$\mathcal{I}_{\mathrm{sm}}(D)$]\label{def_smooth_perturbation}
In the above situations, an operator $\tilde{D} : C^\infty(M; E) \to C^\infty(M; E)$ is  called a $\mathbb{C}l_1$-smooth invertible perturbation of $D$ if 
\begin{itemize}
\item $\tilde{D}=\{\tilde{D}_x\}_{x \in X}$ is a smooth family of invertible odd formally self-adjoint operators. 
\item $\tilde{D} - D = A = \{A_x\}_{x \in X}$, where $A_x$ is a pseudodifferential operator of order $0$ for each $x \in X$. 
\end{itemize}
Let us denote the set of $\mathbb{C}l_1$-smooth invertible perturbations for $D$ by $\tilde{\mathcal{I}}_{\mathrm{sm}}(D)$.  
We can introduce the obvious homotopy equivalence relations in $\tilde{\mathcal{I}}_{\mathrm{sm}}(D)$. 
We denote $\mathcal{I}_{\mathrm{sm}}(D)$ the set of homotopy classes of elements in
 $\tilde{\mathcal{I}}_{\mathrm{sm}}(D)$. 
\end{defn}

There is a canonical map 
\begin{align}\label{perturbation_eq}
\tilde{\mathcal{I}}_{\mathrm{sm}}(D) & \to \tilde{\mathcal{I}}(\psi(D)) \\
( \tilde{D} = D + A) &\mapsto \left( t \in [0, 1] \mapsto \psi(D + tA)\right). 
\end{align}
In fact this map induces an isomorphism $\mathcal{I}_{\mathrm{sm}}(D) \simeq \mathcal{I}(\psi(D))$, by the following fact. 

\begin{fact}[\cite{MP}]\label{affinestr}
The family $D$ admits a $\mathbb{C}l_1$-smooth invertible perturbation if and only if $[D] = 0 \in K^0(X)$. 

If $[D] = 0 \in K^0(X)$, then $\mathcal{I}_{\mathrm{sm}}(D)$ has a natural structure of an affine space over $K^{-1}(X)$, described as follows. 

Let $Q_i \in \mathcal{I}_{\mathrm{sm}}(D)$, $i = 0, 1$.  
Choose a representative $\tilde{D}_{i}$ for $Q_i$. 
Consider the family $D_{[0, 1]}$ of operators parametrized by $X \times [0, 1]$, defined as 
\[
D_{[0,1]}|_{X \times \{u\}} := u\tilde{D}_0 + (1-u)\tilde{D}_1. 
\]
Since the family $D_{[0, 1]}$ is invertible on $X \times \{0, 1\}$, it defines a family index class in $[D_{[0,1]}] \in K^0(X \times [0, 1]; X \times \{0, 1\}) \simeq K^0(X \times (0, 1)) \simeq K^{-1}(X)$. We have
\[
[Q_1 - Q_0] = [D_{[0, 1]}] \in K^{-1}(X). 
\]
Every element in $K^{-1}(X)$ can be written as the index of some operator of the form $D_{[0, 1]}$ above. 
\end{fact}

\begin{rem}
Actually, in \cite{MP} they define $\mathbb{C}l_1$-invertible perturbations as perturbations by fiberwise smoothing operators. 
Our class of smooth $\mathbb{C}l_1$-invertible perturbations in Definition \ref{def_smooth_perturbation} is larger because we allow the perturbations to be zeroth order operators. 
But divided by the homotopy equivalences, they are canonically isomorphic.  
\end{rem}

Since the above affine structure corresponds to the affine structure on $\mathcal{I}(\psi(D))$ under the canonical map (\ref{perturbation_eq}), we have the following corollary. 
\begin{cor}
In the above situations, we have a canonical isomorphism
\[
\mathcal{I}_{\mathrm{sm}}(D) \simeq \mathcal{I}(\psi(D))
\]
between affine spaces over $K^{-1}(X)$, induced by the map (\ref{perturbation_eq}). 
\end{cor}

With an abuse of notation we write $\mathcal{I}(D):= \mathcal{I}(\psi(D))$ for a positive order elliptic family $D$. 

\subsection{Groupoids}

\subsubsection{Basic definitions}
We recall basic definitions on groupoids and pseudodifferential calculus on them. 
The material is taken from \cite{DL}. 
\begin{defn}[Groupoids]
Let $G$ and $G^{(0)}$ be two sets. 
A groupoid structure on $G$ over $G^{(0)}$ is given by the following maps. 
\begin{itemize}
\item An injective map $u : G^{(0)} \to G$, called the unit map. 
We often identify $G^{(0)}$ with its image $u(G^{(0)}) \subset G$. 
$G^{(0)}$ is called the space of units. 
\item Two surjective maps $r, s : G \to G^{(0)}$, satisfying $r\circ u = s \circ u = id_{G^{(0)}}$. These are called range and source map, respectively.  
\item An involution $i : G \to G , \gamma \mapsto \gamma^{-1}$, called the inverse map. 
It satisfies $s \circ i = r$. 
\item A map $m : G^{(2)} \to G, (\gamma_1, \gamma_2) \mapsto \gamma_1 \cdot \gamma_2$,  called product, where $G^{(2)} = \{(\gamma_1, \gamma_2) \in G \times G| s(\gamma_1) = r(\gamma_2)\}$. 
Moreover for $(\gamma_1, \gamma_2) \in G^{(2)}$, we have $r(\gamma_1 \cdot \gamma_2) = r(\gamma_1)$ and $s(\gamma_1 \cdot \gamma_2) = s(\gamma_2)$. 
\end{itemize}
The following properties must be satisfied: 
\begin{itemize}
\item The product is associative: for any $\gamma_1$, $\gamma_2$, $\gamma_3$ in $G$ such that $s(\gamma_1) = r(\gamma_2)$ and $s(\gamma_2) = r(\gamma_3)$, the following equality holds. 
\[
(\gamma_1 \cdot \gamma_2) \cdot \gamma_3 = \gamma_1 \cdot (\gamma_2 \cdot \gamma_3). 
\]
\item For any $\gamma$ in $G$, we have $r(\gamma) \cdot \gamma = \gamma \cdot s(\gamma)= \gamma$ and $\gamma \cdot \gamma^{-1} = r(\gamma)$.  
\end{itemize}
A groupoid structure on $G$ over $G^{(0)}$ is usually denoted by $G \rightrightarrows G^{(0)}$, where the arrows stand for the source and range maps. 
\end{defn}
For $A, B \subset G^{(0)}$, we use the following notations. 
\[
G_A := s^{-1}(A), \ G^B := r^{-1}(B), \ G_A^B := G_A \cap G^B \ \mathrm{and} \ G|_A := G_A^A. 
 \]
We say a subset $A \subset G^{(0)}$ is {\it saturated} if it satisfies $G_A = G^A = G|_A$. 

Suppose that $G \rightrightarrows G^{(0)}$ is a locally compact groupoid and $\phi : X \to G^{(0)}$ is an open surjective map, where $X$ is a locally compact space. 
The {\it pull back groupoid} is the groupoid
\[
^*\!\phi^*(G) \rightrightarrows X, 
\]
where 
\[
^*\!\phi^*(G) = \{(x, \gamma, y) \in X \times G \times X \ | \ \phi (x) = r(\gamma) \mbox{ and } \phi(y) = s(\gamma)\}
\]
with $s(x, \gamma, y) = y$, $r(x, \gamma, y) = x$, $(x, \gamma_1, y) \cdot (y, \gamma_2, z) = (x, \gamma_1\cdot \gamma_2, z)$ and $(x, \gamma, y)^{-1} = (y, \gamma^{-1}, x)$. 
This endows $^*\!\phi^*(G)$ with a structure of locally compact groupoid. 
Moreover the groupoids $G$ and $^*\!\phi^*(G)$ are Morita equivalent (see \cite[Section 1.2]{DL}). 

\begin{defn}[Lie groupoids]\label{def_lie_groupoid}
We call $G \rightrightarrows G^{(0)}$ a Lie groupoid when $G$ and $G^{(0)}$ are second-countable smooth manifolds with $G^{(0)}$ Hausdorff, and all the structural homomorphisms are smooth and $s$ is a submersion (for definitions of submersions between manifolds with corners, we refer to \cite[Definition 1]{LN}). 
\end{defn}
Note that by requiring $s$ to be a submersion, for each $x \in G^{(0)}$, the $s$-fiber $G_x$ is a smooth manifold without boundary or corners. 

For a Lie groupoid $G$, let us denote $\Omega^{\frac{1}{2}}(\mathrm{ker}(ds) \oplus \mathrm{ker}(dr)) \to G$ the half density bundle of the vector bundle $\mathrm{ker}(ds) \oplus \mathrm{ker}(dr) \to G$. 
We also denote this vector bundle by $\Omega^{\frac{1}{2}} \to G$. 
Then $C^\infty_c(G; \Omega^{\frac{1}{2}})$ has a structure of a $*$-algebra with 
\begin{itemize}
\item The involution given by $f^*(\gamma)= \overline{f(\gamma^{-1})}$. 
\item The convolution product given by $f*g(\gamma)=\int_{G_{s(\gamma)}} f(\gamma \eta^{-1})g(\eta)$. 
\end{itemize}
For all $x \in G^{(0)}$ there is a $*$-homomorphism $\lambda_x : C^\infty_c(G; \Omega^{\frac{1}{2}}) \to B(L^2(G_x; \Omega^{\frac{1}{2}}(G_x)))$ defined by 
\[
\lambda_x(f) \xi (\gamma) = \int_{G_x} f(\gamma\eta^{-1})\xi(\eta). 
\]

\begin{defn}[Reduced groupoid $C^*$-algebras]
Let $G$ be a Lie groupoid. 
The reduced $C^*$-algebra of $G$, denoted by $C^*(G)$, is the completion of $C^\infty_c(G; \Omega^{\frac{1}{2}})$ with respect to the norm
\[
\|f\|_r = \sup_{x \in G^{(0)}} \|\lambda_x(f)\|_x, 
\]
where $\|\|_x$ is the operator norm on $B(L^2(G_x; \Omega^{\frac{1}{2}}(G_x)))$. 
\end{defn}

\begin{rem}
In general, there are many possible $C^*$-completion of $C^\infty_c(G; \Omega^{\frac{1}{2}})$ which are not necessarily isomorphic to $C^*(G)$. 
For example the full $C^*$-algebra of $G$ is the completion of $C^\infty_c(G; \Omega^{\frac{1}{2}})$ with respect to all continuous representations. 
All the groupoids we actually use in this paper are amenable, so the full and reduced $C^*$-algebras coincide. 
We use reduced $C^*$-algebras in this paper, in order to make the argument in subsection \ref{subsubsec_ellipticity} work. 
\end{rem}

\begin{defn}[Lie algebroids]
A Lie algebroid $\mathfrak{A}=(p : \mathfrak{A} \to M , [{\cdot},{\cdot}]_{\mathfrak{A}})$ on a smooth manifold $M$ is a vector bundle equipped with a Lie bracket 
$[{\cdot},{\cdot}]_{\mathfrak{A}} : C^\infty(M; \mathfrak{A}) \times  C^\infty(M; \mathfrak{A}) \to  C^\infty(M; \mathfrak{A})$ together with a homomorphism of fiber bundle $p : \mathfrak{A} \to TM$ called the anchor map, satisfying the following. 
\begin{itemize}
\item The bracket $[{\cdot},{\cdot}]_{\mathfrak{A}}$ is $\mathbb{R}$-bilinear, antisymmetric and satisfies the Jacobi identity. 
\item $[X, fY]_{\mathfrak{A}} = f[X ,Y]_{\mathfrak{A}} + p(X)(f)Y$ for all $X, Y \in C^\infty(M; \mathfrak{A}) $ and $f \in C^\infty(M)$. 
\item $p([X, Y]_{\mathfrak{A}}) = [p(X) , p(Y)]$ for all $X, Y \in C^\infty(M; \mathfrak{A})$. 
\end{itemize}
\end{defn}

Given a Lie groupoid $G$, we associate a Lie algebroid as follows. 
The vector bundle is given by $\ker(ds)|_{G^{(0)}} = \cup_{x \in G^{(0)}}TG_x \to G^{(0)}$. 
This has the structure of a Lie algebroid over $G^{(0)}$ with the anchor map $dr$. 
We denote this Lie algebroid by $\mathfrak{A}G$ and call it the Lie algebroid of $G$. 

For a Lie groupoid $G$, a submanifold $V \subset G^{(0)}$ is said to be {\it transverse} to $G$ if for each $x \in V$, the composition $p_x \circ \sharp_x : \mathfrak{A}_xG \to (N_V^M)_x = T_xM / T_x V$ is surjective.  

\begin{defn}[$G$-operators]
Let $G$ be a Lie groupoid. 
Let $E, F \to G^{(0)}$ be two vector bundles. 
A linear $G$-operator $D$ is a continuous linear operator
\[
D : C^\infty_c(G; r^{*}E \otimes \Omega^{\frac{1}{2}}) \to C^\infty(G; r^{*}F \otimes \Omega^{\frac{1}{2}})
\]
satisfying the following. 
\begin{itemize}
\item The operator $D$ restricts to a continuous family $\{D_x\}_{x \in G^{(0)}}$ of linear operators 
$D_x : C^\infty_c(G_x; r^{*}E \otimes \Omega^{\frac{1}{2}}) \to C^\infty(G_x; r^{*}F \otimes \Omega^{\frac{1}{2}})$ such that
\[
Df(\gamma) = D_{s(\gamma)}f_{s(\gamma)}(\gamma) \ \forall f \in  C^\infty_c(G; r^{*}E \otimes \Omega^{\frac{1}{2}}) . 
\]
\item The following equivariance property holds: 
\[
U_\gamma D_{s(\gamma)}= D_{r(\gamma)}U_\gamma, 
\]
where $U_\gamma$ is the map induced by the right multiplication by $\gamma$. 
\end{itemize}
A linear $G$-operator $D$ is called pseudodifferential of order $m$ if it satisfies the following. 
\begin{itemize}
\item Its Schwartz kernel $k_D$ is a distribution on $G$ that is smooth outside $G^{(0)}$. 
\item For every distinguished chart $\psi : U \subset G \to \Omega \times s(U) \subset \mathbb{R}^{n-p} \times \mathbb{R}^p$ of $G$, 
\[
\xymatrix{
	U \ar[rd]_s  \ar[rr]^\psi&& \Omega \times s(U) \ar[ld]_{p_2} \\
	& s(U) &
}
\]
the operator $(\psi^{-1})^*D\psi^* : C^\infty_c(\Omega \times s(U); (r\circ \psi^{-1})^*E) \to C^\infty_c(\Omega \times s(U); (r\circ \psi^{-1})^*F)$ is a smooth family parametrized by $s(U)$ of pseudodifferential operators of order $m$ on $\Omega$. 
\end{itemize}
We say that $D$ is smoothing if $k_D$ is smooth and that $D$ is compactly supported if $k_D$ is compactly supported. 
We denote the space of compactly supported order $m$ $G$-pseudodifferential operators from $E$ to $F$ by $\Psi^m_c(G; E, F)$. 
We also denote $\Psi^m_c(G; E) = \Psi^m_c(G; E, E)$ and when $E$ is the trivial bundle we denote $\Psi^m_c(G)= \Psi^m_c(G; E)$. 
\end{defn}
One can show that the space $\Psi^*_c(G; E)$ of compactly supported pseudodifferential $G$-operators on $E$ is an involutive algebra. 

Let us denote the cosphere bundle of $\mathfrak{A}G \to G^{(0)}$ as $\mathfrak{S}^*(G) \to G^{(0)}$. 
Given a $G$-pseudodifferential operator $D$, we can associate its {\it principal symbol} $\sigma(D) \in C^\infty_c(\mathfrak{S}^*(G); \mathrm{Hom}(E; F))$ as follows. 
Recall that $D$ is given by a family $\{D_x\}_{x\in G^{(0)}}$ of pseudodifferential operators on $G_x$. 
We define
\[
\sigma(D, \xi) := \sigma_{pr}(D_x)(x, \xi), 
\]
where $\sigma_{pr}(D_x)$ denotes the principal symbol of the pseudodifferential operator $D_x$. 

Now we give important examples of Lie groupoids which are building blocks of groupoids appearing in this paper. 
For more examples including the ones below, see \cite[Example 6.2 and Example 6.4]{DL}. 
\begin{ex}[Vector bundle groupoids]
If we are given a smooth vector bundle $\pi : E \to X$, we get a Lie groupoid $E \rightrightarrows X$ by setting $s = r = \pi$ and multiplication induced from the addition on $E_x$ for each $x$. 
Choosing any smooth family of fiberwise riemannian metric on $E$, the $C^*$-algebra $C^*(E)$ is the fiberwise convolution algebra of $E$, and we have $C^*(E) \simeq C_0(E^*)$ by the fiberwise Fourier transform. 
An $E$-pseudodifferential operator $D_E$ is equivalent to a family of pseudodifferential operators $\{D_x\}_{x \in X}$ parametrized by $X$, and each $D_x$ is an operator on the space $E_x$ which is translation invariant. 
\end{ex}
\begin{ex}[Groupoids associated to fiber bundles]
If we are given a smooth fiber bundle $\pi : M \to X$, we get a Lie groupoid $M \times_\pi M = \{(m, n) \in M \times M \ | \ \pi(m) = \pi(n)\} \rightrightarrows M$. 
Here $s(m, n) = n$, $r(m ,n) = m$ and $(m, n)\cdot (n, l) = (m, l)$. 
Choosing any smooth family of fiberwise riemannian metric for $\pi$, the $C^*$-algebra $C^*(M \times_\pi M)$ is isomorphic to $\mathcal{K}(L^2_X(M))$, where $L^2_X(M)$ is the Hilbert $C_0(X)$-module given by the completion of $C_c^\infty(M)$ by the canonical $C_0(X)$-valued inner product, and the symbol $\mathcal{K}$ denotes the $C^*$-algebra of compact operators in the sense of a Hilbert module. 
We have the canonical Morita equivalence (for the notion of Morita equivalence, see \cite[Section 1.2]{DL}) between $C^*(M \times_\pi M)$ and $C_0(X)$. 
An $M\times_\pi M$-pseudodifferential operator $D_\pi$ is equivalent to a family of pseudodifferential operators $\{D_x\}_{x \in X}$ parametrized by $X$, and each $D_x$ is an operator on $\pi^{-1}(x)$. 
\end{ex}

\subsubsection{Geometric operators}
Here we define geometric operators, such as spin Dirac operators and signature operators, on a given Lie groupoid $G$. 
For detailed discussion and other examples, we refer to \cite{LN}. 

In this paper, we often deal with $\mathbb{Z}_2$-graded vector bundles and algebras. 
If we are given two $\mathbb{Z}_2$-graded vector bundles $V$ and $W$, or algebras $A$ and $B$, we always consider their graded tensor product $V \hat{\otimes} W$ and $A \hat{\otimes} B$, following the conventions in \cite[Section 1.1]{LM}. 

In this subsubsection, for an Euclidean space $E$, we denote $\mathrm{Cliff}(E)$ by the $*$-algebra over $\mathbb{R}$, generated by the elements of $E$ and relations
\[
e = -e^*, \mbox{ and } e^2 = -||e||^2\cdot 1 \mbox{ for all } e \in E. 
\]
This construction applies to Euclidean vector bundles as well. 

First we define spin Dirac operators. 
In order to do this, we first define our convention on spin and $spin^c$ structures on vector bundles. 
Denote $\widetilde{GL^+_k(\mathbb{R})} \to GL^+_k(\mathbb{R})$ the unique non-trivial covering of $GL^+_k(\mathbb{R})$ for $k \geq 2$. 
For $k=1$, denote $\widetilde{GL^+_k(\mathbb{R})} := GL^+_k(\mathbb{R}) \times \mathbb{Z}_2 \to GL^+_k(\mathbb{R})$ the projection to the first factor.  
\begin{defn}{(Spin/Pre-spin structures on vector bundles)}\label{convention_spin}

Let $E \to X$ be a real vector bundle on a space $X$ with rank $k$. 
\begin{itemize}
\item 
A pre-spin structure on $E$ consists of the following data $(o, P')$. 
\begin{itemize}
\item An orientation $o$ on the vector bundle $E \to X$. 
\item A principal $\widetilde{GL^+_k(\mathbb{R})}$-bundle $P' \to X$ equipped with a bundle map $P' \to P_{GL^+}(E)$ which is equivariant with respect to the canonical homomorphism $\widetilde{GL^+_k(\mathbb{R})} \to GL^+_k(\mathbb{R})$. 
Here we denoted $P_{GL^+}(E)$ the oriented frame bundle of $E$ defined by $o$. 
\end{itemize}
\item
A spin structure on $E$ consists of the following data $(o, g, P)$. 
\begin{itemize}
\item An orientation $o$ and a riemannian metric $g$ on the vector bundle $E\to X$. 
\item A principal $Spin_k$-bundle $P \to X$ equipped with a bundle map $P \to P_{SO}(E)$ which is equivariant with respect to the canonical homomorphism $Spin_k \to SO_k$. 
\end{itemize}
\end{itemize}
\end{defn}

Note that a pre-spin structure on $E$ together with any riemannian metric $g$ on $E$ defines a spin structure on $E$ uniquely. 
A pre-spin structure on $E$ can also be regarded as a homotopy class of spin structures on $E$. 
See \cite[pp.131--132]{LM}. 

\begin{defn}[$Spin^c$/Pre-$spin^c$ structures on vector bundles]\label{convention_spinc}
Let $E \to X$ be a real vector bundle on a space $X$. 
\begin{itemize}
\item 
A pre-$spin^c$ structure on $E$ consists of the following data $(o, P')$. 
\begin{itemize}
\item An orientation $o$ on the vector bundle $E \to X$. 
\item A principal $\widetilde{GL^+_k(\mathbb{R})} \times_{\mathbb{Z}_2}  \mathbb{C}^*$-bundle $P' \to X$ equipped with a bundle map $P' \to P_{GL^+}(E)$ which is equivariant with respect to the canonical homomorphism $\widetilde{GL^+_k(\mathbb{R})} \times_{\mathbb{Z}_2} \mathbb{C}^* \to GL^+_k(\mathbb{R})$. 
\end{itemize}
\item
A $spin^c$ structure on $E$ consists of the following data $(o, g, P)$. 
\begin{itemize}
\item An orientation $o$ and a riemannian metric $g$ on the vector bundle $E\to X$. 
\item A principal $Spin^c_k$-bundle $P \to X$ equipped with a bundle map $P \to P_{SO}(E)$ which is equivariant with respect to the canonical homomorphism $Spin^c_k \to SO_k$. 
\end{itemize}
\end{itemize}
If $E$ has a $spin^c$-structure, by the group homomorphism
\[
p : Spin^c_k = Spin_k \times_{\mathbb{Z}_2} U(1) \to U(1), [g, z] \mapsto z^2 
\]
we get a hermitian line bundle 
\[
L := P \times_p \mathbb{C} \to X. 
\]
We call $L$ the {\it determinant line bundle} associated to the $spin^c$-structure. 
\begin{itemize}
\item
A differential $spin^c$ structure on $E$ consists of the following data $(o, g, P, \nabla^L)$. 
\begin{itemize}
\item A $spin^c$ structure $(o, g, P)$ on $E$. 
\item A unitary connection $\nabla^L$ on the determinant line bundle $L$. 
\end{itemize}
\end{itemize}
\end{defn}

\begin{defn}
A spin (pre-spin, $spin^c$, pre-$spin^c$, differential $spin^c$) structure on a Lie groupoid $G$ is a spin (pre-spin,$spin^c$, pre-$spin^c$, differential $spin^c$ ) structure on its Lie algebroid $\mathfrak{A}G \to G^{(0)}$ (regarded as a vector bundle). 
\end{defn}

Suppose we are given a metric $g$ on $\mathfrak{A}G$. 
For each $x \in G^{(0)}$, since we have $TG_x \simeq (r^*\mathfrak{A}G)|_{G_x}$ canonically, $g$ induces a riemannian metric on $G_x$. 
Levi-Civita connection on each $G_x$, denoted by $\nabla^x : C^\infty(G_x; TG_x) \to C^\infty(G_x; TG_x \otimes T^*G_x)$, combines to give a linear map
\begin{equation}\label{levicivita}
\nabla^{LC} : C^\infty(G; r^*\mathfrak{A}G) \to C^\infty(G;r^*\mathfrak{A}G\otimes r^*\mathfrak{A}^*G ). 
\end{equation}
For each $X \in C^\infty(G^{(0)}; \mathfrak{A}G)$, $r^*X \in C^\infty(G; r^*\mathfrak{A}G)$ gives a first order differential operator
\[
\nabla^{LC}_{r^*X} : C^\infty(G; r^*\mathfrak{A}G) \to C^\infty(G; r^*\mathfrak{A}G), 
\]
and it is right invariant, i.e., $\nabla_{r^*X} \in \mathrm{Diff}^1(G; \mathfrak{A}G)$. 

Suppose that we are given a spin structure on $G$. 
Let $S \to G^{(0)}$ be the associated complex spinor bundle. 
The Levi-Civita connection on $r^*\mathfrak{A}G$ lifts uniquely to a connection $\nabla^S : C^\infty(G; r^*S) \to C^\infty(G;r^*S\otimes r^*\mathfrak{A}^*G )$ and it has a right invariance property as above. 
Let us denote $c : \mathrm{Cliff}(\mathfrak{A}G) \to \mathrm{End}(S)$ the Clifford action on the spinor bundle. 

\begin{defn}[Spin Dirac operators on Lie groupoids]
Let $G$ be a Lie groupoid equipped with a spin structure. 
Let $\{e_\alpha\}_{\alpha}$ be a local orthonormal frame of $\mathfrak{A}G \to G^{(0)}$. 
The differential operator $D^S$ on $C^\infty(G; r^*S)$, locally defined as
\[
D^S := \sum_\alpha c(e_\alpha)\nabla^S_{r^*e_\alpha}, 
\]
gives an element $D^S \in \mathrm{Diff}^1(G; S)$. 
We call it the spin Dirac operator on $G$. 
If the rank of $\mathfrak{A}G$ is even, the spinor bundle is naturally $\mathbb{Z}_2$-graded and the Dirac operator is odd with respect to this grading. 

Equivalently, the definition of $D^S$ can also be described as follows. 
Given a spin structure on $G$, for each $x \in G^{(0)}$ the spin structure on $G_x$ is associated. 
If we denote $D^S_x$ the spin Dirac operator for each $G_x$, the family $D^S = \{D^S_x\}_{x \in G^{(0)}}$ forms a right invariant family, and coincides with the definition given above. 
\end{defn}

This construction generalizes to Clifford modules and Dirac operators on a Lie groupoid $G$, defined as follows. 
\begin{defn}[Clifford modules, connections and Dirac operators on Lie groupoids]
Let $G$ be a Lie groupoid equipped with an orientation and a metric on $\mathfrak{A}G$. 
Let $\mathrm{Cliff}(\mathfrak{A}G) \to G^{(0)}$ denote the Clifford bundle of $\mathfrak{A}G$. 
Let $W \to G^{(0)}$ be a $\mathrm{Cliff}(\mathfrak{A}G)$-module bundle. 
Let us denote $c : \mathrm{Cliff}(\mathfrak{A}G) \to \mathrm{End}(W)$ the Clifford action. 
\begin{itemize}
\item
A continuous linear map $\nabla^W : C^\infty(G; r^*W) \to C^\infty(G; r^*W \otimes r^*\mathfrak{A}^*G)$ is called an admissible connection if
\[
\nabla^W_X (c(Y) \xi) = c(\nabla^{LC}_XY)\xi + c(Y) \nabla_X^W(\xi),
\]
for all $\xi \in C^\infty(G; r^*W)$ and $X,Y \in C^\infty(G; r^*\mathfrak{A}G)$.
\item
A right invariant admissible connection $\nabla^W$ is called a Clifford connection on $W$. 
\item 
For a $\mathrm{Cliff}(\mathfrak{A}G)$-module bundle $W$ equipped with a Clifford connection $\nabla^W$, 
the Dirac operator $D^W \in \mathrm{Diff}^1(G; W)$ is defined by
\[
D^W := \sum_\alpha c(e_\alpha)\nabla^W_{r^*e_\alpha}, 
\]
using a local orthonormal frame $\{e_\alpha\}_\alpha$ for $\mathfrak{A}G$. 
\end{itemize}
In other words, a Clifford connection is given by a smooth family of Clifford connections on $r^*W \to G_x$ for each $x \in G^{(0)}$, satisfying the right invariance. 
The associated Dirac operators $D_x^W$ form a right invariant family, and define the element $D^W \in \mathrm{Diff}^1(G; W)$ which coincides with the above definition.  
\end{defn}

\begin{ex}[$Spin^c$-Dirac operators]
Let $G$ be a Lie groupoid equipped with a differential $spin^c$-structure. 
The $spin^c$-structure on $\mathfrak{A}G$ gives the spinor bundle $S(\mathfrak{A}G) \to G^{(0)}$ with a $\mathrm{Cliff}(\mathfrak{A}G)$-module structure. 
Moreover, as in the classical case, the unitary connection $\nabla^L$ on the determinant line bundle, together with the fiberwise Levi-Civita connection $\nabla^{LC}$ for $G$ as in (\ref{levicivita}), determines a Clifford connection $\nabla^S$ on the complex spinor bundle of $\mathfrak{A}G$, denoted by $S(\mathfrak{A}G)$. 
We call the associated Dirac operator $D^S \in \mathrm{Diff}^1(G; S(\mathfrak{A}G))$ the $spin^c${\it -Dirac operator}. 
\end{ex}

\begin{ex}[Twisted $spin^c$ Dirac operators]\label{ex_twisted_spinc}
Let $G$ be a Lie groupoid equipped with a differential $spin^c$ structure. 
Let $E \to G^{(0)}$ be a $\mathbb{Z}_2$-graded hermitian vector bundle with unitary connection $\nabla^E$ which preserves the grading. 
If we denote by $c : \mathrm{Cliff}(\mathfrak{A}G) \to \mathrm{End}(S(\mathfrak{A}G))$ the Clifford action on the spinor bundle, 
$c \hat{\otimes} 1 : \mathrm{Cliff}(\mathfrak{A}G) \to \mathrm{End}(S(\mathfrak{A}G)\hat{\otimes}E)$ gives a Clifford module structure on $S(\mathfrak{A}G)\hat{\otimes}E$. 

For each $x \in G^{(0)}$, denote $r_x := r|_{G_x} :G_x \to G^{(0)}$. 
We consider the pullback connection $r_x^*\nabla^E$ on $r_x^*E \to G_x$ for each $x \in X$. 
These combine to give a right invariant continuous linear map
\[
r^*\nabla^E : C^\infty(G; r^*E) \to C^\infty(G; r^*E \otimes r^*\mathfrak{A}G). 
\]
The map
\begin{align*}
\nabla^{S\hat{\otimes}E} &: C^\infty(G; r^*(S(\mathfrak{A}G)\hat{\otimes}E)) \to C^\infty(G; r^*((\mathfrak{A}G)\hat{\otimes}E) \otimes r^*\mathfrak{A}^*G) \\
\nabla^{S\hat{\otimes}E} &:= \nabla^S \hat{\otimes} 1 + 1 \hat{\otimes} r^*\nabla^E
\end{align*}
gives a Clifford connection on $S\hat{\otimes}E$. 
We call the associated Dirac operator $D^{S\hat{\otimes}E} \in \mathrm{Diff}^1(G; S(\mathfrak{A}G)\hat{\otimes}E)$ the {\it spin Dirac operator twisted by} $(E, \nabla^E)$. 
\end{ex}

\begin{ex}[The signature operator]
Let $G$ be a Lie groupoid equipped with a metric on $\mathfrak{A}G$. 
As in the classical case, the complexified exterior algebra bundle $\wedge_{\mathbb{C}} \mathfrak{A}^*G \to G^{(0)}$ has the $\mathrm{Cliff}(\mathfrak{A}G)$-module structure. 
The fiberwise Levi-Civita connection as in (\ref{levicivita}) induces a Clifford connection on $\wedge_{\mathbb{C}} \mathfrak{A}^*G \to G^{(0)}$. 
We call the associated Dirac operator $D^{\mathrm{sign}} \in \mathrm{Diff}^1(G; \wedge_{\mathbb{C}} \mathfrak{A}^*G )$ the {\it signature operator} on $G$. 
Of course this is the family consisting of the signature operator on $G_x$ for each $x\in G^{(0)}$.  
If the rank of $\mathfrak{A}G$ is even (let us denote it by $n$), the exterior algebra bundle $\wedge_{\mathbb{C}}\mathfrak{A}^*G$ is $\mathbb{Z}_2$-graded by the Hodge star operator. 
We only consider this grading on complexified exterior algebra bundles of even-rank real vector bundles in this paper. 
Under this grading, the signature operators are odd. 
\end{ex}

\subsubsection{Ellipticity and index classes}\label{subsubsec_ellipticity}
From now on we assume that $G^{(0)}$ is compact. 

A $G$-pseudodifferential operator $D$ is called {\it elliptic} if $\sigma(D)$ is invertible. 
If $D \in \Psi_c^m(G; E, F)$ is elliptic, as in the classical situations, it has a parametrix $Q \in \Psi_c^{-m}(G; F, E)$ such that $DQ - \mathrm{Id} \in \Psi^{-\infty}_c(G; F)$ and $QD - \mathrm{Id} \in \Psi^{-\infty}_c(G; E)$. 

For an elliptic operator $F \in \Psi_c^0(G; E, F)$, we can define the {\it index class} $\mathrm{Ind}(F) \in K_0(C^*(G))$ as follows. 
For simplicity we work in the case where coefficient bundles are trivial; for the general case we use the nontrivial-coefficient version of algebras (such as $C^*(G; E)$) which are Morita equivalent to trivial coefficient versions ($C^*(G)$), and proceed exactly in the same way. 
We have $\Psi^0_c(G) \subset \mathcal{M}(C^*(G))$, where $\mathcal{M}(C^*(G))$ denotes the multiplier algebra of $C^*(G)$. 
We denote $\overline{\Psi^0_c(G)}$ the completion of $\Psi^0_c(G)$ by the norm induced from $\mathcal{M}(C^*(G))$. 
The $*$-homomorphism $\sigma : \Psi^0_c(G)  \to C^\infty(\mathfrak{S}^*(G)) $ extends to the $*$-homomorphism $\sigma : \overline{\Psi^0_c(G)}  \to C(\mathfrak{S}^*(G))$ and fits into the exact sequence
\[
0 \to C^*(G) \to \overline{\Psi^0_c(G)}  \stackrel{\sigma}{\to} C(\mathfrak{S}^*(G)) \to 0. 
\]
We denote the connecting element associated to the above exact sequence by $\mathrm{ind}^{G^{(0)}}(G) \in KK^1(C(\mathfrak{S}^*(G)) ,C^*(G) )$. 

For $F \in \overline{\Psi^0_c(G)}$, we say it is elliptic if $\sigma(F)$ is invertible.  
When $F$ is elliptic, it defines a class $[\sigma(F)] \in K_1(C(\mathfrak{S}^*(G)))$. 
We define the {\it index class} of the elliptic operator $F$ as
\[
\mathrm{Ind}(F) := [\sigma(F)]  \otimes \mathrm{ind}^{G^{(0)}}(G) \in K_0(C^*(G)). 
\]

Suppose there is a compact space $Y$ and a continuous map $f : G^{(0)} \to Y$ such that $f^{-1}(y)$ is saturated for $G$ for all $y\in Y$.
Then $C(Y)$ acts canonically on $C^*$-algebras associated to $G$ above, namely $\overline{\Psi^0_c(G)}$, $C^*(G)$ and $C(\mathfrak{S}^*(G))$. 
These actions commute with elements in these algebras, so for an elliptic operator $F \in \overline{\Psi^0_c(G)}$, we get a finer index class, 
\begin{equation}\label{eq_ind_y}
\mathrm{Ind}^Y(F) \in KK(C(Y), C^*(G)), 
\end{equation}
which maps to $\mathrm{Ind}(F)$ by the $*$-homomorphism $\mathbb{C} \to C(Y)$. 

Next, we consider the case where an elliptic operator $F \in \overline{\Psi^0_c(G)}$ is invertible in some closed saturated subset of $G^{(0)}$. 
Recall that we call a subset $A \subset G^{(0)}$ {\it saturated} if $G_A = G^A_A = G^A$. 
Assume $X \subset G^{(0)}$ is closed and saturated, and $G_X$ is amenable. 
We have the following diagram, whose rows and columns are all exact. 
\begin{equation}\label{diag_3by3}
\xymatrix{
&0 \ar[d]&0\ar[d]&0\ar[d]& \\
0 \ar[r]&C^*(G|_{G^{(0)}\setminus X}) \ar[r] \ar[d]&C^*(G) \ar[r] \ar[d] & C^*(G|_X) \ar[r] \ar[d] &0 \\
0 \ar[r]&\overline{\Psi^0_c(G|_{G^{(0)}\setminus X})} \ar[r] \ar[d]_\sigma&\overline{\Psi^0_c(G)} \ar[r] \ar[d]_\sigma & \overline{\Psi^0_c(G|_X)} \ar[r] \ar[d]_\sigma &0 \\
0 \ar[r]&C_0(\mathfrak{S}^*G|_{G^{(0)}\setminus X}) \ar[r] \ar[d]&C(\mathfrak{S}^*G) \ar[r] \ar[d] & C(\mathfrak{S}^*G|_X) \ar[r] \ar[d] &0 \\
&0 &0&0&
}
\end{equation}
Throughout this article, we denote the connecting element of the top row of the above exact sequence by $\partial^{G^{(0)}\setminus X}(G) \in KK^1(C^*(G|_X), C^*(G|_{G^{(0)}\setminus X}))$. 
We denote the pullback $C^*$-algebra of the downright corner of the above diagram by $\Sigma^{G^{(0)} \setminus X}(G) :=\overline{\Psi^0_c(G|_X)} \oplus_X C(\mathfrak{S}^*G)$. 
We have the exact sequence
\begin{equation}\label{fullsymb}
0 \to C^*(G|_{G^{(0)}\setminus X}) \to \overline{\Psi^0_c(G)}  \stackrel{\sigma_{f, X}}{\to} \Sigma^{G^{(0)} \setminus X}(G) \to 0. 
\end{equation}
We call $\sigma_{f, X}$ the full symbol map with respect to $X$. 
Since we are assuming that $G|_{X}$ is amenable, this exact sequence is semisplit. 
Denote $\mathrm{ind}^{G^{(0)}\setminus X}(G) \in KK^1(\Sigma^{G^{(0)} \setminus X}(G), C^*(G|_{G^{(0)}\setminus X}))$ the connecting element of the exact sequence (\ref{fullsymb}). 

Suppose $F \in \overline{\Psi^0_c(G)}$ is elliptic and its restriction to $G|_X$, $F|_{X} \in \overline{\Psi^0_c(G|_X)}$, is invertible. 
We call such an operator as {\it fully elliptic with respect to $X$}. 
This means that $\sigma_{f, X}(F) \in \Sigma^{G^{(0)} \setminus X}(G)  $ is invertible, so it defines a class $[\sigma_{f, X}(F)] \in K^1(\Sigma^{G^{(0)} \setminus X}(G))$. 
The class
\[
\mathrm{Ind}_{G^{(0)} \setminus X}(F) :=[\sigma_{f, X}(F)] \otimes \mathrm{ind}^{G^{(0)} \setminus X}(G) \in K_0(C^*(G|_{G^{(0)}\setminus X}))
\]
is called the {\it full index class} of the operator $F$ with respect to $G^{(0)} \setminus X$. 

In the case of an elliptic positive order pseudodifferential operator $D\in \Psi^m_c(G)$, we also define the index class as follows. 
Let us denote $\psi(x) = \frac{x}{\sqrt{1+x^2}}$ and consider the operator $\psi(D)$. 
This operator satisfies $\psi(D) \in \overline{\Psi^0_c(G)}$, as shown in [V] (note that it is  not in $\Psi^0_c(G)$ in general). 
Then we define its {\it index class} as 
\[
\mathrm{Ind}(D):= \mathrm{Ind}(\psi(D)) \in K_0(C^*(G)). 
\]

In the case $D \in  \Psi^m_c(G)$ is invertible on a closed saturated subset $X \subset G^{(0)}$, the bounded transform $\psi(D)$ is fully elliptic with respect to X. 
In this case we also say that $D$ is fully elliptic with respect to X, and define its full index class as 
\[
\mathrm{Ind}_{G^{(0)}\setminus X}(D):= \mathrm{Ind}_{G^{(0)}\setminus X}(\psi(D)) \in K_0(C^*(G|_{G^{(0)}\setminus X})). 
\] 

\subsubsection{Deformation goupoids and blowup groupoids}\label{dnc_blup}
Here we recall the two constructions of groupoids; deformation to the normal cone and blowup. 
For details we refer to \cite{DS}. 

First we explain these constructions for manifolds. 
Let $Y$ be a manifold and $X$ a locally closed submanifold. 
First we explain in the case $X$ does not intersect with the boundary of $Y$. 
Denote by $N_X^Y$ the normal bundle of $X$ in $Y$. 
The {\it deformation to the normal cone}, denoted by $DNC(Y, X)$, is a smooth manifold which is obtained by gluing $N_X^Y \times \{0\}$ with $Y \times \mathbb{R}^*$. 
Choose an exponential map $\theta : U' \to U$, where $U'$ is an open neighborhood of the $0$-section in $N_X^Y$ and $U$ is an open neighborhood of $X$ in $Y$.  
The smooth structure is defined in the way that the following maps are diffeormorphisms onto open subsets of $DNC(Y, X)$. 
\begin{itemize}
\item the inclusion $Y \times \mathbb{R}^* \to DNC(Y, X)$. 
\item the map $\Theta : \Omega' :=\{((x,\xi), \lambda) \in N_X^Y \times \mathbb{R} \ | \ (x, \lambda \xi) \in U' \} \to DNC(Y, X)$ defined by $\Theta((x, \xi), 0) = ((x, \xi), 0) $ and $\Theta((x, \xi), \lambda)=(\theta(x, \lambda \xi), \lambda) \in Y\times \mathbb{R}^*$ if $\lambda \neq 0$. 
\end{itemize}
This condition defines the smooth structure on $DNC(Y, X)$ uniquely and it does not depend on the choice of $\theta$. 
We also denote by $DNC_+(Y, X) := Y \times (\mathbb{R}^*_+) \sqcup N_X^Y \times \{0\} \in DNC(Y, X)$. 

There exists a canonical action of the group $\mathbb{R}^*$ on the manifold $DNC(Y, X)$, called the {\it gauge action}. 
This is defined by, for an element $\lambda \in \mathbb{R}^*$, $\lambda.(w, t) = (w, \lambda t)$ and $\lambda. ((x, \xi), 0) = ((x, \lambda^{-1}\xi), 0)$ (with $t \in \mathbb{R}^*$, $w \in Y$, $x \in X$ and $\xi \in ((N_X^Y)_x)$). 
This action is free and locally proper on the open subset $DNC(Y, X) \setminus X \times \mathbb{R}$. 

The $DNC$-construction has the functoriality as follows. 
Let $f : (Y, X) \to (Y', X') $ be a smooth map between the pair as above. 
We can show that $f$ induces a smooth map
\[
DNC(f) : DNC(Y, X) \to DNC(Y', X'). 
\]
This map is equivariant with respect to the gauge action by $\mathbb{R}^*$. 

The blowup $Blup(Y, X)$ is a smooth manifold which is a union of $Y \setminus X$ with $\mathbb{P}(N_X^Y)$, the projective space of the normal bundle $N_X^Y$. 
We also define the spherical blowup $SBlup(Y, X)$, which is a manifold with boundary obtained by gluing $Y \setminus X$ with the sphere bundle $\mathbb{S}(N_X^Y)$. 
The definition is as follows. 
\begin{align*}
Blup(Y, X) &:= (DNC(Y, X) \setminus X \times \mathbb{R} )/\mathbb{R}^* \\
SBlup(Y, X) &:= (DNC_+(Y, X) \setminus X \times \mathbb{R}_+ )/\mathbb{R}^*_+. 
\end{align*}
Here we take quotient by the gauge action. 

The functoriality of $Blup$ is described as follows. 
Let $f : (Y, X) \to (Y', X') $ be a smooth map between the pair as above. 
Let $U_f := DNC(Y, X)\setminus DNC(f)^{-1}(X' \times \mathbb{R})$. 
Denote $Blup_f(Y, X) := U_f / \mathbb{R}^*$. 
Then we obtain a smooth map $Blup(f) : Blup_f(Y, X) \to Blup(Y', X')$. 
Similarly we obtain a smooth map $SBlup(f) : SBlup_f(Y, X) \to SBlup(Y', X')$. 

Let us explain the case where $Y$ is a manifold with corners and $X$ meets $\partial Y$. 
$X$ is called an {\it interior $p$-submanifold} of $Y$ if it is a smooth submanifold which meets all the boundary faces of $Y$ transversally, and covered by coordinate neighborhoods $\{U, (v, w)\}$ in $Y$ such that $v$ is a tuple of boundary defining functions on $Y$ and $U \cap X = \{w_i = 0 \ | \ i = 1, \cdots, \mathrm{codim} X\}$. 
If $X$ is an interior $p$-submanifold of $Y$, we consider the inward normal bundle $(N_X^{Y})^+$ and we can define $DNC_+(Y, X) = (N_X^Y)^+ \times \{0\}\sqcup Y \times \mathbb{R}^*_+$. 
This manifold admits the gauge action by $\mathbb{R}^*_+$. 
We define $SBlup(Y, X )$ by the same formula as above. 

Now we apply these constructions to inclusions of Lie groupoids. 
Let $\Gamma$ be a closed Lie subgroupoid of a Lie groupoid $G$. 
First we assume that $\Gamma$ does not meet the boundary of $G$. 
Using the functoriality of the $DNC$ construction, we get a Lie groupoid
\[
DNC(G, \Gamma) \rightrightarrows DNC(G^{(0)}, \Gamma^{(0)}), 
\]
where the source and range maps are $DNC(s)$ and $DNC(r)$, and the multiplication is $DNC(m)$. 
Denoting the subset $\widetilde{DNC}(G, \Gamma) : = U_r \cap U_s \subset DNC(G, \Gamma)$, we define $Blup_{r, s}(G, \Gamma) := \widetilde{DNC}(G, \Gamma)/\mathbb{R}^*$. 
We get a Lie groupoid
\[
Blup_{r, s}(G, \Gamma) \rightrightarrows Blup(G^{(0)}, \Gamma^{(0)}), 
\]
where the source and range maps are $Blup(s)$ and $Blup(r)$, and the multiplication is $Blup(m)$. 
Using $\widetilde{DNC}_+(G, \Gamma) := \widetilde{DNC}(G, \Gamma) \cap DNC_+(G, \Gamma)$ instead of $\widetilde{DNC}(G, \Gamma)$ in the above construction and taking the quotient by the gauge action of $\mathbb{R}^*_+$, we get a Lie groupoid
\[
SBlup_{r, s}(G, \Gamma) \rightrightarrows SBlup_{r,s}(G^{(0)}, \Gamma^{(0)}). 
\]

In the case $\Gamma$ meets the boundary of $G$, 
if $\Gamma$ is a $p$-submanifold of $G$, we can define the Lie groupoids $DNC_+(G, \Gamma)$ and $SBlup_{r, s}(G, \Gamma)$ in the same way as above. 
\subsection{$b$, $\Phi$, $e$-calculus and corresponding groupoids}
In this subsection, we recall the basics of $b$, $\Phi$, and $e$-calculus, in terms of the groupoid approach. 
The settings are as follows. 
\begin{itemize}
\item Let $(M, \partial M)$ be a compact manifold with closed boundaries.  
Here closed means that $\partial M$ is a compact manifold without boundary. 
\item Let $\partial M = \sqcup_{i = 1}^m H_i$ be the decomposition into connected components.  
\item Let $\pi_i : H_i \to Y_i$ be a smooth oriented fiber bundle structure with closed fibers. 
The typical fibers are allowed to vary from one component to another.  
We also denote $Y = \sqcup_{i} Y_i$ and $\pi : \sqcup_{i} \pi_i$. 
\item Let $x \in C^\infty(M)$ be a boundary defining function. 
Here a boundary defining function is a smooth function $x$ on $M$ such that $x^{-1}(0) = \partial M$, $x > 0$ on $\mathring{M}$ and $dx(p) \neq 0$ for all $p \in \partial M$.  
\end{itemize}

Define $\mathcal{V}(M) := C^\infty(M; TM)$. 
Consider the subspaces $\mathcal{V}_b(M)$, $\mathcal{V}_\Phi(M)$ and $\mathcal{V}_e(M)$ of $\mathcal{V}(M)$, defined as follows. 
\begin{align*}
\mathcal{V}_b (M) &:= \{\xi \in \mathcal{V}(M) \ | \ \xi |_{\partial M}\mbox{ is tangent to }\partial M \} \\
\mathcal{V}_\Phi (M) &:= \{\xi \in \mathcal{V}(M) \ | \ \xi |_{\partial M}\mbox{ is tangent to the fibers of }\pi \mbox{ and } \xi(x) \in x^2C^\infty(M) \} \\
\mathcal{V}_e (M) &:= \{\xi \in \mathcal{V}(M) \ | \ \xi |_{\partial M}\mbox{ is tangent to the fibers of }\pi \}.  
\end{align*}   
These are Lie subalgebras of $\mathcal{V}(M)$. 
Using the Serre-Swan theorem, we see that there exist smooth vector bundles $T^bM$, $T^\Phi M$, and $T^e M$ over $M$ such that $\mathcal{V}_b(M) = C^\infty(M; T^bM)$, $\mathcal{V}_\Phi(M) = C^\infty(M; T^\Phi M)$, and $\mathcal{V}_e(M) = C^\infty(M; T^eM) $. 
Note that, restricted to $\mathring{M}$, these vector bundles are canonically isomorphic to $T\mathring{M}$. 

A $b$, $\Phi$, $e$-metric on $M$ is a smooth riemannian metric on the vector bundles $T^bM$, $T^\Phi M$, and $T^e M$, respectively. 
We also call a riemannian metric on $\mathring{M}$ a $b$, $\Phi$, $e$-metric, if it extends to a smooth metric on these vector bundles. 
Examples of such metrics are described as follows. 
Let $T\partial M \simeq \pi^* TY \oplus T^V\partial M$ be a fixed splitting for the boundary fibration. 
We consider three classes of metrics on $M$, which have the following forms near the boundary. 
\begin{align}
g_b &= \frac{dx^2}{x^2} \oplus g_{\partial M} \notag \\
g_\Phi &= \frac{dx^2}{x^4} \oplus \frac{\pi^*g_Y}{x^2} \oplus g_\pi \label{orthomet_phi}\\
g_e &= \frac{dx^2}{x^2} \oplus \frac{\pi^*g_Y}{x^2} \oplus g_\pi. \label{orthomet_e} 
\end{align}
Here $g_{\partial M}$ and $g_Y$ are some riemannian metrics on $\partial M$ and $Y$ respectively, and $g_\pi$ is a fiberwise riemannian metric for $\pi$. 
These are examples of $b$, $\Phi$, $e$-metrics respectively, and a metric of the form above is called {\it rigid}. 

Denote by $\Omega^b$, $\Omega^\Phi$ and $\Omega^e$ the bundle of smooth densities on the vector bundle $T^bM$, $T^\Phi M$ and $T^eM$, respectively, and we call them $b$, $\Phi$, $e$-density bundles, respectively.  

We define the space of $b$, $\Phi$, and $e$-pseudodifferential operators. 
Let $\mathrm{Diff}^*_b(M)$ denote the filtered algebra generated by $\mathcal{V}_b(M)$ and $C^\infty(M)$. 
An element in this algebra is called a $b$-differential operator. 
The space of $b$-pseudodifferential operators contains this algebra. 
We define the algebra $\mathrm{Diff}^*_\Phi(M)$ and $\mathrm{Diff}^*_e(M)$ in the analogous way, and the analogous result holds. 
This space of pseudodifferential operators can be described in two ways, {\it microlocal approach} and {\it groupoid approach}.  
The microlocal approach originates from Melrose \cite{Me} for the $b$-case, and the $\Phi$-case was given by Mazzeo and Melrose \cite{MM} and the $e$-case was given by Mazzeo \cite{Maz}.
In this paper, we use the groupoid approach, which is more suited with $K$-thoretic approach using $C^*$-algebras, as explained below. 
For relations between these two approaches, see \cite[Section 6.6]{PZ}. 

\subsubsection{The groupoid approach}\label{subsubsec_groupoid_approach}
Here we recall the groupoid approach. 
We can construct groupoids $G_b$, $G_\Phi$, $G_e$ associated to a manifold with fibered boundary, and define $b$, $\Phi$, $e$-pseudodifferential operators as operators in $\Psi^*_c(G_b)$, $\Psi^*_c(G_\Phi)$ and $\Psi^*_c(G_e)$, respectively. 
The groupoid corresponding to $b$-calculus is introduced by \cite{Mo}, and a general construction by \cite{N} includes the $\Phi$ and $e$ cases. 
Here we use the description using the blowup construction of groupoids. 
We use the blowup construction for groupoids explained in the subsubsection \ref{dnc_blup}. 
For this description, also see \cite[Section 13]{PZ}. 
\begin{itemize}
\item The $b$-groupoids. 

We start with the pair groupoid $M \times M \rightrightarrows M$. 
Note that this does not satisfy the definition of Lie groupoid given in Definition \ref{def_lie_groupoid}, since $s$ is not a submersion; however it is easy to see that the spherical blowup construction is also valid in this case. 
Consider the subgroupoid 
$\sqcup_i (H_i \times H_i) \rightrightarrows \partial M$ of $M\times M$. 
Then $b$-groupoid of $M$ is defined by
\[
G_b := SBlup_{r, s}(M \times M, \sqcup_i (H_i \times H_i) ) \rightrightarrows M. 
\]
$\mathring{M}$ and $H_i$, $1 \leq i \leq m$ are saturated subsets of $G_b$, and we have
\[
G_b = \mathring{M} \times \mathring{M} \sqcup  \sqcup_i (H_i \times H_i \times \mathbb{R}) \rightrightarrows M. 
\]

\item The $\Phi$-groupoids. 

Consider the subgroupoid $\partial M \times_\pi \partial M = \sqcup_{i} (H_i \times_{\pi_i} H_i) \rightrightarrows \partial M$ of $G_b$. 
Then $\Phi$-groupoid of $M$ is defined by
\[
G_\Phi := SBlup_{r, s}(G_b, \partial M \times_\pi \partial M ) \rightrightarrows M. 
\]
Let us look at the singular part. 
The inward normal bundle groupoid of $H_i \times_{\pi_i} H_i$ in $G_b$ is 
\begin{align*}
H_i \times_{\pi_i} TY_i \times_{\pi_i} H_i \times \mathbb{R} \times \mathbb{R}_+
&\rightrightarrows H_i \times \mathbb{R}_+ \\
s(x, v, y, a, b) &= (y, b) \\
r(x, v, y, a, b) &= (x, b) \\
m((x, v, y, a, b), (y, w, z, a', b)) &= (x, v+w, z, a+a', b)
\end{align*}
And the gauge action by $\lambda \in \mathbb{R}_+^*$ is given by $(x, v, y, a, b) \mapsto (x, \lambda v, y, a, \lambda b)$. 
Thus dividing by this gauge action, we get an isomorphism 
\[
G_\Phi|_{H_i} \simeq H_i \times_{\pi_i} H_i \times_{\pi_i} TY_i\times \mathbb{R}
\]
(this can be seen by restricting to $b = 1$). 
In other words we have
\[
G_\Phi = \mathring{M} \times \mathring{M} \sqcup \partial M \times_\pi \partial M \times_\pi TY \times \mathbb{R} \rightrightarrows M. 
\]
  
\item The $e$-groupoids. 

Consider the groupoid $M \times M \rightrightarrows M$ and its subgroupoid 
$\partial M \times_\pi \partial M =\sqcup_i (H_i \times_{\pi_i} H_i) \rightrightarrows \partial M$. 
Then $e$-groupoid of $M$ is defined by
\[
G_e = SBlup_{r, s}(M \times M, \partial M \times_\pi \partial M ) \rightrightarrows M. 
\]
Let us look at the singular part. 
The inward normal bundle groupoid of $ H_i \times_{\pi_i} H_i$ in $M \times M$ is
\begin{align*}
H_i \times_{\pi_i} TY_i \times_{\pi_i} H_i \times ( \mathbb{R}_+)^2
&\rightrightarrows H_i \times \mathbb{R}_+ \\
s(x, v, y, a, b) &= (y, b) \\
r(x, v, y, a, b) &= (x, a) \\
m((x, v, y, a, b), (y, w, z, b, c)) &= (x, v+w, z, a, c)
\end{align*}
And the gauge action by $\lambda \in \mathbb{R}_+^*$ is given by $(x, v, y, a, b) \mapsto (x, \lambda v, y, \lambda a, \lambda b)$. 
So dividing by this action, we get
\[
G_e|_{H_i} \simeq H_i \times_{\pi_i} H_i \times_{\pi_i} (TY_i \rtimes \mathbb{R}_+^*)
\]
where $\mathbb{R}_+^*$ acts on $TY_i$ by multiplication. 
\end{itemize}

We apply the general construction of the subsubsection \ref{subsubsec_ellipticity} to these groupoids. 
Recall that a $G_\Box$ operator $P \in \overline{\Psi^0_c(G)}$ is called elliptic if its symbol $\sigma(P) \in C(\mathfrak{S}^*G_\Box)$ is invertible. 
Note that $\partial M \subset M = G_\Box^{(0)}$ is a closed saturated submanifold. 
Applying the construction in (\ref{fullsymb}) to the case $\partial M =X$, we get the exact sequence
\[
0 \to C^*(G_\Box|_{\mathring{M}}) \to \overline{\Psi^0_c(G_\Box)}  \stackrel{\sigma_{f, \partial M}}{\to} \Sigma^{\mathring{M}}(G_\Box) \to 0. 
\]
We say $P\in \overline{\Psi^0_c(G)}$ is fully elliptic if $\sigma_{f, \partial M}(P) \in \Sigma^{\mathring{M}}(G_\Box)$ is invertible. 
Recall that if $P$ is fully elliptic it defines the index class $\mathrm{Ind}_{\mathring{M}}(P)\in K_0(C^*(G_\Box|_{\mathring{M}})) = K_0(K(L^2(\mathring{M}))) \simeq \mathbb{Z}$. 
By the exact sequence above, the restriction of a fully elliptic operator $P$ to $\mathring{M}$ is Fredholm, and its Fredholm index corresponds to $\mathrm{Ind}_{\mathring{M}}(P) \in \mathbb{Z}$. 
\section{Indices of geometric operators on manifolds with fibered boundaries : the case without perturbations}\label{sec_index_without_perturbation}

\subsection{The definition of indices}\label{subsec_def_index}

In subsections \ref{subsec_def_index} and \ref{subsec_properties}, for simplicity we only consider spin Dirac operators, without any twists or perturbations. 
For our conventions on spin structures and pre-spin structures on vector bundles, see Definition \ref{convention_spin}. 

For a given even dimensional compact manifold with fibered boundary $(M^{\mathrm{ev}}, \pi : \partial M \to Y)$ equipped with pre-spin structures on $TM$ and $TY$ as well as a riemannian metric on the vertical tangent bundle of the boundary fibration, $T^V\partial M$, for which the fiberwise spin Dirac operator forms an invertible family, we associate its index in $\mathbb{Z}$. 
This index can be realized using either $\Phi$-metrics or $e$-metrics. 
In the next section, 
we show that they actually coincide. 
Also we show some properties of this index, using groupoid deformation techniques. 
For simplicity, we only work in the case where $Y$ is odd dimensional. 
The case where $Y$ is even dimensional can be treated similarly. 

\begin{rem}\label{rem_spin_str}
For a manifold with fibered boundary $(M, \pi: \partial M \to Y)$, assume that we are given pre-spin structures on $TM$ and $TY$. 
The pre-spin structure on $TM$ induces a pre-spin structure on $T\partial M$. 
Choose any splitting $T\partial M = T^V\partial M \oplus \pi^*TY$. 
We introduce the pullback pre-spin structure on $\pi^*TY$. 
Then a pre-spin structure on $T^V\partial M$ is induced, and it does not depend on the choice of the splitting of $T\partial M$. 
We always consider this choice of pre-spin structure on $T^V\partial M$. 
In particular, when we are given pre-spin structures on $TM$ and $TY$ as well as a riemannian metric on $T^V\partial M$, a spin structure on $T^V\partial M$ is canonically induced and the fiberwise spin Dirac operator $D_\pi$ is defined. 

\end{rem}
 
First, we show that for a fixed spin structure on $T^\Phi M$ or $T^e M$ which has a product decomposition at the boundary, we get the Fredholmness from the invertibility of the fiberwise Dirac operators.  

Let $(M^{\mathrm{ev}}, \pi : \partial M \to Y^{\mathrm{odd}})$ be a compact manifold with fibered boundary, equipped with pre-spin structures on $TM$ and $TY$, as well as a riemannian metric $g_\pi$ on $T^V\partial M$. 

Fix some riemannian metric $g_Y$ on $Y$. 
Choose a smooth riemannian metric $g_\Phi$ ($g_e$) for $\mathfrak{A}G_\Phi$ ($\mathfrak{A}G_e$), whose restriction to $\mathfrak{A}G_\Phi |_{\partial M} = T^V\partial M \oplus \pi^*TY \oplus \mathbb{R}_x$ ($\mathfrak{A}G_e |_{\partial M} = T^V\partial M \oplus \pi^*TY \oplus \mathbb{R}_x$) can be written as
\begin{align*}
g_\Phi|_{\partial M} &= g_\pi \oplus \pi^*g_{Y} \oplus dx^2 \\
g_e|_{\partial M} &= g_\pi \oplus \pi^*g_{Y} \oplus dx^2. 
\end{align*}
For example rigid metrics as in (\ref{orthomet_phi}) and (\ref{orthomet_e}) on the interior $\mathring{M}$ extends to metrics on $\mathfrak{A}G_\Phi$ and $\mathfrak{A}G_e$ satisfying this condition. 
Let $D_\Phi \in \mathrm{Diff}^1(G_\Phi; S(\mathfrak{A}G_\Phi))$, $D_e\in \mathrm{Diff}^1(G_e; S(\mathfrak{A}G_e))$ be the spin Dirac operators associated to the metrics $g_\Phi$ and $g_e$, respectively. 
Denote $D_\pi$ the fiberwise spin Dirac operators for the boundary fibration structure $\pi$($D_\pi$ is a family of operators parametrized by $Y$). 

\begin{prop}\label{fredholmness}
In the above settings, 
assume that the family $D_\pi$ is invertible. 
Then both $D_\Phi$ and $D_e$ are Fredholm, as operators on $\mathring{M}$ with metric induced from $g_\Phi$ and $g_e$, respectively.  
\end{prop}
\begin{proof}
First we prove in the $\Phi$-case. 
We have the decomposition
\begin{equation}\label{phigrpd}
G_\Phi = \mathring{M} \times \mathring{M} \sqcup \partial M \times_\pi \partial M \times_\pi TY \times \mathbb{R} \rightrightarrows M. 
\end{equation}
The restriction of $D_\Phi$ to the singular part $ \partial M \times_\pi \partial M \times_\pi TY \times \mathbb{R}$ is a family of operators $D_\Phi|_{\partial M} = \{D_{\Phi, y}\}_{y \in Y}$ parametrized by $Y$, and each $D_{\Phi, y}$ is given by
\begin{align}\label{fredholmness_op1}
D_{\Phi, y} &: C_c^\infty(\pi^{-1}(y) \times T_y Y \times \mathbb{R}; S(\pi^{-1}(y)) \hat{\otimes} S(T_y Y \times \mathbb{R})) \notag \\
&\to  C_c^\infty(\pi^{-1}(y) \times T_y Y \times \mathbb{R}; S(\pi^{-1}(y)) \hat{\otimes} S(T_y Y \times \mathbb{R})) \notag \\
D_{\Phi, y} &= D_\pi \hat{\otimes} 1 + 1 \hat{\otimes} D_{T_yY \times \mathbb{R}}
\end{align}
Here $S(T_y Y \times \mathbb{R})$ and $D_{T_yY \times \mathbb{R}}$ is the translation invariant spinor bundle and the Dirac operator over the Euclidean space $T_yY \times \mathbb{R}$ with respect to the metric $g_Y \oplus dx^2 $. 
The operators $D_\pi \hat{\otimes} 1$ and $1 \hat{\otimes} D_{T_yY \times \mathbb{R}}$ anticommute, and since $D_\pi$ is invertible, we see that $D_{\Phi, y}$ is invertible for all $y\in Y$. 
So $D_\Phi|_{\partial M}$ is invertible. 
Thus $D_\Phi$ is fully elliptic and we get the Fredholm index
\[
\mathrm{Ind}_{\mathring{M}}(D_\Phi) \in K_0(C^*(G_\Phi|_{\mathring{M}})) \simeq \mathbb{Z}. 
\]

Next, we prove the Proposition in the $e$-case. 
The restriction of $D_e$ to the boundary component $G_e|_{\partial M} =\partial M \times_{\pi} \partial M \times_{\pi} (TY \rtimes \mathbb{R}_+^*)$ is described as follows. 
$D_e|_{\partial M}$ is a family of operators parametrized by $Y$, $D_e|_{\partial M} = \{D_{e, y}\}_{y \in Y}$, and for each $y \in Y$, we have
\begin{align}\label{fredholmness_op2}
D_{e, y} & : C^\infty(\pi^{-1}(y) \times (T_yY \rtimes \mathbb{R}_{+}^*); S(\pi_i^{-1}(y)) \hat{\otimes} S(T_yY \rtimes \mathbb{R}_{+}^*)) \notag \\
&\to C^\infty(\pi^{-1}(y) \times (T_yY \rtimes \mathbb{R}_{+}^*); S(\pi_i^{-1}(y)) \hat{\otimes} S(T_yY \rtimes \mathbb{R}_{+}^*)) \notag \\
D_{e, y} &= D_{\pi^{-1}(y)} \hat{\otimes} 1 + 1 \hat{\otimes} D_{T_yY \rtimes \mathbb{R}_+^*}. 
\end{align}
Here, $S(T_yY \rtimes \mathbb{R}_{+}^*)$ and $D_{T_yY\rtimes \mathbb{R}_{+}^*}$ denotes the spinor bundle and its Dirac operator on the Lie group $T_yY \rtimes \mathbb{R}_{+}^*$ with the translation invariant spin structure and metric $g_Y \oplus dx^2$. 
From the same argument as in the $\Phi$-case above, we get the full ellipticity of $D_e$. 
\end{proof}

Next we show that the index only depends on the choice of the fiberwise metric $g_\pi$ for the boundary fibration, and does not depend on the choice of base metrics as well as interior metrics. 
We consider the following situations. 
\begin{enumerate}
\item Pre-spin structures $P'_M$ and $P'_Y$ on $TM$ and $TY$, respectively, are fixed. 
\label{stabdata1}
\item A riemannian metric $g_\pi$ on $T^V\partial M$ is fixed. 
Assume that the associated fiberwise spin Dirac operator $D_\pi$ is invertible. 
\label{stabdata2}
\item A smooth riemannian metric $g_\Phi$ for $\mathfrak{A}G_\Phi \simeq T^\Phi M \to M$, whose restriction to $\mathfrak{A}G_\Phi |_{\partial M} = T^V\partial M \oplus \pi^*TY \oplus \mathbb{R}$ can be written as
\[
g_\Phi|_{\partial M} = g_\pi \oplus \pi^*g_{TY \oplus \mathbb{R}}  ,
\] 
where $g_{TY \oplus \mathbb{R}}$ is some riemannian metric on the vector bundle $TY \oplus \mathbb{R} \to Y$. 
\label{stabcond1}
\item A smooth riemannian metric $g_e$ for $\mathfrak{A}G_e \simeq T^e M \to M$, whose restriction to $\mathfrak{A}G_e |_{\partial M} = T^V\partial M \oplus \pi^*TY \oplus \mathbb{R}$ can be written as
\[
g_e|_{\partial M} = g_\pi \oplus \pi^*g_{TY \oplus \mathbb{R}}  ,
\] 
where $g_{TY \oplus \mathbb{R}}$ is some riemannian metric on the vector bundle $TY \oplus \mathbb{R} \to Y$. 
\label{stabcond2}
\item Let us denote the spin Dirac operators associated to $g_\Phi$ and $g_e$ by $D_\Phi$ and $D_e$, respectively. 
\end{enumerate}
\begin{prop}[Stability]\label{stability}
Under the above situations, $\mathrm{Ind}_{\mathring{M}}(D_\Phi)$ and $\mathrm{Ind}_{\mathring{M}}(D_e)$ only depend on the data (\ref{stabdata1}) and (\ref{stabdata2}) above. 
It does not depend on the choice of $g_\Phi$ and $g_e$ which satisfy the conditions (\ref{stabcond1}) and (\ref{stabcond2}) above. 
\end{prop}

\begin{proof}
This can be proved by a simple homotopy argument. 
We prove in the $\Phi$-case. 
The $e$-case is similar. 
Let $g_\Phi^0$ and $g_\Phi^1$ be two choices of smooth metrics on $T^\Phi M$ which satisfies the condition (\ref{stabcond1}) (for the same fiberwise metric $g_\pi$). 
Let us denote $D_\Phi^0$ and $D_\Phi^1$ the spin Dirac operator with respect to these metrics. 
Letting $g_\Phi^t = tg_\Phi^0 + (1-t)g_\Phi^1$ for $t \in [0,1]$, we get a smooth path of riemannian metrics $T^\Phi M$ connecting $g_\Phi^0$ and $g_\Phi^1$. 
Note that for all $t\in[0, 1]$, $g_\Phi^t$ satisfies the condition (\ref{stabcond1}). 

Let us consider the groupoid $G_\Phi \times [0, 1]\rightrightarrows M\times[0,1]$. 
The metrics $\{g_\Phi^t\}_{t \in [0, 1]}$ give a smooth metric on $\mathfrak{A}{G_\Phi \times [0,1]}$. 
Under this metric and the spin structure induced from $\mathring{M}$, we get the spin Dirac operator $D_{\Phi}^{[0,1]}$. 
Since $D_{\Phi}^{[0, 1]}|_{\partial M \times \{t\}}$ is invertible for all $t$, we get the index
\[
\mathrm{Ind}_{\mathring{M} \times [0, 1]}(D_\Phi^{[0,1]}) \in K_0(C^*(G_\Phi|_{\mathring{M}})\times [0, 1]) \simeq \mathbb{Z}. 
\]
and we have, denoting the $*$-homomorphisms $ev_t : C^*(G_\Phi|_{\mathring{M}})\times [0, 1]) \to C^*(\mathcal{G}|_{\mathring{M}} \times \{t\})$ for $t \in [0, 1]$, 
\begin{align*}
(ev_0)_* \mathrm{Ind}_{\mathring{M} \times [0, 1]}(D_\Phi^{[0,1]}) &= \mathrm{Ind}_{\mathring{M}}(D_\Phi^0) \in K_0(C^*(G_\Phi|_{\mathring{M}})) \simeq \mathbb{Z}, \mbox{ and } \\
(ev_1)_* \mathrm{Ind}_{\mathring{M} \times [0, 1]}(D_\Phi^{[0,1]}) &= \mathrm{Ind}_{\mathring{M}}(D_\Phi^1) \in K_0(C^*(G_\Phi|_{\mathring{M}})) \simeq \mathbb{Z}. 
\end{align*}
Since $(ev_t)_* : K_0(C^*(\mathring{M} \times \mathring{M} \times [ 0, 1])) \simeq \mathbb{Z} \to K_0(C^*(\mathring{M} \times \mathring{M} \times \{t\})) \simeq \mathbb{Z}$ is the identity map on $\mathbb{Z}$ for all $t \in [0, 1]$, we get the result. 
\end{proof}

By Proposition \ref{stability}, in order to define the indices of spin Dirac operator $D_\Phi$ and $D_e$, we only have to specify the data (\ref{stabdata1}) and (\ref{stabdata2}) listed before the Proposition \ref{stability}. 
So we define the index of the triple $(P'_M, P'_Y, g_\pi)$ by the above number. 
\begin{defn}\label{def_index_spin}
Let $( M^{\mathrm{ev}} , \pi : \partial M \to Y^{\mathrm{odd}})$ be a compact manifold with fibered boundary. 
For a triple $(P'_M, P'_Y, g_\pi)$, where $P'_M$ and $P'_Y$ are pre-spin structures on $TM$ and $TY$, respectively, and $g_\pi$ is a riemannian metric on $T^V \partial M$ such that the associated fiberwise spin Dirac operator is an invertible family, we define its $\Phi$-index and $e$-index as
\begin{align*}
\mathrm{Ind}_\Phi(P'_M, P'_Y, g_\pi) &:= \mathrm{Ind}_{\mathring{M}}(D_\Phi) \\
\mathrm{Ind}_e(P'_M, P'_Y, g_\pi) &:= \mathrm{Ind}_{\mathring{M}}(D_e). 
\end{align*}
Here $D_\Phi \in \mathrm{Diff}^1(G_\Phi; S(\mathfrak{A}G_\Phi))$ and $D_e \in \mathrm{Diff}^1(G_e; S(\mathfrak{A}G_e))$ are spin Dirac operators which are defined by arbitrary choices of data (\ref{stabcond1}) and (\ref{stabcond2}). 
\end{defn}

\subsection{Properties}\label{subsec_properties}
 
First, we show that two indices $\mathrm{Ind}_\Phi(P'_M, P'_Y, g_\pi)$ and $\mathrm{Ind}_e(P'_M, P'_Y, g_\pi) $ actually coincide. 
\begin{prop}[Equality of $\Phi$ and $e$-indices]\label{equality}
For a compact manifold with fibered boundary $(M^{\mathrm{ev}}, \pi : \partial M \to Y^{\mathrm{odd}})$, assume that we are given pre-spin structures $P'_M$ and $P'_Y$ on $TM$ and $TY$, respectively, and a riemannian metric $g_\pi$ on $T^VM$, for which the fiberwise spin Dirac operator $D_\pi$ is an invertible family. 
Then we have
\[
\mathrm{Ind}_\Phi(P'_M, P'_Y, g_\pi)=\mathrm{Ind}_e(P'_M, P'_Y, g_\pi). 
\]
\end{prop}

\begin{proof}
Let us fix a splitting $T\partial M \simeq \pi^* TY \oplus T^V\partial M$, a boundary defining function $x$, and a metric $g_Y$ on $Y$. 
Fix a collar neighborhood $U$ of $\partial M$ and also fix an identification $\partial M \times [0, a) \simeq U$ which is compatible with $x$. 
Then consider the metric $g_\Phi$ and $g_e$ defined as (\ref{orthomet_phi}) and (\ref{orthomet_e}): these satisfy the conditions above. 
The idea is to consider the family of metrics
\begin{equation}\label{coincidence_eq}
g_e(t) = \frac{dx^2}{x^2(x^2+t^2)} \oplus \frac{\pi^*g_Y}{x^2} \oplus g_\pi. 
\end{equation}
and justify the limit $t \to 0$. 
This can be realized as follows. 

Consider the Lie algebra $TM \times [0, 1] \to M \times [0, 1]$ with the canonical Lie bracket (not to be confused with $T(M \times [0, 1])$). 
Consider the following $C^\infty(M \times [0, 1])$-submodule of $C^\infty(M \times [0, 1] ; TM \times [0, 1])$. 
\[
\mathcal{V} := \left\{
\begin{array}{c|l}
V \in C^\infty(M \times [0, 1] ; TM \times [0, 1]) & V|_{\partial M \times [0, 1]} \in C^\infty(\partial M \times [0, 1]; T^V\partial M \times [0,1]), \mbox{ and } \\
& V(x) \in x(x+t)C^\infty(M \times [0, 1]_t)
\end{array}
\right\}. 
\]
This is a Lie subalgebra of $C^\infty(M \times [0, 1] ; TM \times [0, 1])$. 
By the Serre-Swan theorem, there exists a smooth vector bundle $\mathfrak{A} \to M \times [0, 1]$, unique up to isomorphism, such that $C^\infty(M \times [0, 1]; \mathfrak{A}) \simeq \mathcal{V}$ as a $C^\infty(M \times [0, 1])$-module. 
The map
\[
p : C^\infty(M \times [0, 1]; \mathfrak{A}) \to C^\infty(M \times [0, 1]; T(M \times [0, 1]))
\]
induced by $\mathcal{V} \hookrightarrow C^\infty(M \times [0, 1] ; TM \times [0, 1]) \hookrightarrow C^\infty(M \times [0, 1]; T(M \times [0, 1]))$, 
gives a Lie algebroid structure on $\mathfrak{A} \to M \times [0, 1]$ with anchor $p$. 
We have the following. 
\begin{itemize}
\item $\mathfrak{A}|_{M \times \{t\}} = \mathfrak{A}G_e$ for all $t \in (0, 1]$. 
\item $\mathfrak{A}|_{M \times \{0\}} = \mathfrak{A}G_\Phi$. 
\item The metric $g_{\mathfrak{A}}$ on $\mathfrak{A}$, defined as (see (\ref{coincidence_eq}))
\[
g_{\mathfrak{A}} := \begin{cases}
g_e(t) & \mbox{on }M \times (0, 1]_t \\
g_\Phi & \mbox{on }M \times \{0\}, 
\end{cases}
\] 
gives a smooth metric on $\mathfrak{A}$. 
\end{itemize}
Since $p|_{\mathring{M} \times [0, 1]}$ is injective, $(\mathfrak{A}, p)$ is an almost injective Lie algebroid.
By \cite{D}, there exists a smooth Lie groupoid $\mathcal{G} \to M \times [0, 1]$ such that its Lie algebroid $\mathfrak{A}\mathcal{G}$ is isomorphic to $(\mathfrak{A}, p)$.  

We give the explicit definition of $\mathcal{G}$. 
As a set, 
\begin{align*}
\mathcal{G} &= G_\Phi \times \{0\} \sqcup G_e \times (0, 1]   \\
&= \mathring{M} \times \mathring{M} \times [0, 1] \sqcup \partial M\times_\pi \partial M \times_\pi (TY \rtimes \mathbb{R}^*_+) \times (0, 1] \sqcup \partial M \times_\pi \partial M \times_\pi TY \times \mathbb{R} \times \{0\},  
\end{align*}
We describe the smooth structure as follows. 
Recall we have fixed a tubular neighborhood $U$ of $\partial M$. 
Outside the collar neighborhood, the smooth structure $\mathcal{G} \setminus \mathcal{G}_{U\times [0, 1]}^{U \times [0, 1]}$ is given in the canonical way. 
On $U$, recall we have fixed the isomorphism $U \simeq \partial M \times [0, a)$ and $TU \simeq (T^V\partial M \oplus \pi^*TY \oplus \mathbb{R}) \times [0, a)$. 
Also fix an exponential map for $T\partial M$. 
On $\mathcal{G}_{U\times [0, 1]}^{U \times [0, 1]}$, we consider the following exponential map. 
\begin{align*}
&(T^V\partial M \oplus \pi^*TY \oplus \mathbb{R}) \times \partial M \times [0, a) \times [0, 1] \to \mathcal{G}_{U\times [0, 1]}^{U \times [0, 1]} \\
(\partial_z, \partial_y, \xi, p, x, t) &\mapsto ((\exp_p(\partial_z+x\partial_y), x +e^{x(x+t)\xi}), (p, x), t) \in U \times U \times [0, 1] \mbox{ for }x > 0 \\
&\mapsto (\exp_p(\partial_z), p, \partial_y, e^{t\xi}, t) \in \partial M\times_\pi \partial M \times_\pi (TY \rtimes \mathbb{R}^*_+) \times (0, 1] \mbox{ for } x = 0, t \in (0, 1] \\
&\mapsto (\exp_p(\partial_z), p, \partial_y, \xi, 0) \in \partial M \times_\pi \partial M \times_\pi TY \times \mathbb{R} \times \{0\}\mbox{ for } x = t = 0. 
\end{align*}
We define the smooth structure on $\mathcal{G}_{U\times [0, 1]}^{U \times [0, 1]}$ so that the above map is a diffeomorphism. 
This smooth structure does not depend on any of the choices. 
Note also that, restricted to $\partial M \times [0, 1]\subset M \times [0, 1]$, we have
\[
\mathcal{G}_{\partial M \times [0, 1]} \simeq DNC(\partial M \times_\pi \partial M \times_\pi (TY \rtimes \mathbb{R}_+^*), \partial M \times_\pi \partial M \times_\pi TY )|_{\partial M \times [0, 1]} \rightrightarrows \partial M \times [0, 1]. 
\]

We consider the spin structure on $\mathfrak{A}$ induced from the one on $T{M}$ and the metric $g_\mathfrak{A}$, and denote the associated spin Dirac operator by $\mathcal{D} \in \mathrm{Diff}^1(\mathcal{G}; S(\mathfrak{A}))$. 
We can show that $\mathcal{D}|_{\partial M \times [0, 1]} \in \mathrm{Diff}^1(\mathcal{G}_{\partial M \times [0, 1]}; S(\mathfrak{A}_{\partial M \times [0, 1]}))$ is invertible, as follows.  
By the invertibility of $D_\pi$, there exists $c >0$ such that $D_\pi^2 \geq c$. 
The operator $\mathcal{D}|_{\partial M \times [0, 1]}$ is given by a family of operators $\{D_{y, t}\}_{(y, t) \in Y \times [0, 1]}$ parametrized by $Y \times [0, 1]$. 
Each $D_{y, t}$ has the form (\ref{fredholmness_op1}) for $t=0$ and (\ref{fredholmness_op2}) for $t \in (0, 1]$. 
As in the proof of Proposition \ref{fredholmness}, we have $D^2_{y, t} \geq D^2_\pi \hat{\otimes} 1 \geq c$. 
Thus $\mathcal{D}|_{\partial M \times [0, 1]}$ is invertible. 

So we get the index class 
\[
\mathrm{Ind}_{\mathring{M} \times [0, 1]}(\mathcal{D}) \in K_0(C^*(\mathcal{G}|_{\mathring{M} \times [0, 1]})) \simeq \mathbb{Z}. 
\]
and we have, denoting the $*$-homomorphisms $ev_t : C^*(\mathcal{G}|_{\mathring{M} \times [0, 1]}) \to C^*(\mathcal{G}|_{\mathring{M} \times \{t\}})$ for $t \in [0, 1]$, 
\begin{align*}
(ev_0)_* \mathrm{Ind}_{\mathring{M} \times [0, 1]}(\mathcal{D}) &= \mathrm{Ind}_{\mathring{M}}(D_\Phi) \in K_0(C^*(G_\Phi|_{\mathring{M}})) \simeq \mathbb{Z}, \mbox{ and } \\
(ev_1)_* \mathrm{Ind}_{\mathring{M} \times [0, 1]}(\mathcal{D}) &= \mathrm{Ind}_{\mathring{M}}(D_e) \in K_0(C^*(G_e|_{\mathring{M}})) \simeq \mathbb{Z}. 
\end{align*}
Since $(ev_t)_* : K_0(C^*(\mathring{M} \times \mathring{M} \times [ 0, 1])) \simeq \mathbb{Z} \to K_0(C^*(\mathring{M} \times \mathring{M} \times \{t\})) \simeq \mathbb{Z}$ is the identity map on $\mathbb{Z}$ for all $t \in [0, 1]$, we get the result. 

\end{proof}

Next we show the gluing formula. 

\begin{prop}[The gluing formula]\label{gluing}
Consider the following situation. 
\begin{itemize}
\item $M^0$ and $M^1$ are manifolds with fibered boundaries as above, equipped with  pre-spin structures $P'_{M^i}$ and $P'_{Y^i}$ on $TM^i$ and $TY^i$, respectively, and a riemannian metric $g_{\pi^i}$ on $T^V\partial {M^i}$, for $i = 0, 1$.
\item Assume that on some components of $\partial M^0$ and $-\partial M^1$, we are given isomorphisms of the data $(\pi^i, P'_{M^i}, P'_{Y^i}, g_{\pi^i})$ restricted there.  
\item $(M, \partial M, \pi', Y')$ : the manifold with fibered boundary obtained by the above isomorphism of some boundary components. 
This manifold is equipped with the pre-spin structures $P'_M$ and $P'_{Y'}$ on $TM$ and $TY'$, respectively, and a riemannian metric $g_{\pi'}$ on $T^V\partial M$ induced by the ones on $M^i$. 
\item Assume that on each boundary components of $M^0$ and $M^1$, the fiberwise spin Dirac operators are invertible. 
\end{itemize}
Then, we have 
\begin{align*}
\mathrm{Ind}_e(P'_M, P'_{Y'}, g_{\pi'}) &= \mathrm{Ind}_e(P'_{M^0}, P'_{Y^0}, g_{\pi^0}) + \mathrm{Ind}_e(P'_{M^1}, P'_{Y^1}, g_{\pi^1}).  
\end{align*}
\end{prop}

\begin{proof}
We use a similar argument to the one in Proposition \ref{equality}. 
For simplicity we consider the case where the boundary of each $M^0$ and $M^1$ consists of one component, and the isomormorphism is given between $\partial M^0$ and $-\partial M^1$. 
In particular, the resulting manifold $M$ is a closed manifold in this case. 
The general case can be shown in an analogous way. 
We denote the image of $\partial M^0 \simeq -\partial M^1$ in $M$ by $H \subset M$, which is a closed hypersurface. 
Also we denote $\pi : H \to Y$ the fiber bundle structure induced from the ones on $\partial M^0 \simeq -\partial M^1$ and the given fiberwise metric as $g_\pi$. 

Consider the Lie algebra $TM \times [0, 1] \to M \times [0, 1]$ with the canonical Lie bracket. 
Consider the following $C^\infty(M \times [0, 1])$-submodule of $C^\infty(M \times [0, 1] ; TM \times [0, 1])$. 
\[
\mathcal{V} := \{
V \in C^\infty(M \times [0, 1] ; TM \times [0, 1]) \ | \ V|_{H \times \{0\}} \in C^\infty(H \times \{0\}; T^VH \times \{0\})
\}. 
\]
This is a Lie subalgebra of $C^\infty(M \times [0, 1] ; TM \times [0, 1])$. 
By the Serre-Swan theorem, there exists a smooth vector bundle $\mathfrak{A} \to M \times [0, 1]$, unique up to isomorphism, such that $C^\infty(M \times [0, 1]; \mathfrak{A}) \simeq \mathcal{V}$ as a $C^\infty(M \times [0, 1])$-module. 
The map
\[
p : C^\infty(M \times [0, 1]; \mathfrak{A}) \to C^\infty(M \times [0, 1]; T(M \times [0, 1]))
\]
induced by $\mathcal{V} \hookrightarrow C^\infty(M \times [0, 1] ; TM \times [0, 1]) \hookrightarrow C^\infty(M \times [0, 1]; T(M \times [0, 1]))$, 
gives a Lie algebroid structure on $\mathfrak{A} \to M \times [0, 1]$ with anchor $p$. 
We have the following. 
\begin{itemize}
\item $\mathfrak{A}|_{M \times \{t\}} = \mathfrak{A}(M \times M)=TM$ for all $t \in (0, 1]$. 
\item $\mathfrak{A}|_{M \times \{0\}} = \mathfrak{A}G_e^0 \cup_{H} \mathfrak{A} G_e^1$. 
Here we denoted the $e$-groupoid of $M^i$ by $G_e^i $ for $i = 0, 1$. 
\item The metric $g_{\mathfrak{A}}$ on $\mathfrak{A}$, defined as
\[
g_{\mathfrak{A}} := 
\frac{dx^2}{x^2 + t^2}\oplus \frac{\pi^*g_Y}{x^2+t^2} \oplus g_\pi 
\] 
gives a smooth metric on $\mathfrak{A}$. 
Here $x$ is a defining function for $H \subset M$ and $t$ is the $[0, 1]$-coordinate in $M \times [0, 1]$. 
The metric $g_Y$ can be any metric on $Y$. 
\end{itemize}
Since $p|_{M \times [0, 1] \setminus (Y \times \{0\})}$ is injective, $(\mathfrak{A}, p)$ is an almost injective Lie algebroid and by \cite{D} we can integrate this to get a Lie groupoid $\mathcal{G} \rightrightarrows M \times [0, 1]$. 
We can describe explicitly such groupoid which can be written as 
\[
\mathcal{G} = (G_e^0 \cup_H G_e^1) \times \{0\} \sqcup M \times M \times (0, 1] \rightrightarrows M \times [0, 1]. 
\]
The description is similar to the one in the proof of the Proposition \ref{equality}. 
We consider the spin Dirac operator $\mathcal{D} \in \Psi^1(\mathcal{G}; S(\mathfrak{A}))$ with respect to the given spin structure and metric $g_{\mathfrak{A}}$. 
The submanifold $H \times \{0\} \subset M \times [0, 1]$ is a closed saturated submanifold for $\mathcal{G}$. 
The restriction $\mathcal{D}|_{H \times \{0\}}$ is of the form (\ref{fredholmness_op2}), and since we are assuming that $D_\pi$ is invertible, we see that $\mathcal{D}|_{H \times \{0\}}$ is invertible. 
Thus we get the index class
\[
\mathrm{Ind}_{M \times [0, 1]\setminus (H \times \{0\})}(\mathcal{D}) \in K_0(C^*(\mathcal{G}|_{M \times [0, 1]\setminus (H \times \{0\})})). 
\]
Note that we have
\begin{align*}
(ev_0)_*\mathrm{Ind}_{M \times [0, 1]\setminus (H \times \{0\})}(\mathcal{D}) = (\mathrm{Ind}(D_e^0), \mathrm{Ind}(D_e^1)) &\in K_0(C^*(G_e^0|_{\mathring{M_0}})) \oplus K_0(C^*(G_e^1|_{\mathring{M_0}})) \simeq \mathbb{Z} \oplus \mathbb{Z} \\
(ev_1)_*\mathrm{Ind}_{M \times [0, 1]\setminus (H \times \{0\})}(\mathcal{D}) =\mathrm{Ind}(D_M) &\in K_0(C^*(M \times M)) \simeq \mathbb{Z}. 
\end{align*}
Here we denoted $D_M$ the spin Dirac operator on $M$. 
This coincides with $D_e$ in the statement of this proposition. 
Note that $[ev_0] \in KK(C^*(\mathcal{G}|_{M \times [0, 1]\setminus (H \times \{0\})}), C^*(G_e^0|_{\mathring{M_0}}) \oplus C^*(G_e^1|_{\mathring{M_1}}))$ is a $KK$-equivalence.  
Thus, it is enough to show that $[ev_0]^{-1} \otimes [ev_1] : \mathbb{Z} \oplus \mathbb{Z} \to \mathbb{Z}$ is given by addition. 

The groupoid $\mathcal{G}|_{M \times [0, 1)\setminus (H \times \{0\})}$ is an open subgroupoid of $M \times M \times [0, 1) \rightrightarrows M \times [0, 1)$. 
We get the following commutative diagram, 
\[
\xymatrix{
0 \ar[r]  &C^*(M\times M \times (0, 1)) \ar[r] \ar[d]&C^*(\mathcal{G}|_{M \times [0, 1)\setminus (H \times \{0\})}) \ar[r]\ar[d] &C^*(G_e^0|_{\mathring{M_0}}) \oplus C^*(G_e^1|_{\mathring{M_1}}) \ar[r]\ar[d]&0
\\
0 \ar[r]  &C^*(M\times M \times (0, 1)) \ar[r]  &C^*(M\times M\times [0, 1)) \ar[r] &C^*(M\times M) \ar[r] &0,
}
\]
where the rows are exact. 
The element $[ev_0]^{-1} \otimes [ev_1] \in KK(C^*(G_e^0|_{\mathring{M_0}}) \oplus C^*(G_e^1|_{\mathring{M_1}}) , C^*(M \times M))$ coincides with the connecting element of the top row. 
By the functoriality of connecting maps, we see that $[ev_0]^{-1} \otimes [ev_1] = [j_0\oplus j_1]$, where $j_i$ denotes the inclusion $j_i : C^*(G_e^i|_{\mathring{M_i}}) \hookrightarrow C^*(M\times M)$ for $i = 0, 1$. 
Since the inclusion $G_e^i|_{\mathring{M_i}} = \mathring{M_i} \times \mathring{M_i} \hookrightarrow M \times M$ is a Morita equivalence for $i = 0, 1$, we see that $[ev_0]^{-1} \otimes [ev_1] = [j_0]\oplus [j_1] : \mathbb{Z} \oplus \mathbb{Z} \to \mathbb{Z}$ induced between the $K_0$-groups is given by addition. 
\end{proof}

Next we show that the $\Phi$-index can be written as a limit of the Atiyah-Patodi-Singer (APS) indices. 
For a manifold with fibered boundary $(M, \pi : \partial M \to Y)$ as above, we fix riemannian metrics $g_Y$ and $g_\pi$ for $Y$ and $T^V\partial M$. 
For $\mu >0$, we consider a $b$-metric of the form
\begin{equation}\label{aps_met}
g_{b, \mu} =\frac{1}{\mu^2}(\frac{dx^2}{x^2} \oplus \pi^*(g_Y)) \oplus g_\pi
\end{equation}
on a collar neighborhood of the boundary. 
Denote the Dirac operator associated to this metric by $D_\mu$. 
As always we assume that $D_\pi$ is an invertible family. 
The boundary operator of $D_\mu$ is the Dirac operator on $\partial M$ with respect to the metric $\mu^{-2}\pi^*g_Y \oplus g_\pi$. 
It has the form
\[
D_{\mu, \partial} = D_\pi \hat{\otimes} 1 + \mu D_Y + \mu^2R
\]
where $D_{Y}$ is a first order differential operator whose principal symbol is equal to the Clifford multiplication by $TY$, and $R$ is an operator of order $0$, coming from the curvature of the fibration $\pi$. 
For the precise formula, we refer to \cite[Section 4]{BC}. 
As explained in the proof of Proposition 4.41 in \cite{BC}, the anticommutator $[D_\pi \hat{\otimes} 1, D_Y + \mu R]$ is a fiberwise operator, so using fiberwise elliptic estimate and invertibility of $D_\pi$, we see that for $0 < \mu <<1$, $D_{\mu, \partial}$ is invertible. 
When the boundary operator $D_{\mu, \partial}$ is invertible, the APS index $\mathrm{Ind}_{\mathrm{APS}}(D_\mu)$ of the $b$-operator $D_\mu$ is, by definition, the Fredholm index of $D_\mu$ as an operator on the $L^2$-space with respect to the metric $g_{b, \mu}$. 
Since $D_{\mu, \partial}$ stays invertible for $\mu > 0$ small enough, there exist a well-defined limit
\[
\lim_{\mu \to +0}\mathrm{Ind}_{\mathrm{APS}}(D_\mu) \in \mathbb{Z}. 
\]
(The existence of the limit can also be seen as a consequence of the proof of Proposition \ref{prop_lim_aps} below. )
\begin{prop}[The limit of the APS index is the $\Phi$-index]\label{prop_lim_aps}
We have
\[
\lim_{\mu \to +0}\mathrm{Ind}_{\mathrm{APS}}(D_\mu)
 = \mathrm{Ind}_\Phi(P'_M, P'_Y, g_\pi) .  
\]
\end{prop}
\begin{proof}
Again we use a similar argument as in Proposition \ref{equality}. 
Consider the Lie algebra $TM \times [0, 1] \to M \times [0, 1]$ with the canonical Lie bracket (not to be confused with $T(M \times [0, 1])$). 
Consider the following $C^\infty(M \times [0, 1])$-submodule of $C^\infty(M \times [0, 1] ; TM \times [0, 1])$. 
\[
\mathcal{V} := \left\{
\begin{array}{c|l}
V \in C^\infty(M \times [0, 1] ; TM \times [0, 1]) & V|_{\partial M \times \{0\}} \in C^\infty(\partial M \times \{0\}; T^V\partial M \times \{0\}), \mbox{ and } \\
& V(x) \in x(x+t)C^\infty(M \times [0, 1]_t)
\end{array}
\right\}. 
\]
This is a Lie subalgebra of $C^\infty(M \times [0, 1] ; TM \times [0, 1])$. 
As in the proof of Proposition \ref{equality}, this gives a Lie algebroid $\mathfrak{A} \to M \times [0, 1]$. 
The family of metric in (\ref{aps_met}) on $M \times (0, 1]$ and a $\Phi$-metric $g_\Phi$ on $M \times \{0\}$, which has the form $g_\Phi|_{\partial M} = g_\pi \oplus \pi^*g_{TY \oplus \mathbb{R}} $, gives a smooth metric $g_{\mathfrak{A}}$ on $\mathfrak{A} \to M \times [0, 1]_\mu$. 
We can construct a groupoid $\mathcal{G} \rightrightarrows M\times[0, 1]$ which integrates $\mathfrak{A}$ and can be written as
\[
\mathcal{G} = G_\Phi \times \{0\} \sqcup G_b \times (0, 1] \rightrightarrows M \times [0, 1]. 
\]
We denote the spin Dirac operator on $\mathcal{G}$ with respect to the metric $g_{\mathfrak{A}}$ by $\mathcal{D}$. 
As explained above, there exists a positive number $0 <\epsilon \leq 1$ such that $\mathcal{D}|_{\partial M \times [0, \epsilon]}$ is invertible. 
Thus we get the index class
\[
\mathrm{Ind}_{\mathring{M} \times [0, \epsilon]}(\mathcal{D}|_{M \times [0, \epsilon]}) \in K_0(C^*(\mathcal{G}_{\mathring{M} \times [0, \epsilon]})) \simeq \mathbb{Z}. 
\]
We have $\mathrm{Ind}_{\mathrm{APS}}(D_\mu) =(ev_\mu)_*\mathrm{Ind}_{\mathring{M} \times [0, \epsilon]}(\mathcal{D}|_{M \times [0, \epsilon]})\in K_0(C^*(\mathcal{G}_{\mathring{M} \times \{\mu\}})) \simeq \mathbb{Z}$ for all $0<\mu \leq \epsilon$. 
Moreover, the restriction of $\mathcal{D}$ to $M \times \{0\}$ is exactly the same as the operator $D_\Phi$ defining the index $\mathrm{Ind}_\Phi(P'_M, P'_Y, g_\pi) = \mathrm{Ind}_{\mathring{M}}(D_\Phi)$. 
We have $\mathrm{Ind}_{\mathring{M}}(D_\Phi) = (ev_0)_*\mathrm{Ind}_{\mathring{M} \times [0, \epsilon]}(\mathcal{D}|_{M \times [0, \epsilon]})\in K_0(C^*(\mathcal{G}_{\mathring{M} \times \{0\}})) \simeq \mathbb{Z}$. 
Since $ev_\mu$ induces the identity map on $\mathbb{Z}$ for all $\mu \in [0, \epsilon]$, we get the result. 
\end{proof}

Next we show the vanishing formula for the case where the spin fiber bundle structure (preserving the boundary) extends to the whole manifold, and the fiberwise operators are invertible for the whole family. 

\begin{prop}[The vanishing formula]\label{vanishing}
We consider the following situation. 
\begin{itemize}
\item Let $(M^{\mathrm{ev}}, \partial M, \pi : \partial M \to Y^{\mathrm{odd}})$ be a compact manifold with fibered boundary,  equipped with pre-spin structures $P'_M$ and $P'_Y$ on $TM$ and $TY$, respectively, and a riemannian metric $g_\pi$ on $T^V\partial M$, for which the fiberwise spin Dirac operator $D_\pi$ is an invertible family. 
\item There exist data $(\pi', X, P'_X, g_{\pi'})$ such that 
\begin{itemize}
\item $X$ is a compact manifold with boundary $\partial X$, with a fixed diffeormorphism $\partial X \simeq Y$. 
We identify $\partial X$ with $Y$. 
\item $\pi' : (M, \partial M) \to (X, \partial X)$ is a fiber bundle structure which preserves the boundary, and $\pi'|_{\partial M} = \pi$. 
Note that the typical fibers of $\pi$ and $\pi'$ are the same. 
\item $P'_X$ is a pre-spin structure on $TX$ which satisfies $P'_X |_Y = P'_Y$. 
\item $g_{\pi'}$ is a riemannian metric on $T^VM$ (the fiberwise tangent bundle of the fiber bundle $\pi'$) satisfying $g_{\pi'} |_{Y} = g_\pi$. 
We denote $D_{\pi'}$ the family of fiberwise spin Dirac operators for $\pi'$. 
\end{itemize}
\end{itemize}
Assume that $D_{\pi'}$ is invertible. 
Then we have
\begin{align*}
\mathrm{Ind}_\Phi(P'_M, P'_Y, g_\pi) = \mathrm{Ind}_e(P'_M, P'_Y, g_\pi)&=0 . 
\end{align*}
\end{prop}

\begin{proof}
The first equality follows from Proposition \ref{equality}. 
Consider the subgroupoid $M \times_{\pi'} M \subset G_e$ and define
$\mathcal{G} := DNC(G_e, M \times_{\pi'}M)|_{M \times [0, 1]}$. 
Denote the closed saturated subset $M_1 := M \times \{0\} \cup\partial M \times [0, 1] \subset M \times [0, 1]$ for this groupoid. 
Note that we have $\mathcal{G}|_{\partial M \times [0, 1]} = \partial M \times_\pi \partial M \times_\pi DNC(TY \rtimes \mathbb{R}^*_+, Y)|_{Y \times [0, 1]}$.
We can also see that the restriction $\mathcal{G}|_{M \times \{0\}}$ is of the form $M \times_{\pi'} M \times_{\pi'} E_X$, where $E_X \to X$ is a vector bundle over $X$. 
In particular, there exists canonical direct sum decomposition of $\mathfrak{A}\mathcal{G}|_{M_1}$ such that one component is $T^VM_1$. 
Choose any riemannian metric $g_\mathfrak{A}$ on $\mathfrak{A}\mathcal{G}$ such that, on $M_1$, the two direct sum components are orthogonal, and $T^VM_1$-component is equal to $g_{\pi'}\cup g_\pi \times [0, 1]$. 
We consider the spin structure on $\mathfrak{A}\mathcal{G}$ defined by the given data and metric $g_\mathfrak{A}$ chosen above, and consider the spin Dirac operator $\mathcal{D} \in \mathrm{Diff}^1(\mathcal{G}; S(\mathfrak{A}\mathcal{G}))$. 

Then, the restriction of $\mathcal{D}$ to $M_1$ has the product form as in (\ref{fredholmness_op1}) and (\ref{fredholmness_op2}). 
Since we are assuming that the fiberwise operator $D_{\pi'}$ is invertible, we see that $\mathcal{D}|_{M_1}$ is invertible. 
So we get the index class 
\[
\mathrm{Ind}_{\mathring{M} \times (0, 1]}(\mathcal{D}) \in K_0(C^*(\mathcal{G}|_{\mathring{M}\times (0, 1]})), 
\]
and we see that $\mathrm{Ind}_e(P'_M, P'_Y, g_\pi) = (ev_1)_* \mathrm{Ind}_{\mathring{M} \times (0, 1]}(\mathcal{D})$. 
However, since $C^*(\mathcal{G}|_{\mathring{M}\times (0, 1]}) = C^*(\mathring{M} \times \mathring{M}) \otimes C_0((0, 1])$ is contractible, its $K$-group is trivial and we get the result. 
\end{proof}

\subsection{The cases of twisted $spin^c$ and signature operators}
The above argument easily generalizes to the cases of twisted $spin^c$-Dirac operators and twisted signature operators, as follows. 
Let $(M^{\mathrm{ev}}, \pi : \partial M \to Y^{\mathrm{odd}})$ be a compact manifold with fibered boundary, and $E \to M$ be a $\mathbb{Z}_2$-graded complex vector bundle. 

\subsubsection{Twisted $spin^c$ Dirac operators}\label{subsubsec_twisted_spinc}
For our conventions on $spin^c$/pre-$spin^c$/differential $spin^c$ structures, see Definition \ref{convention_spinc}. 
In order to define the $\Phi$ and $e$-indices of the $spin^c$-Dirac operator on $M$ twisted by $E$, we need the following data. 
\begin{enumerate}
\item[(D1)] Pre-$spin^c$ structures $P'_M$ and $P'_Y$ on $TM$ and $TY$, respectively. 
\label{twistedD1}
\item[(D2)] A differential $spin^c$ structure $P_\pi$ on $T^V\partial M$, which is compatible with the pre-$spin^c$-structure induced from $P'_M$ and $P'_Y$ (see Remark \ref{rem_spin_str}). 
\label{twistedD2}
\item[(D3)] A hermitian structure on $E|_{\partial M}$ as well as a smooth family of fiberwise unitary connection for the boundary fibration, i.e., a continuous map
\[
\nabla^E_\pi : C^\infty(\partial M; E|_{\partial M}) \to C^\infty(\partial M; E|_{\partial M} \otimes (T^{V}\partial M)^*)
\]
given by a family of unitary connections $\{\nabla_y^E\}_{y \in Y}$ on the vector bundle $E|_{\pi^{-1}(y)} \to \pi^{-1}(y)$ for each $y\in Y$. 
\label{twistedD3}
\item[(D4)] Denote the fiberwise twisted $spin^c$-Dirac operators for $\pi$ by $D^E_\pi = \{D^E_{\pi^{-1}(y)}\}_{y \in Y}$. 
Here $D^E_{\pi^{-1}(y)}$ acts on $ C^\infty(\pi^{-1}(y); S(\pi^{-1}(y)) \hat{\otimes} E)$. 
We assume that $D^E_\pi$ forms an invertible family. 
\label{twistedD4}
\end{enumerate}

Additional data which are needed to define an operator are as follows. 
\begin{enumerate}
\item[(d1)] A differential $spin^c$ structure on $\mathfrak{A}G_\Phi$ ($\mathfrak{A}G_e$) such that  
\begin{itemize}
\item it is compatible with the pre-$spin^c$ structures in (D1). 
\item it has a product structure with respect to the decomposition $\mathfrak{A}G_\Phi|_{\partial M} = T^V\partial M \oplus \pi^*TY \oplus \mathbb{R}$ ($\mathfrak{A}G_e|_{\partial M} = T^V\partial M \oplus \pi^*TY \oplus \mathbb{R}$) at the boundary. 
\item the $T^V\partial M$-component coincides with the one in (D2). 
\end{itemize}
\label{twistedd1}
\item[(d2)] A hermitian structure on $E$ which restricts to the one given in (D3), and a unitary connection $\nabla^E$ which restricts to $\nabla^E_\pi$ in (D3).  
\label{twistedd2}
\end{enumerate}
From these data, we get the twisted $spin^c$-Dirac operators $D_\Phi^{S \hat{\otimes}E} \in \mathrm{Diff}^1(G_\Phi; S(\mathfrak{A}G_\Phi) \hat{\otimes}E)$ and $D_e^{S \hat{\otimes}E} \in \mathrm{Diff}^1(G_e; S(\mathfrak{A}G_e) \hat{\otimes}E)$. 
By the assumption on the invertibility of $D^E_\pi$ in (D4), we get the fredholmness of these operators as in Proposition \ref{fredholmness}, as follows. 
We only explain it in the $\Phi$-case. 
It is enough to see that the restriction to the boundary, $D_\Phi^{S \hat{\otimes}E}|_{\partial M} \in \mathrm{Diff}^1(G_\Phi|_{\partial M}; S(\mathfrak{A}G_\Phi)|_{\partial M} \hat{\otimes}E|_{\partial M})$, is invertible. 
This operator is given by a family $\{D^E_y\}_{y \in Y}$ parametrized by $Y$, and
each $D^E_y$ is the $spin^c$-Dirac operator twisted by $E$ on the groupoid $G_\Phi|_{\pi^{-1}(y)} = \pi^{-1}(y) \times \pi^{-1}(y) \times T_yY \times \mathbb{R}$, in the sense of Example \ref{ex_twisted_spinc}. 
We have the isomorphism $S(\mathfrak{A}G_\Phi|_{\pi^{-1}(y)})\simeq S(\pi^{-1}(y)) \hat{\otimes} S(T_yY \times \mathbb{R})$ by the assumption (d1). 
Define
\[
r : \pi^{-1}(y)  \times T_yY \times \mathbb{R} \to \pi^{-1}(y), \ 
(x, \xi, t) \mapsto x. 
\]
By the construction of twisted $spin^c$-Dirac operators on groupoids explained in Example \ref{ex_twisted_spinc}, we introduce the connection on the hermitian vector bundle $r^*(E|_{\pi^{-1}(y)}) \to \pi^{-1}(y)\times T_y Y \times \mathbb{R}$ as the pullback $r^*\nabla^E$ of the connection $\nabla^E$ on $E$. 
By the assumption in (d2), it coincides with the pullback $r^*\nabla^E_\pi$. 
Thus the operator $D^E_y$ is written as 
\begin{align}\label{twisted_spinc_op}
D^E_y &: C_c^\infty(\pi^{-1}(y) \times T_y Y \times \mathbb{R}; (S(\pi^{-1}(y)) \hat{\otimes}E|_{\pi^{-1}(y)})\hat{\otimes}S(T_yY \times \mathbb{R}))  \notag \\
&\to C_c^\infty(\pi^{-1}(y) \times T_y Y \times \mathbb{R}; (S(\pi^{-1}(y)) \hat{\otimes}E|_{\pi^{-1}(y)})\hat{\otimes}S(T_yY \times \mathbb{R})) \notag \\
D^E_y &= D^E_{\pi^{-1}(y)} \hat{\otimes} 1 + 1 \hat{\otimes} D_{T_yY \times \mathbb{R}}. 
\end{align}
Here the operator $D_{T_yY \times \mathbb{R}}$ is the $spin^c$-Dirac operator on the Euclidean space $T_yY\times \mathbb{R}$ defined by (d1). 
The operators $D^E_\pi \hat{\otimes} 1$ and $ 1 \hat{\otimes} D_{T_yY \times \mathbb{R}}$ anticommute, and by the assumption (D4), we get the invertibility of the family $D_\Phi^{S \hat{\otimes}E}|_{\partial M} =\{D^E_y\}_{y \in Y}$. 

So we get their indices
\[
\mathrm{Ind}_{\mathring{M}}(D_\Phi^{S \hat{\otimes}E}), \ \mathrm{Ind}_{\mathring{M}}(D_e^{S \hat{\otimes}E}) \in \mathbb{Z}. 
\]

These indices depend only on the data (D1)$\sim$(D4) and do not depend on the additional data (d1) or (d2), as in Proposition \ref{stability}. 
So we can define the $\Phi$ and $e$-indices of the data $(P'_M, P'_Y, P_\pi, E, \nabla^E_\pi)$
 as follows. 

\begin{defn}\label{def_index_perturbation}

Given data $(P'_M, P'_Y, P_\pi, E, \nabla^E_\pi)$ as in (D1)$\sim$(D4) above, we choose additional data (d1) and (d2) arbitrarily and define
\begin{align*}
\mathrm{Ind}_\Phi(P'_M,P'_Y, P_\pi, E, \nabla^E_\pi)&:=\mathrm{Ind}_{\mathring{M}}(D_\Phi^{S \hat{\otimes}E}) ,\\
\mathrm{Ind}_e(P'_M, P'_Y, P_\pi, E, \nabla^E_\pi)&:=\mathrm{Ind}_{\mathring{M}}(D_e^{S \hat{\otimes}E}).
\end{align*}
These do not depend on the choice in (d1) or (d2). 
\end{defn}

We can show the equality $\mathrm{Ind}_\Phi(P'_M,P'_Y, P_\pi, E, \nabla^E_\pi) = \mathrm{Ind}_e(P'_M, P'_Y, P_\pi, E, \nabla^E_\pi)$ as in Proposition \ref{equality}. 
The gluing formula as in Proposition \ref{gluing} and the vanishing property as in Proposition \ref{vanishing} hold analogously. 

\subsubsection{Twisted signature operators}\label{subsubsec_twistedsign}
For twisted signature operators, we need the following data. 
Let $(M^{\mathrm{ev}}, \pi : \partial M \to Y^{\mathrm{odd}})$ be a compact manifold with fibered boundaries, where both $M$ and $Y$ are oriented. 
We call such $(M, \pi : \partial M \to Y)$ {\it oriented} ; note that this includes the orientation on $Y$. 
These orientations induce an orientation on $T^V\partial M$.
The data needed to define the $\Phi$ and $e$-signature are as follows. 
\begin{enumerate}
\item[(D1)] A riemanian metric $g_\pi$ on $T^V\partial M$. 
\label{signD1}
\item[(D2)] A hermitian structure on $E|_{\partial M}$ as well as a smooth family of fiberwise unitary connection for the boundary fibration, i.e., a continuous map
\[
\nabla^E_\pi : C^\infty(\partial M; E|_{\partial M}) \to C^\infty(\partial M; E|_{\partial M} \otimes (T^{V}\partial M)^*)
\]
given by a smooth family of unitary connections $\{\nabla_y^E\}_{y \in Y}$ on $E|_{\pi^{-1}(y)} \to \pi^{-1}(y)$ for each $y\in Y$. 
\label{signD2}
\item[(D3)] Denote the fiberwise twisted signature operators for $\pi$ by $D^{\mathrm{sign}, E}_\pi = \{D^{\mathrm{sign}, E}_{\pi^{-1}(y)}\}_{y \in Y}$. 
Here $D^{\mathrm{sign}, E}_{\pi^{-1}(y)}$ acts on $ C^\infty(\pi^{-1}(y); \wedge_{\mathbb{C}} T^*(\pi^{-1}(y)) \hat{\otimes} E)$. 
We assume that $D^{\mathrm{sign}, E}_\pi$ forms an invertible family. 
\label{signD3}
\end{enumerate}

Additional data which are needed to define an operator are as follows. 
\begin{enumerate}

\item[(d1)] A smooth riemannian metric $g_\Phi$ for $\mathfrak{A}G_\Phi \simeq T^\Phi M \to M$, whose restriction to $\mathfrak{A}G_\Phi |_{\partial M} = T^V\partial M \oplus \pi^*TY \oplus \mathbb{R}$ can be written as
\[
g_\Phi|_{\partial M} = g_\pi \oplus \pi^*g_{TY \oplus \mathbb{R}}  ,
\] 
where $g_{TY \oplus \mathbb{R}}$ is some riemannian metric on $TY \oplus \mathbb{R} \to Y$. 

\item[(d1)$'$] A smooth riemannian metric $g_e$ for $\mathfrak{A}G_e \simeq T^e M \to M$, whose restriction to $\mathfrak{A}G_e |_{\partial M} = T^V\partial M \oplus \pi^*TY \oplus \mathbb{R}$ can be written as
\[
g_e|_{\partial M} = g_\pi \oplus \pi^*g_{TY \oplus \mathbb{R}}  ,
\] 
where $g_{TY \oplus \mathbb{R}}$ is some riemannian metric on $TY \oplus \mathbb{R} \to Y$. 

\label{signd1}
\item[(d2)] A hermitian structure on $E$ which restricts to the one given in (D3), and a unitary connection $\nabla^E$ which restricts to $\nabla^E_\pi$ in (D3). 
\label{signd2}
\end{enumerate}
From these data, we get the twisted signature operators $D_\Phi^{\mathrm{sign}, E} \in \mathrm{Diff}^1(G_\Phi; \wedge_{\mathbb{C}} \mathfrak{A}^*G_\Phi \hat{\otimes}E)$ and $D_e^{\mathrm{sign}, E} \in \mathrm{Diff}^1(G_e; \wedge_{\mathbb{C}}\mathfrak{A}^*G_e\hat{\otimes}E)$. 
By the assumption on the invertibility of $D^{\mathrm{sign}, E}_\pi$ in (D4), we get the fredholmness of these operators as in the subsubsection \ref{subsubsec_twisted_spinc}. 
Note that in this case, the boundary operator $D_\Phi^{\mathrm{sign}, E}|_{\partial M}= \{D^{\mathrm{sign}, E}_y\}_{y \in Y}$ is given by
\begin{align*}
D^{\mathrm{sign}, E}_y &: C_c^\infty(\pi^{-1}(y) \times T_y Y \times \mathbb{R}; (\wedge_{\mathbb{C}} T^*(\pi^{-1}(y)) \hat{\otimes}E|_{\pi^{-1}(y)})\hat{\otimes}\wedge_{\mathbb{C}} (T_yY \times \mathbb{R})^*) \\
&\to  C_c^\infty(\pi^{-1}(y) \times T_y Y \times \mathbb{R}; (\wedge_{\mathbb{C}} T^*(\pi^{-1}(y)) \hat{\otimes}E|_{\pi^{-1}(y)})\hat{\otimes}\wedge_{\mathbb{C}} (T_yY \times \mathbb{R})^*) \\
D_y &= D^{\mathrm{sign}, E}_{\pi^{-1}(y)} \hat{\otimes} 1 + 1 \hat{\otimes} D^{\mathrm{sign}}_{T_yY \times \mathbb{R}}. 
\end{align*}
Here $D^{\mathrm{sign}}_{T_yY \times \mathbb{R}}$ is the Euclidean signature operator defined by the metric $g_{TY \oplus \mathbb{R}}$ in (d1).  
 
So we get their indices
\[
\mathrm{Ind}_{\mathring{M}}(D_\Phi^{\mathrm{sign}, E}), \ \mathrm{Ind}_{\mathring{M}}(D_e^{\mathrm{sign}, E}) \in \mathbb{Z}. 
\]

These indices depend only on the data (D1)$\sim$(D3) and do not depend on the additional data (d1), (d1)$'$ or (d2), as in Proposition \ref{stability}. 
So we can define the $\Phi$ and $e$-indices of the data $( g_\pi, E, \nabla^E_\pi)$
 as follows. 

\begin{defn}
Given a compact oriented manifold with boundary $(M^{\mathrm{ev}}, \pi : \partial M \to Y^{\mathrm{odd}})$ with data $(g_\pi, E, \nabla^E_\pi)$ as in (D1)$\sim$(D3) above, we choose additional data (d1), (d1)$'$ and (d2) arbitrarily and define
\begin{align*}
\mathrm{Sign}_\Phi(M, g_\pi, E, \nabla^E_\pi)&:=\mathrm{Ind}_{\mathring{M}}(D_\Phi^{\mathrm{sign}, E}) ,\\
\mathrm{Sign}_e(M, g_\pi, E, \nabla^E_\pi)&:=\mathrm{Ind}_{\mathring{M}}(D_e^{\mathrm{sign}, E}).
\end{align*}
This does not depend on the choice in (d1), (d1)$'$ or (d2).  
\end{defn}

We can show the equality $\mathrm{Sign}_\Phi( M, g_\pi, E, \nabla^E_\pi) = \mathrm{Sign}_e(M,  g_\pi, E, \nabla^E_\pi)$ as in Proposition \ref{equality}. 
The gluing formula as in Proposition \ref{gluing} and the vanishing property as in Proposition \ref{vanishing} holds analogously. 

\section{Indices of geometric operators on manifolds with fibered boundaries : the case with fiberwise invertible perturbations}\label{sec_index_perturbation}
Next we consider operators with fiberwise invertible perturbations on the boundary family. 
The idea is that, if we are given a pair $(P'_M, P'_Y, g_\pi)$ as in Definition \ref{def_index_spin}, even if we do not have the invertibility of fiberwise Dirac operator $D_\pi$ for the boundary fibration, if we are given an invertible perturbation $\tilde{D}_\pi$ by a lower order family, then we can construct fully elliptic $\Phi$/$e$-operators $\tilde{D}$ such that
\begin{itemize}
\item on the interior $\mathring{M}$, $\tilde{D}$ differs from $D_\Phi$ ($D_e$) by an operator of order $0$. 
\item the boundary operator of $\tilde{D}$ is given by $\tilde{D}_\pi \hat{\otimes} 1 + 1 \hat{\otimes}D_{TY \times \mathbb{R}}$ ($\tilde{D}_\pi \hat{\otimes} 1 + 1 \hat{\otimes}D_{TY \rtimes \mathbb{R}}$). 
\end{itemize} 

We would like to define the index of this operator as the index of the pair $(P'_M, P'_Y, g_\pi)$ defined by the fiberwise invertible perturbation $\tilde{D}_\pi$. 
This index has a simpler description, as below. 

\subsection{The general situation}
In this subsection, we recall the well-known general construction of indices, defined using invertible perturbations of an operator on a closed saturated subset for a Lie groupoid. 
We start with a general setting as follows. 
\begin{itemize}
\item Let $M$ be a compact manifold possibly with boundaries and corners. 
\item Let $G \rightrightarrows M$ be a Lie groupoid. 
\item Let $V \subset M$ be a closed saturated subset for $G$. 
\item Let $(\sigma_M, \tilde{F}_V) \in C(\mathfrak{S}^*G_M) \oplus_V \overline{\Psi_c^0(G_V)}$ be an invertible element. 
\end{itemize}
Denote the full symbol algebra $\Sigma^{M \setminus V}(G) :=C(\mathfrak{S}^*G_M) \oplus_V \overline{\Psi_c^0(G_V)}$ as in subsubsection \ref{subsubsec_ellipticity}. 
We consider the following exact sequence. 
\[
0 \to C^*(G_{M \setminus V}) \to \overline{\Psi^0_c(G)}  \stackrel{\sigma_{f, V}}{\to} \Sigma^{M \setminus V}(G) \to 0. 
\]
We denote the connecting element for this short exact sequence as $\mathrm{ind}^{M \setminus V}(G) \in KK^1(\Sigma^{M \setminus V}(G) , C^*(G_{M \setminus V}))$. 
The element $(\sigma_M, \tilde{F}_V)$ gives a class in $K_1(\Sigma^{M \setminus V}(G))$, so defines the index class as
\[
\mathrm{Ind}_{M \setminus V}((\sigma_M, \tilde{F}_V)) := [(\sigma_M, \tilde{F}_V)] \otimes \mathrm{ind}^{M \setminus V}(G) \in K_0(C^*(G_{M \setminus V})).  
\]

This index can be generalized to the case where we are given a path from the operator $F_V$ to an invertible operator.  
The settings are as follows. 
\begin{itemize}
\item Let $(\sigma_M, F_V) \in \Sigma^{M \setminus V}(G)$ be an element such that $\sigma_M \in C(\mathfrak{S}^*G_M)$ is invertible. 
\item Let $F_{V \times [0, 1]} = \{F_{V \times \{t\}}\}_{t \in [0, 1]}$ be a continuous path of operators $F_{V \times \{t\}} \in \overline{\Psi_c^0(G_V)}$ parametrized by $t \in [0, 1]$ such that
\begin{itemize}
\item $F_{V\times \{0\}} = F_V$. 
\item $F_{V\times\{t\}}$ is elliptic for all $t \in [0, 1]$. 
\item $F_{V \times \{1\}}$ is invertible. 
\end{itemize}
We call such a path ``an {\it invertible perturbation} for $F_V$''.
\end{itemize}

\begin{rem}\label{rem_cpt_perturbation}
In the following, we often work in the situation where we are given 
\begin{itemize}
\item An element $F_V \in \overline{\Psi^0_c(G_V)}$ for which $\sigma(F_V) \in C(\mathfrak{S}^*G_V)$ is invertible, and
\item An invertible element $\tilde{F}_V \in \overline{\Psi^0_c(G_V)}$ which satisfies $\tilde{F}_V - F_V \in C^*(G_V)$. 
\end{itemize}
In this case, we have a canonical choice, up to homotopy, of path $F_{V \times [0, 1]} = \{F_{V \times \{t\}}\}_{t \in [0, 1]}$ such that $F_{V\times \{0\}}=F_V$ and $F_{V\times \{1\}}=\tilde{F}_V$. 
Namely, we choose any such continuous path which satisfies $F_{V \times \{t\}} - F_V \in C^*(G_V)$ for all $t\in [0,1]$. 
With the abuse of notation we also call such $\tilde{F}_V$ ``an {\it invertible perturbation} for $F_V$'' and actually consider such path of operators. 
\end{rem}

From the data above, we define $\mathrm{Ind}_{M \setminus V}(\sigma_M, F_{V \times [0, 1]}) \in K_0(C^*(G_{M \setminus V}))$ as follows. 
Denote
\begin{equation}
\begin{aligned}\label{mappingcone}
M_1 &:= M \cup_{V \times \{0\}} V \times [0, 1] \mbox{ and } \mathring{M}_1 := M \cup_{V \times \{0\}} V \times [0, 1) \\
G_1 &:= G \cup_{V \times \{0\}} G_V \times [0, 1] \rightrightarrows M_1 \mbox{ and }\mathring{G}_1 := G_1|_{\mathring{M}_1}. 
\end{aligned}
\end{equation}
Although $M_1$ is not a manifold, $G_1$ is a longitudinally smooth groupoid, so we abuse the notations such as $C^*(G_1) :=C^*(G) \oplus_{V \times \{0\}}C^*(G_V \times [0, 1]) $. 

We have the following exact sequence. 
\[
0 \to C^*(\mathring{G}_1) \to \overline{\Psi^0_c(G_1)} 
\xrightarrow{\sigma_{f, V \times \{1\}}} \Sigma^{\mathring{M}_1}(G_1)\to 0. 
\]
We denote the associated connecting element as $\mathrm{ind}^{\mathring{M}_1}(G_1) \in KK^1(\Sigma^{\mathring{M}_1}(G_1), C^*(\mathring{G}_1))$. 

Consider the canonical $*$-homomorphism
\[
\sigma'_{f, V \times \{1\}} : \Sigma^{M_1 \setminus V \times [0, 1]}(G_1) \to \Sigma^{\mathring{M}_1}(G_1), 
\]
defined by applying the symbol map on $G_{V} \times [0, 1]$. 
Here we have $\Sigma^{M_1 \setminus V \times [0, 1]}(G_1) = C(\mathfrak{S}^*G_1) \oplus_{V \times [0, 1]} \overline{\Psi_c^0(G_V \times [0, 1])} \simeq C(\mathfrak{S}^*G) \oplus_{V \times \{0\}} \overline{\Psi_c^0(G_V \times [0, 1])}$. 

Given a pair $(\sigma_M, F_{V\times[0,1]}) \in \Sigma^{M_1 \setminus V \times [0, 1]}(G_1)$ as above, by the conditions, the element $\sigma'_{f, V \times \{1\}}(\sigma_M, F_{V\times[0,1]}) \in \Sigma^{\mathring{M}_1}(G_1)$ is invertible. 
So we get a class
\[
[\sigma'_{f, V \times \{1\}}(\sigma_M, F_{V\times[0,1]})]\in K_1(\Sigma^{\mathring{M}_1}(G_1)). 
\]
Furthermore the inclusion $i : C^*(G_{M \setminus V}) \to C^*(\mathring{G}_1)$ gives a $KK$-equivalence $[i] \in KK(C^*(G_{M \setminus V}) , C^*(\mathring{G}_1))$. 
So we define the index as follows. 
\begin{defn}
\[
\mathrm{Ind}_{M \setminus V}(\sigma_M, F_{V \times [0, 1]}) =
[\sigma'_{f, V \times \{1\}}(\sigma_M, F_{V\times[0,1]})] \otimes \mathrm{ind}^{\mathring{M}_1}(G_1) \otimes [i]^{-1} \in K_0(C^*(G_{M \setminus V})). 
\]
\end{defn}

Next we prove the following relative formula for this index. 
Recall that, if we are given two invertible perturbations $F_{V \times [0, 1]}^i$, $i = 0, 1$ for an operator $F_V$, they define the difference class in $K_1(C^*(G_V))$ as follows. 
Let $F'_{V \times [0, 1]} = \{F'_{V \times \{t\}}\}_{t \in [0, 1]}$ be a continuous path of operators $F'_{V \times \{t\}} \in \overline{\Psi_c^0(G_V)}$ defined by
\begin{equation}\label{diffpath}
F'_{V\times \{t\}} = \begin{cases}
F^0_{V \times \{1-2t\}} & \mbox{ if }t \in [0, 0.5] \\
F^1_{V \times \{2t-1\}} & \mbox{ if }t \in [0.5, 1]. 
\end{cases}
\end{equation}
i.e., first follow the path $F_{V \times [0. 1]}^0$ in the reversed direction and next follow $F_{V \times [0,1]}^1$. 
This operator satisfies $F'_{V \times [0, 1]} \in \overline{\Psi_c^0(G_V \times [0, 1])}$. 
Consider the exact sequence
\[
0 \to C^*(G_V \times (0, 1))\to \overline{\Psi_c^0(G_V \times [0, 1])} \xrightarrow{\sigma_{f, V \times \{0, 1\}}}
 \Sigma^{V \times (0, 1)}(G_V \times [0,1]) \to 0. 
\]
By assumption $\sigma_{f, V \times \{0, 1\}}(F'_{V \times [0, 1]})$ is invertible. 
Thus we get the index class
\[
\mathrm{Ind}_{V \times (0, 1)}(F'_{V \times [0, 1]}) \in K_0(C^*(G_V \times (0, 1))) \simeq K_1(C^*(G_V)). 
\]
We define this class as the difference class of the invertible perturbations $F^0_{V \times [0, 1]}$ and $F^1_{V \times [0, 1]}$: 
\[
[F^1_{V \times [0, 1]} - F^0_{V \times [0, 1]}]:=\mathrm{Ind}_{V\times (0,1)}(F'_{V \times [0, 1]}) \in K_1(C^*(G_V)). 
\]

\begin{rem}
As in subsection \ref{subsec_invertibleperturbation},
we denote by $\tilde{\mathcal{I}}(F_V)$ the set of invertible perturbations for the operator $F_V$. 
This set has the obvious homotopy relation, and we denote $\mathcal{I}(F_V)$ the set of homotopy classes of elements in $ \tilde{\mathcal{I}}(F_V)$. 
We can show that $\mathcal{I}(F_V)$ is nonempty if and only if $\mathrm{Ind}(F_V) = 0 \in K_0(C^*(G_V))$. 
The above definition of the difference class induces the affine space structure on $\mathcal{I}(F_V)$ modeled on $K_1(C^*(G_V))$. 
\end{rem}

\begin{rem}\label{rem_cpt_diff_class}
In Remark \ref{rem_cpt_perturbation}, we explained that an operator $\tilde{F}_V$ such that $\tilde{F}_V - F_V \in C^*(G_V)$ can be regarded as an invertible perturbation of $F_V$. 
Assume we have two invertible perturbations $\tilde{F}^0_V$ and $\tilde{F}^1_V$ of $F_V$ in this sense. 
Then the difference class defined above between these perturbations, which we denote by $[\tilde{F}^1_V - \tilde{F}^0_V]$, can be described as follows. 
We take any path $F'_{V \times [0, 1]} \in \overline{\Psi_c^0(G_V \times [0, 1])}$ satisfying $F'_{V\times \{i\}} = \tilde{F}_V^i$ for $i = 0, 1$ and $F'_{V \times \{t\}} - \tilde{F}_V^0 \in C^*(G_V)$ for all $t \in [0,1]$. 
Then we get
\[
[\tilde{F}^1_V - \tilde{F}^0_V] = \mathrm{Ind}_{V\times (0,1)}(F'_{V \times [0, 1]}) \in K_1(C^*(G_V)). 
\]
Two different choices of such path are homotopic, and the one which is obtained by the construction in (\ref{diffpath}) is one of such choices. 
\end{rem}

\begin{prop}[The general relative formula]\label{gen_relative}
Let $G \rightrightarrows M$ be a longitudinally smooth Lie groupoid over a compact manifold $M$, and $V \subset M$ be a closed saturated subset. 
Let $(\sigma_M, F_V) \in \Sigma^{M \setminus V}$ be an element such that $\sigma_M \in C(\mathfrak{S}^*G_M)$ is invertible. 
Suppose we are given two invertible perturbations $F_{V \times [0, 1]}^i$, $i = 0, 1$ for $F|_V$. 
Then we have
\[
\mathrm{Ind}_{M \setminus V}(\sigma_M, F_{V \times [0, 1]}^1) - \mathrm{Ind}_{M \setminus V}(\sigma_M, F_{V \times [0, 1]}^0 )
= [F^1_{V \times [0, 1]} - F^0_{V \times [0, 1]}] \otimes \partial^{M \setminus V}(G) \in K_0(C^*(G_{M \setminus V})). 
\]
Here, the element $\partial^{M\setminus V}(G) \in KK^1(C^*(G_V), C^*(G_{M \setminus V}))$ is the connecting element of the short exact sequence, 
\[
0 \to C^*(G_{M \setminus V}) \to C^*(G) \to C^*(G_V) \to 0. 
\]
as defined in subsection \ref{subsubsec_ellipticity}. 
In particular, the element $\mathrm{Ind}_{M \setminus V}(\sigma_M, F_{V \times [0, 1]}) \in K_0(C^*(G_{M\setminus V}))$ only depends on the class of $ F_{V \times [0, 1]}$ in $\mathcal{I}(F_V)$. 
\end{prop}

\begin{proof}
We use the notations 
\begin{itemize}
\item $M_t := M \cup_{V \times \{0\}} V \times [0, t]$ and $\mathring{M}_t := M \cup_{V \times \{0\}} V \times [0, t)$, 
\item $G_t := G \cup_{V \times \{0\}} G_V \times [0, t] \rightrightarrows M_t$ and $\mathring{G}_t := G_t|_{\mathring{M}_t}$
\item The inclusion which gives a $KK$-equivalence
$
i_t : C^*(G_{M\setminus V}) \to C^*(\mathring{G}_t)
$. 
\end{itemize}
for $t > 0$. 
Consider the path of operators $F'_{V \times [0, 1]}$ defined in (\ref{diffpath}). 
We change the parameters $t\in [0, 1] \mapsto t+1 \in [1, 2]$ and consider it as an operator $F'_{V \times [1, 2]} \in \overline{\Psi_c^0(G_V \times [1, 2])}$. 
By construction the union $F^0_{V \times [0, 1]} \cup_{V \times \{1\}} F'_{V \times [1, 2]}$ is a continuous path of elliptic operators and defines an element in $\overline{\Psi_c^0(G_V \times [0, 2])}$. 
Denote $\sigma_{M_2} = \sigma_M \cup \sigma_{V \times [0, 2]}(F^0_{V \times [0, 1]} \cup F'_{V \times [1, 2]}) \in C(\mathfrak{S}^*(G_2))$. 
The pair $(\sigma_{M_2}, F^0_{V \times \{1\}} \sqcup F'_{V \times \{2\}})$ gives an element in $\Sigma^{M_2 \setminus V \times \{1, 2\}}(G_2)$. 
By construction, this is invertible. 
Thus we get the index class
\begin{align}\label{relative_eq1}
\mathrm{Ind}_{M_2 \setminus V \times \{1, 2\}}(\sigma_{M_2}, F^0_{V \times \{1\}} \sqcup F'_{V \times \{2\}}) &\in K_0(C^*(G_2|_{M_2 \setminus V \times \{1, 2\}})) \notag \\
&= K_0(C^*(\mathring{G}_1)) \oplus K_0(C^*(G_V \times (1, 2))). 
\end{align}
We denote by $p_{\mathring{M}_1}$ and $p_{V \times (1,2)}$ the projections to the first and second factor on the group appearing in the right hand side of the above equation (\ref{relative_eq1}). 
By construction, we have
\begin{align*}
p_{\mathring{M}_1} (\mathrm{Ind}_{M_2 \setminus V \times \{1, 2\}}(\sigma_{M_2}, F^0_{V \times \{1\}} \sqcup F'_{V \times \{2\}})) &= 
\mathrm{Ind}_{M \setminus V}(\sigma_M, F_{V \times [0, 1]}^0) \otimes [i_1] \in K_0(C^*(\mathring{G}_1)). \\
p_{V \times (1,2)}(\mathrm{Ind}_{M_2 \setminus V \times \{1, 2\}}(\sigma_{M_2}, F^0_{V \times \{1\}} \sqcup F'_{V \times \{2\}})) &=
[F^1_{V \times [0, 1]} - F^0_{V \times [0, 1]}] \in K_0(C^*(G_V \times (1,2))). 
\end{align*}
Moreover, we see that under the inclusion
\[
j: C^*(\mathring{G}_1) \oplus C^*(G_V \times (1,2)) \to 
C^*(\mathring{G}_2), 
\]
we have
\begin{align*}
\mathrm{Ind}_{M \setminus V}(\sigma_M, F_{V \times [0, 1]}^1) \otimes [i_2] &=\mathrm{Ind}_{\mathring{M}_2}(\sigma_M, F'_{V \times \{2\}}) \\
&= \mathrm{Ind}_{M_2 \setminus V \times \{1, 2\}}(\sigma_{M_2}, F^0_{V \times \{1\}} \sqcup F'_{V \times \{2\}})
\otimes [j] . 
\end{align*}
So we have
\begin{align*}
\mathrm{Ind}_{M \setminus V}(\sigma_M, F_{V \times [0, 1]}^1) &= \mathrm{Ind}_{M_2 \setminus V \times \{1, 2\}}(\sigma_{M_2}, F^0_{V \times \{1\}} \sqcup F'_{V \times \{2\}})
\otimes [j] \otimes [i_2]^{-1} \\
&= \mathrm{Ind}_{M \setminus V}(\sigma_M, F_{V \times [0, 1]}^0)  +  [F^1_{V \times [0, 1]} - F^0_{V \times [0, 1]}]\otimes [j] \otimes [i_2]^{-1}.  
\end{align*}
Thus it is enough to show that $\partial^{M \setminus V}(G) = [j] \otimes [i_2]^{-1} \in KK(C^*(G_V \times ( 1,2)), C^*(G_{M \setminus V}))$. 
But this is well-known, since in general the element $[\partial_{\phi}] \in KK^1(B, J)$ associated to an extention of $C^*$-algebra
\[
0 \to J \to A \xrightarrow{\phi} B \to 0,
\]
where $B$ is nuclear, is given by $[\partial_\phi] =[j]\otimes [i]^{-1}$, where
$
j :  B \otimes C_0((0, 1)) \to A \oplus_{\phi} (B \otimes C_0([0, 1))) 
$
is the inclusion and $[i] \in KK(J, A \oplus_{\phi} (B \otimes C_0([0, 1))))$ is the $KK$-equivalence (see \cite{Bla}). 

If we have two invertible perturbations $F^i_{V \times [0, 1]}$ ($i = 0, 1$) which define the same class in $\mathcal{I}(F_V)$, the difference class $[F^1_{V \times [0, 1]} - F^0_{V \times [0, 1]}]$ vanishes, so we have $\mathrm{Ind}_{M \setminus V}(\sigma_M, F_{V \times [0, 1]}^1) = \mathrm{Ind}_{M \setminus V}(\sigma_M, F_{V \times [0, 1]}^0 )$
\end{proof}

\begin{rem}\label{bdd_trans}
If we deal with a positive order elliptic operator $D\in \Psi^*_c(G)$, we consider the bounded transform $\psi(D) := D / (1 + D^*D)^{-1/2} \in \overline{\Psi^0_c(G)}$ and do the same arguments. 
More generally we can deal with an elliptic operator $F' \in \overline{\Psi^0_c(G; E_0, E_1)}$ acting between two vector bundles in an essentially the same way. 
Namely, we consider the vector bundle $E_0 \oplus E_1$ with the $\mathbb{Z}_2$ grading so that $E_0$ is the even part and $E_1$ is the odd part. 
Consider the odd self-adjoint operator on $E_0\oplus E_1$ defined as
\[
F := 
\begin{pmatrix}
0 & F' \\
F'^* & 0
\end{pmatrix}. 
\]
We construct the $C^*$-algebras with coefficients in $E_0 \oplus E_1$, such as $C^*(G; E_0\oplus E_1)$ and $\Sigma^{M \setminus V}(G; E_0 \oplus E_1)$, with the $\mathbb{Z}_2$-grading associated to the grading on $E_0\oplus E_1$. 
These $C^*$-algebras are Morita equivalent to the corresponding algebras with trivial coefficients. 
Associated to an elliptic symbol and an invertible perturbation as before, we get an odd self-adjoint invertible element
\footnote{Given a unital graded $C^*$-algebra $A$ and an odd self-adjoint unitary operator $u \in A$, we can construct a unital graded $*$-homomorphism $\mathbb{C}l_1 \to A$ by sending the generator $\epsilon$ to $u$. 
The $K_1$ class of this element, $[u] \in K_1(A) \simeq KK(\mathbb{C}l_1, A)$ is defined to be the class given by this graded $*$-homomorphism. 
The space of odd self-adjoint invertible elements on $A$ retracts to the space of odd self-adjoint unitary elements, so an odd self-adjoint invertible element also defines the class in $K_1(A)$ this way. }
$(\sigma_M, F_{V\times[0,1]}) \in \Sigma^{M \setminus V}(G_1; E_0\oplus E_1)$ (c.f. \cite[Definition 1.3]{CS}). 
Thus we get the class $[(\sigma_M, F_{V\times[0,1]})]\in K_1(\Sigma^{M_1 \setminus V \times [0, 1]}(G_1; E_0\oplus E_1))= K_1(\Sigma^{M_1 \setminus V \times [0, 1]}(G_1))$ and the same argument applies. 
\end{rem}
\subsection{The connecting elements of $G_\Phi$ and $G_e$}
In this preparatory subsection, we show that the connecting elements of the exact sequences
\begin{align*}
0 \to C^*(G_\Phi|_{\mathring{M}}) \to C^*(G_\Phi) &\to C^*(G_\Phi|_{\partial M}) \to 0 \\
0 \to C^*(G_e|_{\mathring{M}}) \to C^*(G_e) &\to C^*(G_e|_{\partial M}) \to 0. 
\end{align*}
correspond to the Poincar\'e dual to the element $[\underline{\mathbb{C}}_Y] \in KK(\mathbb{C}, C(Y))$. 
This result is used in the proof of relative formulas for $\Phi$ and $e$-indices in Proposition \ref{relative_spinc} and Proposition \ref{relative_signature}. 

\begin{lem}[The connecting elements of $G_\Phi$ and $G_e$]\label{connecting}
Consider the exact sequences
\begin{align}
0 \to C^*(G_\Phi|_{\mathring{M}}) \to C^*(G_\Phi) &\to C^*(G_\Phi|_{\partial M}) \to 0 \label{connecting_eq1} \\
0 \to C^*(G_e|_{\mathring{M}}) \to C^*(G_e) &\to C^*(G_e|_{\partial M}) \to 0. \label{connecting_eq2}
\end{align}
Denote by $\partial^{\mathring{M}}(G_\Phi) \in KK^1(C^*(G_\Phi|_{\partial M}) , C^*(G_\Phi|_{\mathring{M}}))$ and
$\partial^{\mathring{M}}(G_e) \in KK^1(C^*(G_e|_{\partial M}) , C^*(G_e|_{\mathring{M}}))$ 
the connecting elements associated to the above exact sequences. 
Denote by $\underline{\mathbb{C}}_Y : \mathbb{C} \to C(Y)$ the canonical $*$-homomorphism. 
Denote by $[\sigma_Y] \in KK(C^*(TY), \mathbb{C})$ the element which is Poincar\'e dual to $[\underline{\mathbb{C}}_Y] \in KK(\mathbb{\mathbb{C}}, C(Y))$. 
\begin{enumerate}
\item[($\Phi$)]
Under the Morita equivalence between $G_\Phi|_{\partial M} $ and $ TY \times \mathbb{R}$, the element $\partial^{\mathring{M}}(G_\Phi) \in KK^1(C^*(G_\Phi|_{\partial M}) , C^*(G_\Phi|_{\mathring{M}})) \simeq KK(C^*(TY), \mathbb{C})$ identifies with the element $[\sigma_Y]$. 
\item[($e$)]
Under the Morita equivalence between $G_\Phi|_{\partial M}$ and $ TY \rtimes \mathbb{R}_+^*$ and the $KK^1$-equivalence between $C^*(TY \rtimes \mathbb{R}_+^*)$ and $ C^*(TY) $ given by the Connes-Thom isomorphism, the element $\partial^{\mathring{M}}(G_e) \in KK^1(C^*(G_e|_{\partial M}) , C^*(G_e|_{\mathring{M}})) \simeq KK(C^*(TY), \mathbb{C})$ identifies with the element $[\sigma_Y]$. 
\end{enumerate}
\end{lem}

\begin{proof}
First we show that it is enough to consider the case $M = Y \times \mathbb{R}_+$ and $\pi : \partial M = Y \times \{0\} \to Y$ is the identity map. 
Indeed, fixing a tubular neighborhood $U \simeq \partial M \times \mathbb{R}_+$ of $\partial M$ in $M$, 
$U \subset M$ is a transverse submanifold of both $G_e$ and $G_\Phi$. 
Thus the connecting element of (\ref{connecting_eq1}) is equal to the connecting element of the exact sequence
\begin{equation}\label{connecting_eq3}
0 \to C^*(G_\Phi|_{\mathring{U}}) \to C^*(G_\Phi|_U) \to C^*(G_\Phi|_{\partial M}) \to 0, 
\end{equation}
and analogously for (\ref{connecting_eq2}). 
We consider the manifold with fibered boundary $(Y \times \mathbb{R}_+, id_Y: Y \times \{0\}  \to Y)$ and denote its $\Phi$ and $e$-groupoids as $\tilde{G}_\Phi$ and $\tilde{G}_e$. 
Denote $\tilde{\pi} := \pi \times id_{\mathbb{R}_+} : U \simeq \partial M \times \mathbb{R}_+ \to Y \times \mathbb{R}_+$. 
We easily see that $G_\Phi|_U \simeq ^*\!\!\tilde{\pi}^*\tilde{G}_\Phi$ and $G_e|_U \simeq ^*\!\!\tilde{\pi}^*\tilde{G}_e$. 
Under this Morita equivalence, the connecting element of the exact sequence (\ref{connecting_eq3}) is equal to the connecting element of the corresponding exact sequence of $\tilde{G}_\Phi$, and analogously for the $e$-case. 
Thus it is enough to consider the case of manifold with fibered boundary $(Y \times \mathbb{R}_+, id_Y : Y \times \{0\} \to Y)$, as stated. 
From now on, in this proof we denote the $b$, $\Phi$ and $e$ groupoids of $(Y \times \mathbb{R}_+, id_Y : Y \times \{0\} \to Y)$ by $G_b$, $G_\Phi$ and $G_e$, respectively. 

From now on in this proof, we use symbols such as $\underline{\mathbb{R}}^*_+$ or $\hat{\mathbb{R}}_+$ in order to distinguish various $\mathbb{R}$-factors which have different roles. 
First we show in the $e$-case. 
Recall the definition of $G_e$ given in subsubsection \ref{subsubsec_groupoid_approach}; $G_e$ is defined by the spherical blowup construction of the pair groupoid $Y \times Y \times \mathbb{R}_+ \times \mathbb{R}_+ \rightrightarrows Y \times \mathbb{R}_+$ by the subgroupoid $Y \times \{(0, 0)\} \rightrightarrows Y \times \{0\}$, i.e., 
$G_e = SBlup_{r, s}(Y \times Y \times \mathbb{R}_+ \times \mathbb{R}_+, Y \times \{(0, 0)\})$. 
Recall the Connes tangent groupoid (\cite{Con}) of $Y$, $\mathbb{T}Y = TY \times \{0\} \sqcup Y \times Y \times \mathbb{R}_+^* \rightrightarrows Y \times \mathbb{R}_+$. 
Its Lie groupoid structure is described as $\mathbb{T}Y = DNC_+(Y \times Y, Y)$ (cf. \cite[section 5.3.2]{DS}). 
We easily see that $G_e \simeq \mathbb{T}Y \rtimes \underline{\mathbb{R}}_+^*$, where $\underline{\mathbb{R}}_+^* \ni \lambda$ acts on $TY$ as multiplication by $\lambda$ and on $\mathbb{R}_+^*$ as multiplication by $1/\lambda$ (cf. \cite[section 5.3.3]{DS}. Apply the construction there for $G=Y \times Y$). 
Thus we have a commutative diagram in $KK$-theory, 
\[
\xymatrix{
0 \ar[r]  &C^*(G_e|_{Y \times \mathbb{R}_+^*}) \ar[r] \ar@{-}[d]^{\simeq}&C^*(G_e) \ar[r]\ar@{-}[d]^{\simeq}  &C^*(G_e|_{Y \times \{0\}}) \ar[r]\ar@{-}[d]^{\simeq}&0
\\
0 \ar[r]  &C^*(Y \times Y \times \mathbb{R}_+^* \rtimes \underline{\mathbb{R}}_+^*) \ar[r] \ar@{-}[d]^{KK^1}&C^*(\mathbb{T}Y \rtimes \underline{\mathbb{R}}_+^*) \ar[r]\ar@{-}[d]^{KK^1}  &C^*(TY \rtimes \underline{\mathbb{R}}_+^*) \ar[r]\ar@{-}[d]^{KK^1}&0
\\
0 \ar[r]  &C^*(Y \times Y \times \mathbb{R}^*_+) \ar[r]  &C^*(\mathbb{T}Y) \ar[r] &C^*(TY) \ar[r] &0,
}
\]
where the rows are exact and the vertical maps between the middle and the bottom rows are $KK^1$-equivalences given by the Connes-Thom isomorphism. 
The connecting element of the bottom row is equal to $[\sigma_Y] \otimes [Bott]\in KK^1(C^*(TY), C_0(\mathbb{R}_+^*))$ (see \cite[Lemma 6 in Chapter 2, Section 5]{Con}), so we get the result. 

Next we prove the $\Phi$-case. 
Recall that $G_\Phi$ is defined as $G_\Phi = SBlup_{r, s}(G_b, Y)$, where we regard $Y\rightrightarrows Y$ as a subgroupoid $Y \times \{0\} \times \{0\} \rightrightarrows Y \times \{0\}$ of the groupoid $G_b = Y \times Y \times \mathbb{R} \times \{0\} \sqcup Y\times Y \times \mathbb{R}^*_+ \times \mathbb{R}^*_+ \rightrightarrows Y \times \mathbb{R}_+$.  
We define $\partial G_b := G_b|_{Y \times \{0\}} = Y \times Y \times \mathbb{R}$ and $\mathring{G}_b = G_b|_{Y \times \mathbb{R}_+^*} = Y\times Y \times \mathbb{R}^*_+ \times \mathbb{R}^*_+ $. 
Noting that $Y \times \{0\}$ is a closed saturated submanifold for $G_b$, we have a commutative diagram
\[
\xymatrix{
&0 \ar[d] &0 \ar[d] &0 \ar[d] & \\
0 \ar[r] & C^*(\mathring{G}_b \times {\mathbb{R}}_+^*) \ar[r] \ar[d] &C^*(\widetilde{DNC}_+(G_b, Y)) \ar[r]\ar[d] &C^*(\mathring{N}_Y^{G_b}) \simeq (C^*(TY \times \underline{\mathbb{R}}) \otimes C_0(\hat{\mathbb{R}}_+^*)) \ar[r] \ar[d] &0 \\
0 \ar[r] & C^*(G_b \times {\mathbb{R}}_+^*) \ar[r] \ar[d] &C^*({DNC}_+(G_b, Y))  \ar[r]\ar[d] &C^*({N}_Y^{G_b}) \simeq (C^*(TY \times \underline{\mathbb{R}}) \otimes C_0(\hat{\mathbb{R}}_+))\ar[r] \ar[d] &0 \\
0 \ar[r] & C^*(\partial G_b \times {\mathbb{R}}_+^*) \ar[r] \ar[d] & C^*({DNC}_+(\partial G_b, Y)) \ar[r]\ar[d] & C^*({N}_Y^{\partial G_b}) \simeq C^*(TY \times \underline{\mathbb{R}})\ar[r] \ar[d] &0 \\
&0 & 0 & 0 & 
}
\]
where the rows and columns are exact. 
We easily see that $DNC_+(\partial G_b, Y)\simeq \mathbb{T}Y \times \underline{\mathbb{R}} \rightrightarrows Y \times \mathbb{R}_+$, where the factor $\underline{\mathbb{R}}$ does not act on the base.
Thus the connecting element of the bottom row is equal to $[\sigma_Y]\otimes_{\mathbb{C}} id_{\underline{\mathbb{R}}} \otimes_{\mathbb{C}} [Bott] \in KK^1(C^*(TY \times \underline{\mathbb{R}}), C^*(Y \times Y  \times \underline{\mathbb{R}}\times \mathbb{R}_+^*))$. 
The connecting element of the right column is equal to $id_{C^*(TY\times\underline{\mathbb{R}})}\otimes_{\mathbb{C}} [\widehat{Bott}] \in KK^1(C^*(TY \times \underline{\mathbb{R}}), C^*(TY \times \underline{\mathbb{R}}) \otimes C_0(\hat{\mathbb{R}}_+^*))$. 
The connecting element of the left column is equal to $[\underline{Bott}]^{-1} \otimes_{\mathbb{C}} id_{\mathbb{R}_+^*}\in KK^1(C^*(\partial G_b \times {\mathbb{R}}_+^*), C^*(\mathring{G}_b \times {\mathbb{R}}_+^*)) = KK^1(C^*(\underline{\mathbb{R}})\otimes C_0(\mathbb{R}_+^*), C_0(\mathbb{R}_+^*))$ (this well-known fact is a special case of the $e$-case above). 

On the other hand, recalling that $SBlup_{r, s}$ is defined as the quotient by the $\mathbb{R}_+^*$-action on $\widetilde{DNC}_+$ (see subsubsection \ref{dnc_blup}), we have a commutative diagram in $KK$-theory, 
\[
\xymatrix{
0 \ar[r]  &C^*(G_\Phi|_{Y \times \mathbb{R}_+^*}) \ar[r] \ar@{-}[d]^{KK^1}&C^*(G_\Phi) \ar[r]\ar@{-}[d]^{KK^1}  &C^*(G_\Phi|_{Y \times \{0\}}) \ar[r]\ar@{-}[d]^{KK^1}&0
\\
0 \ar[r]  &C^*(\mathring{G}_b \times {\mathbb{R}}_+^*) \ar[r]  &C^*(\widetilde{DNC}_+(G_b, Y)) \ar[r] &C^*(\mathring{N}_Y^{G_b}) \ar[r] &0,
}
\]
where the rows are exact and vertical arrows are $KK^1$-equivalences by the composition of the Connes-Thom isomorphism and the Morita equivalence between the crossed product and the quotient. 
Combining these, we get the result. 
\end{proof}

\subsection{The definitions and relative formulas for the $\Phi$ and $e$-indices}
We apply this general construction to our settings. 

\subsubsection{Twisted $spin^c$-Dirac operators}
Here we explain the case for twisted $spin^c$-Dirac operator. 
First we give a fundamental remark on the space of $\mathbb{C}l_1$-invertible perturbations of geometric operators. 

\begin{rem}\label{rem_twisted_spinc}
Let $X$ be a closed manifold equipped with a pre-$spin^c$ structure, 
and $E \to X$ be a $\mathbb{Z}_2$-graded complex vector bundle. 
In order to define the twisted $spin^c$-Dirac operator $D^E$, we have to specify a differential $spin^c$-structure, a hermitian metric on $E$ and a unitary connection on $E$. 
However, since the space of these choices is contractible, the sets of homotopy classes of $\mathbb{C}l_1$-invertible perturbations, $\mathcal{I}(D^E)$, for two different choices are {\it canonically} isomorphic.  

An analogous remark applies when we consider a family of twisted $spin^c$-Dirac operators. 
Suppose we are given a fiber bundle $\pi : N \to Y$ whose typical fiber is a closed manifold,  a pre-$spin^c$-structure $P'_\pi$ for $\pi$, and a complex vector bundle $E\to N$. 
Choosing the additional data to define a twisted $spin^c$-Dirac operator $D^E_\pi$, we define
\[
\mathcal{I}(P'_\pi, E):= \mathcal{I}(D_\pi^E). 
\]
These sets for two different choices of additional data are canonically isomorphic. 

For a family of signature operators analogous remark applies. 
Suppose a fiber bundle $\pi : N \to Y$ whose typical fiber is a closed manifold, is oriented, and
let $E \to N$ be a $\mathbb{Z}_2$-graded hermitian vector bundle. 
We define 
\[
\mathcal{I}^{\mathrm{sign}}(\pi, E):= \mathcal{I}(D^{\mathrm{sign}, E}_\pi)
\]
where $D_\pi^{\mathrm{sign}, E}$ is the twisted signature operator defined by any fiberwise metric, hermitian metric on $E$ and unitary connection on $E$. 
 
\end{rem}

Let $(M, \pi : \partial M^{\mathrm{ev}} \to Y^{\mathrm{odd}}, E \to M)$ be a compact manifold with fibered boundaries, equipped with a complex vector bundle. 
The data needed to define the index are the following. 
\begin{enumerate}
\item[(D1)] Pre-$spin^c$ structures $P'_M$ and $P'_Y$ on $TM$ and $TY$, respectively. 
These induce a pre-$spin^c$ structure on $T^V\partial M$, denoted by $P'_\pi$. 
\item[(D2)]A  homotopy class of $\mathbb{C}l_1$-invertible perturbation $Q_\pi \in \mathcal{I}(P'_\pi, E)$. 
\end{enumerate}

The additional data needed to construct operators are as follows. 
\begin{enumerate}
\item[(d1)] A differential $spin^c$ structure on $\mathfrak{A}G_\Phi$ ($\mathfrak{A}G_e$) such that  
\begin{itemize}
\item it is compatible with the pre-$spin^c$ structures in (D1). 
\item it has a product structure with respect to the decomposition $\mathfrak{A}G_\Phi|_{\partial M} = T^V\partial M \oplus \pi^*TY \oplus \mathbb{R}$ ($\mathfrak{A}G_e|_{\partial M} = T^V\partial M \oplus \pi^*TY \oplus \mathbb{R}$) at the boundary. 
\end{itemize}
\label{twistedd1}
\item[(d2)] A hermitian structure on $E$ and a unitary connection $\nabla^E$.  
Denote the fiberwise twisted $spin^c$-Dirac opeartor $D_\pi^E$. 
\label{twistedd2}
\item[(d3)] A family of operators $\tilde{D}^E_\pi \in \tilde{\mathcal{I}}_{\mathrm{sm}}(D^E_\pi)$ which is a representative of the class $Q_{\pi} \in \mathcal{I}(P'_\pi, E)$ in (D2). 
\end{enumerate}

Let us denote by $D^{S\hat{\otimes}E }_\Phi$ and $D^{S\hat{\otimes}E }_e$ the twisted spin Dirac operators constructed from the above data, respectively. 
Recall that, under the assumption (d1) above, the restriction of $D^{S\hat{\otimes}E}_\Phi$ to $G_\Phi|_{\partial M}$ is given by a family $\{D^E_y\}_{y \in Y}$ parametrized by $Y$, of the form
\begin{align*}
D^E_y &: C_c^\infty(\pi^{-1}(y) \times T_y Y \times \mathbb{R}; (S(\pi^{-1}(y)) \hat{\otimes}E|_{\pi^{-1}(y)})\hat{\otimes}S(T_yY \times \mathbb{R})) \\
&\to C_c^\infty(\pi^{-1}(y) \times T_y Y \times \mathbb{R}; (S(\pi^{-1}(y)) \hat{\otimes}E|_{\pi^{-1}(y)})\hat{\otimes}S(T_yY \times \mathbb{R})) \\
D^E_y &= D^E_{\pi^{-1}(y)} \hat{\otimes} 1 + 1 \hat{\otimes} D_{T_yY \times \mathbb{R}}. 
\end{align*}
as in (\ref{twisted_spinc_op}). 

Using the $\mathbb{C}l_1$-invertible perturbation $\tilde{D}^E_\pi=\{\tilde{D}^E_{\pi^{-1}(y)}\}_{y \in Y}$ in the data (d3) above, we define an operator $\tilde{D}^{S\hat{\otimes}E}_{\Phi, \partial M} \in \Psi^1_c(G_\Phi|_{\partial M}; S|_{\partial M} \times E|_{\partial M})$ as a family $\{\tilde{D}^E_y\}_{y \in Y}$, given by
\[
\tilde{D}^E_y := \tilde{D}^E_{\pi^{-1}(y)} \hat{\otimes} 1 + 1 \hat{\otimes} D_{T_yY \times \mathbb{R}}. 
\]
This gives an invertible operator on $G_\Phi|_{\partial M}$, which satisfies $D^{S \hat{\otimes} E}_\Phi|_{\partial M} - \tilde{D}^{S\hat{\otimes}E}_{\Phi, \partial M} \in\Psi^0_c(G_\Phi|_{\partial M}; S|_{\partial M})$. 
It is easy to see that the class $[\tilde{D}^{S\hat{\otimes}E}_{\Phi, \partial M}] \in \mathcal{I}(D^{S\hat{\otimes}E}_\Phi|_{\partial M})$ does not depend on the choice of the explicit operator $\tilde{D}^E_\pi$ representing the class $Q_\pi \in \mathcal{I}(P_\pi, E)$. 
Applying the bounded transform, it defines a class
\begin{equation}\label{fullsymb_perturbation_spinc}
[(\sigma_M(D^{S\hat{\otimes}E}_\Phi), \psi(\tilde{D}^{S\hat{\otimes}E}_{\Phi, \partial M}))] \in K_1(\Sigma^{\mathring{M}}(G_\Phi)). 
\end{equation}
This class only depends on the data (D1) and (D2), and does not depend on the additional data (d1), (d2), or (d3). 

In the $e$-case, $D_e|_{\partial M}$ also has the product form as in (\ref{fredholmness_op2}), so we define an invertible operator $\tilde{D}^{S\hat{\otimes}E}_{e, \partial M}$ in an analogous way. 

\begin{defn}\label{def_index_perturbation_spinc}
Given the data (D1) and (D2) as above, choose any additional data (d1), (d2) and (d3). 
We define the $\Phi$ and $e$-indices, defined by the boundary fiberwise invertible perturbations as 
\begin{align*}
\mathrm{Ind}_\Phi(P'_M,P'_Y,E, Q_\pi) &:= \mathrm{Ind}_{\mathring{M}}(\sigma_M(D^{S\hat{\otimes}E}_\Phi), \psi(\tilde{D}^{S\hat{\otimes}E}_{\Phi, \partial M})) \in K_0(C^*(G_\Phi|_{\mathring{M}})) \simeq \mathbb{Z}, \\
\mathrm{Ind}_e(P'_M,P'_Y,E, Q_\pi)&:= \mathrm{Ind}_{\mathring{M}}(\sigma_M(D^{S\hat{\otimes}E}_\Phi), \psi(\tilde{D}^{S\hat{\otimes}E}_{\Phi, \partial M})) \in K_0(C^*(G_e|_{\mathring{M}})) \simeq \mathbb{Z}. 
\end{align*} 
This number only depends on the data (D1) and (D2), and does not depend on the additional data (d1), (d2), or (d3).
\end{defn}

For this index we also have the equality 
\begin{equation}\label{equality_perturbation}
\mathrm{Ind}_\Phi(P'_M,P'_Y, E, Q_\pi) = \mathrm{Ind}_e(P'_M,P'_Y,E, Q_\pi)
\end{equation}
as in Proposition \ref{equality}. 
Also, similar results to Proposition \ref{gluing} and \ref{vanishing} hold in this case. 
For the vanishing formula, the assumption becomes that ``the fibration extends to the whole manifold and the fiberwise invertible perturbation extends to the whole family''. 
We give the precise formulation of these properties, as follows. 

\begin{prop}[The gluing formula]\label{gluing_perturbation}

We consider the following situations. 
\begin{itemize}
\item Let $(M^0, \pi^0 : \partial M^0 \to Y^0, E^0 \to M^0)$ and $(M^1, \pi^1 : \partial M^1 \to Y^1, E^1 \to M^1)$ be manifolds with fibered boundaries equipped with complex vector bundles. 
\item Assume we are given data $(P'_{M^i}, P'_{Y^i},  Q_{\pi^i})$ satisfying the conditions (D1) and (D2) above for each $i = 0, 1$.  
\item Assume that on some components of $\partial M^0$ and $-\partial M^1$, we are given isomorphisms of the data $(\pi^i, P'_{M^i}, P'_{Y^i}, E^i, Q_{\pi^i})$ restricted there.  
\item Let us denote $(M, \pi' : \partial M \to Y')$ the manifold with fibered boundary obtained by identifying isomorphic boundary components.  
This manifold is equipped with data $( P'_{M}, P'_{Y'},  E, Q_{\pi'})$ induced from those on $M^0$ and $M^1$. 
\end{itemize}
Then, we have
\begin{align*}
\mathrm{Ind}_\Phi(P'_M,P'_{Y'}, E, Q_{\pi'})&=
\mathrm{Ind}_e(P'_M,P'_{Y'}, E, Q_{\pi'})\\
 &= \mathrm{Ind}_e(P'_{M^0},P'_{Y^0}, E^0, Q_{\pi^0}) + \mathrm{Ind}_e(P'_{M^1},P'_{Y^1},E^1, Q_{\pi^1}) 
\end{align*}
\end{prop}

\begin{prop}[The vanishing formula]\label{vanishing_perturbation}

We consider the following situations. 
\begin{itemize}
\item Let $(M^{\mathrm{ev}}, \partial M, \pi : \partial M \to Y^{\mathrm{odd}})$ be a compact manifold with fibered boundary, equipped with a complex vector bundle $E \to M$. 
\item Let $(P'_M, P'_{Y},  Q_{\pi})$ be data satisfying the conditions in (D1) and (D2). 
\item Assume that there exists data $(\pi', X, P'_X, Q_{\pi'})$ such that 
\begin{itemize}
\item A compact manifold $X$ with boundary $\partial X$, with a fixed diffeormorphism $\partial X \simeq Y$. 
We identify $\partial X$ with $Y$. 
\item A fiber bundle structure $\pi' : (M, \partial M) \to (X, \partial X)$ which preserves the boundary, and $\pi'|_{\partial M} = \pi$. 
Note that the typical fibers of $\pi$ and $\pi'$ are the same. 
\item A pre-$spin^c$ structure $P'_X$ on $TX$ which restricts to $P'_Y$. 
\item Assume that the induced pre-$spin^c$-structure induced on $T^VM$ restricts to $P'_\pi$ at the boundary. 
\item An element $Q_{\pi'}$ in $\mathcal{I}(P'_{\pi'}, E)$ which satisfies $Q_{\pi'}|_{\partial M} = Q_\pi$. 
\end{itemize}
\end{itemize}

Then we have
\begin{align*}
\mathrm{Ind}_\Phi(P'_M,P'_Y, E, Q_\pi)=
\mathrm{Ind}_e(P'_M,P'_Y, E, Q_\pi)&=0 . 
\end{align*}
\end{prop}

Next we show the relative formula for such indices. 
Recall that, for a family $D_\pi$ of $\mathbb{Z}_2$-graded self-adjoint operators parametrized by $Y$, if we are given two elements $Q_\pi^0$ and $Q_\pi^1$ in $\mathcal{I}(D_\pi)$, their difference class $[Q_\pi^1 - Q_\pi^0]$ is defined in $K^{-1}(Y)$. 
\begin{prop}[The relative formula]\label{relative_spinc}
Let $(M, \partial M, \pi)$ as before, and $Q_\pi^0$ and $Q_\pi^1$ be two elements in $\mathcal{I}(P'_\pi ,E)$. 
Then we have
\begin{align*}
&\mathrm{Ind}_\Phi(P'_M, P'_Y,E, Q_\pi^1) - \mathrm{Ind}_\Phi(P'_M, P'_Y, E, Q_\pi^0) \\
&=\mathrm{Ind}_e(P'_M, P'_Y,  E, Q_\pi^1) - \mathrm{Ind}_e(P'_M, P'_Y, E, Q_\pi^0) =
\langle [Q_\pi^1 - Q_\pi^0], [D_Y]\rangle. 
\end{align*}
Here $[D_Y] \in K_1(Y)$ is the class of $spin^c$-Dirac operator on $Y$ defined by the data (D1) and (D2), and $\langle {\cdot},{\cdot}\rangle : K^1(Y) \otimes K_1(Y) \to \mathbb{Z}$ denotes the index pairing. 
\end{prop}
\begin{proof}
The first equality follows from (\ref{equality_perturbation}). 
Choose any additional data (d1), (d2) and (d3) to define the operator $D^E_e$. 
For each $i = 0, 1$, choose any representative $\tilde{D}^{E, i}_\pi\in \tilde{\mathcal{I}}_{\mathrm{sm}}(P'_\pi, E)$ for the class $Q_\pi^i \in \mathcal{I}(P'_\pi, E)$. 
By the general relative formula, Proposition \ref{gen_relative}, it is enough to show that the difference class of the invertible perturbations $D^{E,i}_{\partial M}:=\tilde{D}^{E, i}_\pi \hat{\otimes} 1 + 1 \hat{\otimes}D_{TY \rtimes \mathbb{R}^*_+}$ for $i = 0, 1$, defined in $K_1(C^*(G_e|_{\partial M})) (\simeq K_1(C^*(TY \rtimes \mathbb{R}^*_+)) \simeq K_0(Y))$, maps to $\langle [Q_\pi^1-Q_\pi^0], [D_Y]\rangle$ under the boundary map $\partial^{\mathring{M}} (G_e): K_1(C^*(G_e|_{\partial M})) \to K_0(C^*(G_e|_{\mathring{M}}))$. 

Consider the operator $\mathcal{D}$ on the groupoid $G_e|_{\partial M} \times [0, 1]_t \rightrightarrows \partial M \times [0, 1]_t$ defined by the family
\[
\mathcal{D}|_{\partial M \times \{t\}} := (t\tilde{D}^{E, 0}_{\pi} + (1-t)\tilde{D}^{E, 1}_{\pi}) \hat{\otimes} 1 + 1 \hat{\otimes} D_{TY \rtimes \mathbb{R}^*_+}. 
\]
The restriction to $\partial M \times \{0, 1\}$ is invertible. 
Thus we get the index class of $\mathcal{D}$ in $K_0(C^*(G_e|_{\partial M}\times (0, 1)))$, and by Remark \ref{rem_cpt_diff_class} (and also Remark \ref{bdd_trans}), the difference class of invertible perturbations $D^{E,i}_{\partial M}$ coincides with this class:
\[
[D^{E,1}_{\partial M} - D^{E,0}_{\partial M}] = \mathrm{Ind}_{\partial M \times (0, 1)}(\mathcal{D}) \in K_0(C^*(G_e|_{\partial M}\times (0, 1))). 
\]

Denote the Connes-Thom element $[th] \in KK^1(C^*(G_e|_{\partial M}), C^*(\partial M \times_\pi \partial M \times_\pi TY ))$. 
Consider the following self-adjoint ungraded operator on the groupoid $\partial M \times_\pi \partial M \times_\pi TY \times (0, 1) \rightrightarrows \partial M \times (0, 1)$:
\begin{equation}\label{eq_tildeD}
\mathcal{D}'|_{\partial M \times \{t\}} := (t\tilde{D}^{E, 0}_{\pi} + (1-t)\tilde{D}^{E, 1}_{\pi}) \hat{\otimes} 1 + 1 \hat{\otimes} D_{TY}, 
\end{equation}
for each $t\in[0, 1]$. 
This operator defines a class $\mathrm{Ind}_{\partial M \times (0, 1)}(\mathcal{D}') \in K_1(C^*(\partial M \times_\pi \partial M \times_\pi TY \times (0, 1)))$, and it satisfies
\begin{equation}\label{relative_eq2}
\mathrm{Ind}_{\partial M \times (0, 1)}(\mathcal{D}) \otimes [th] = \mathrm{Ind}_{\partial M \times (0, 1)}(\mathcal{D}'). 
\end{equation}
We consider the following elements. 
\begin{itemize}
\item $[\tilde{D}^{E, 1}_{\pi}- \tilde{D}^{E, 0}_{\pi}] \in K_0(C^*((\partial M \times_\pi \partial M) \times (0, 1)))\simeq K^1(Y)$. 
\item $[D_{Y}] \in K_1(Y)$. 
\item $\mathrm{Ind}^Y(D_{TY}) \in KK^1(C(Y), C^*(TY))$ represented by the ungraded Kasparov $C(Y)$-$C^*(TY)$ bimodule 
$(C^*(TY ; S(TY)), \mathrm{multi}, \psi(D_{TY}))$,
where multi is the multiplication by $C(Y)$. 
This is an ungraded version of (\ref{eq_ind_y}).
\item $m \in KK(C(Y)\otimes C^*(TY), C^*(TY))$ represented by the Kasparov $C(Y)\otimes C^*(TY)$-$C^*(TY)$ bimodule
$(C^*(TY), \mathrm{multi} \otimes id_{C^*(TY)}, 0)$. 
\item $[\sigma_Y] \in KK(C^*(TY), \mathbb{C})$. 
\end{itemize}
The element $m \otimes_{C^*(TY)}\sigma_Y \in KK(C(Y)\otimes C^*(TY), \mathbb{C})$ is the element which gives the Poincar\'e duality between $C^*(TY)$ and $C(Y)$. 
Also we have $\mathrm{Ind}(D_{TY})\otimes_{C^*(TY)}m  = \mathrm{Ind}^Y(D_{TY})$, since the element $\mathrm{Ind}(D_{TY}) \in KK^1(\mathbb{C}, C^*(TY))$ is represented by the Kasparov module $(C^*(TY; S(TY)), 1, D_{TY})$. 
Since $\mathrm{Ind}(D_{TY})$ is the Poincar\'e dual to $[D_Y]$, we have 
\[
[D_Y] = \mathrm{Ind}(D_{TY})\otimes_{C^*(TY)}m \otimes_{C^*(TY)}\sigma_Y = \mathrm{Ind}^{Y}(D_{TY})\otimes_{C^*(TY)}\sigma_Y.   
\]

Next, we show the following equality. 
\begin{equation}\label{eq_kas_prod}
\mathrm{Ind}_{\partial M \times (0, 1)}(\mathcal{D}') = [\tilde{D}^{E, 1}_{\pi} - \tilde{D}^{E, 0}_{\pi}] \otimes_{C(Y)}\mathrm{Ind}^Y(D_{TY}) \in KK(\mathbb{C}, C^*(TY)). 
\end{equation}
Let $c > 0$ be a positive number such that $\mathrm{Spec}(\tilde{D}^{E, i}_\pi)\cap [-c, c]$ is empty for $i = 0, 1$. 
Choose an odd continuous function $\psi'\in C([-\infty, \infty])$ such that $\psi'\equiv 1$ on $[c, \infty]$ and $\psi' \equiv -1$ on $[-\infty, -c]$. 
We see that $\psi'(\tilde{D}^{E, i}_\pi)$ and $\psi'(\tilde{D}^{E, i}_\pi\hat{\otimes} 1 + 1 \hat{\otimes} D_{TY})$ are self-adjoint unitaries for $i = 0, 1$. 
Using this, the classes appearing in (\ref{eq_kas_prod}) are represented by the Kasparov modules, 
\begin{align*}
[\tilde{D}^{E, 1}_{\pi} - \tilde{D}^{E, 0}_{\pi}] &= [(C^*(\partial M \times_\pi \partial M; S(T^V \partial M)) \otimes C_0(0, 1), 1, \{\psi'(t\tilde{D}^{E, 0}_{\pi} + (1-t)\tilde{D}^{E, 1}_{\pi})\}_t)] \\
\mathrm{Ind}_{\partial M \times (0, 1)}(\mathcal{D}') &= [( C^*(\partial M \times_\pi \partial M \times_\pi TY; S(T^VM) \hat{\otimes}_{C(Y)}S(TY))\otimes C_0(0, 1) , 1, \psi'(\mathcal{D}'))], 
\end{align*}
where the first one is graded and the second one is ungraded. 
We have $C^*(\partial M \times_\pi \partial M; S(T^V \partial M))\otimes_{C(Y)} C^*(TY ; S(TY))= C^*(\partial M \times_\pi \partial M \times_\pi TY; S(T^VM) \otimes_{C(Y)}S(TY)) $. 
Since the above operators commute with the multiplication by elements in $C(Y)$, the computation of this Kasparov product is the family version of the product over $\mathbb{C}$ (more precisely, it is the product in $\mathcal{R}KK(Y; \cdot, \cdot)$; see the paragraph preceeding Proposition \ref{prop_asymp_var} below). 
Here operators satisfy the relation (\ref{eq_tildeD}), by the same argument as in \cite[Section 10.7 and 10.8]{HR}, we get the equality (\ref{eq_kas_prod}). 

By Lemma \ref{connecting}, we know that the connecting element $\partial^{\mathring{M}}(G_e) \in KK^1(G_e|_{\partial M}, G_e|_{\mathring{M}})$ satisfies
\[
[\sigma_Y] = [th]^{-1}\otimes \partial^{\mathring{M}}(G_e)  \in KK(C^*(TY), \mathbb{C}), 
\]
Thus we have 
\begin{align*}
[D^{E,1}_{\partial M} - D^{E,0}_{\partial M}]\otimes_{C^*(TY)} \partial^{\mathring{M}}(G_e)
 &= \mathrm{Ind}_{\partial M \times (0, 1)}(\mathcal{D}) \otimes_{C^*(TY)} \partial^{\mathring{M}}(G_e) \\
&=\mathrm{Ind}_{\partial M \times (0, 1)}(\mathcal{D}') \otimes_{C^*(TY)} [th]^{-1} \otimes[th] \otimes [\sigma_Y] \\
&=[\tilde{D}^{E, 1}_{\pi} - \tilde{D}^{E, 0}_{\pi}] \otimes_{C(Y)}\mathrm{Ind}^Y(D_{TY})\otimes_{C^*(TY)} [\sigma_Y]\\
&=[\tilde{D}^{E, 1}_{\pi} - \tilde{D}^{E, 0}_{\pi}] \otimes_{C(Y)} [D_Y] \\
&=\langle [\tilde{D}^{E, 1}_{\pi} - \tilde{D}^{E, 0}_{\pi}], [D_Y]\rangle. 
\end{align*}
So we get the result. 
\end{proof}

\subsubsection{Twisted signature operators}
Here we explain in the case of twisted signature operators. 
The argument is parallel to that in the case for twisted $spin^c$-Dirac operators. 
Let $(M^{\mathrm{ev}}, \pi : \partial M \to Y^{\mathrm{odd}})$ be a compact oriented manifold with fibered boundaries equipped with a $\mathbb{Z}_2$-graded complex vector bundle $E \to M$. 
Assume we are given an element $Q_\pi \in \mathcal{I}^{\mathrm{sign}}(\pi, E)$. 
We choose additional data as in subsubsection \ref{subsubsec_twistedsign}, and define the twisted signature with respect to the fiberwise invertible perturbation, analogously as in the twisted $spin^c$ Dirac operator case. 

\begin{defn}\label{twisted_sign_perturbation_def}
Given an element $Q_\pi \in \mathcal{I}^{\mathrm{sign}}(\pi, E)$, we define 
\[
\mathrm{Sign}_\Phi(M, E, Q_\pi) \ \mbox{and }
\mathrm{Sign}_e(M, E, Q_\pi) \in \mathbb{Z},
\]
in an analogous way to that in Definition \ref{def_index_perturbation}. 
\end{defn}

We also have the equality of $\Phi$ and $e$-signatures as
\begin{equation}\label{equality_perturbation_signature}
\mathrm{Sign}_\Phi(M, E,  Q_\pi) = \mathrm{Sign}_e( M, E,  Q_\pi). 
\end{equation}
The gluing formula analogous to Proposition \ref{gluing_perturbation}, as well as the vanishing proposition analogous to Proposition \ref{vanishing_perturbation} also holds for this case. 

The relative formula for the twisted signature case is as follows. 
\begin{prop}\label{relative_signature}
Let $(M^{\mathrm{ev}}, \pi : \partial M\to Y^{\mathrm{odd}})$ as before, and $Q_\pi^0$ and $Q_\pi^1$ be two elements in $\mathcal{I}^{\mathrm{sign}}(\pi, E)$. 
Then we have
\begin{align*}
&\mathrm{Sign}_\Phi(M, E, Q_\pi^1) - \mathrm{Sign}_\Phi(M, E,  Q_\pi^0) \\
&=\mathrm{Sign}_e(M, E, Q_\pi^1) - \mathrm{Sign}_e(M,E,  Q_\pi^0) =
2\langle [Q_\pi^1 - Q_\pi^0], [D^{\mathrm{sign}}_Y]\rangle. 
\end{align*}
Here $[D^{\mathrm{sign}}_Y] \in K_1(Y)$ is the class of odd signature operator on $Y$, and $\langle{\cdot},{\cdot}\rangle : K^1(Y) \otimes K_1(Y) \to \mathbb{Z}$ denotes the index pairing. 
\end{prop}

\begin{proof}
The proof is analogous to that for Proposition \ref{relative_spinc}. 
The factor $2$ in the above formula is due to the following observation. 

First of all, recall the definition of odd signature operators acting on odd dimensional manifolds (\cite[Definition and Notation 1]{RW}). 
On an odd dimensional riemannian manifold $Y$, the essentially self-adjoint operator $d + d^*$ acting on $\wedge_{\mathbb{C}}T^*Y$ commutes with the Hodge star $\tau$, so we define the odd signature operator $D_Y^{\mathrm{sign}}$ to be the operator $d + d^*$ restricted to the $+1$-eigenbundle of $\tau$. 
So the total signature operator is isomorphic to the direct sum of two copies of $D_Y^{\mathrm{sign}}$. 
We define odd signature operators for Lie groupoids whose dimensions of $s$-fibers are odd dimensional analogously. 
 
The signature operator $D^{\mathrm{sign}}_{TY \rtimes \mathbb{R}^*_+}$ on the groupoid $TY \rtimes \mathbb{R}^*_+ \rightrightarrows Y$ defines a class $\mathrm{Ind}(D_{TY \rtimes \mathbb{R}^*_+}^{\mathrm{sign}})\in K_0(C^*(TY \rtimes \mathbb{R}^*_+))$. 
The signature operator $D^{\mathrm{sign}}_{TY}$ on the groupoid $TY  \rightrightarrows Y$ defines a class $\mathrm{Ind}(D_{TY}^{\mathrm{sign}})\in K_1(C^*(TY))$. 
We denote the Connes-Thom element $[th] \in KK^1(C^*(TY\rtimes \mathbb{R}^*_+), C^*(TY))$. 
Then these elements are related by
\begin{equation}\label{eq_sign_2}
\mathrm{Ind}(D_{TY \rtimes \mathbb{R}^*_+}^{\mathrm{sign}})\otimes [th] =2 \cdot\mathrm{Ind}(D_{TY}^{\mathrm{sign}})\in K_1(C^*(TY)). 
\end{equation}
Indeed, under Connes-Thom isomorphism $K_0(C^*(TY \rtimes \mathbb{R}^*_+)) \simeq K_0(C^*(TY\times \mathbb{R}))$, the element $\mathrm{Ind}(D_{TY \rtimes \mathbb{R}^*_+}^{\mathrm{sign}})$ maps to $\mathrm{Ind}(D_{TY \times \mathbb{R}}^{\mathrm{sign}})$. 
By the same argument as in the proof of \cite[Lemma 6]{RW}, we see that $2 \cdot \mathrm{Ind}(D_{TY}^{\mathrm{sign}}) \otimes \mathrm{Ind}(D_{\mathbb{R}}^{\mathrm{sign}}) = \mathrm{Ind}(D_{TY \times \mathbb{R}}^{\mathrm{sign}}) \in K_0(C^*(TY \times \mathbb{R}))$. 
Since by definition $\mathrm{Ind}(D_{\mathbb{R}}^{\mathrm{sign}})\in K_1(C^*(\mathbb{R}))$ is equal to the Bott element, (\ref{eq_sign_2}) follows. 

So the factor $2$ appears in the equation corresponding to (\ref{relative_eq2}).  
\end{proof}
\section{The index pairing}\label{sec_K}

In this section, we give a description of the indices defined above, as the index pairing on the $K$-theory ``relative to the boundary pushforward''. 
In the following, we use the following notations. 
\begin{itemize}
\item For a $C^*$-algebra $A$, the symbol $\mathcal{M}(A)$ denotes its multiplier algebra. 
\item For a $C^*$-algebra $A$ and a Hilbert $A$-module $H_A$, the symbols $\mathcal{B}(H_A)$ and $\mathcal{K}(H_A)$ denote the $C^*$-algebras of adjointable operators and compact operators on $H_A$, respectively. 
\item For a Euclidean space $E$, let us denote by $\mathbb{C}l(E)$ the $*$-algebra over $\mathbb{C}$, generated by the elements of $E$ and relations
\[
e = e^* \mbox{ and } e^2 = ||e||^2\cdot 1 \mbox{ for all } e \in E. 
\]
This construction applies to Euclidean vector bundles as well. 
\item Let us denote by $\epsilon \in \mathbb{C}l_1 = \mathbb{C}l(\mathbb{R})$ the element corresponding to the unit vector in $\mathbb{R}$. 
In other words, this element is a generator of $\mathbb{C}l_1$, which is odd, self-adjoint and unitary. 
\end{itemize}

\subsection{The case of $spin^c$-Dirac operators}\label{sec_relativeK_spin}

In this subsection, we consider the case of $spin^c$-Dirac operators. 
First we consider the following setting. 

\begin{itemize}
\item The pair $(M,\pi : \partial M \to Y)$ is a compact manifold with fibered boundary. 
\item The fiber bundle $\pi$ is equipped with a pre-$spin^c$ structure $P'_\pi$. 
\end{itemize}

In order to formulate the index pairing in this setting, we proceed in the following four steps. 
In the following, let $n$ be the dimension of the fiber of $\pi$. 
\begin{enumerate}
\item We define a $C^*$-algebra $\mathcal{A}_\pi$ whose $K$-groups fit in the exact sequence
\[
\cdots   \to K^{*}(M) \xrightarrow{\pi ! \circ i^*} {K}^{*-n}(Y) \to {K}_{*-n}(\mathcal{A}_\pi) \to {K}^{*+1}(M) \xrightarrow{\pi ! \circ i^*}  \cdots. 
\]
(Definition \ref{def_A_pi} and Proposition \ref{A_pi_longexact}). 
The groups $K_*(\mathcal{A}_\pi)$ are regarded as $K$-groups relative to the boundary pushforward. 
\item For a pair $(E, Q_\pi)$ where $E$ is a $\mathbb{Z}_2$-graded complex vector bundle over $M$ and $Q_\pi \in \mathcal{I}(P_\pi, E)$, we show that it naturally defines a class $[(E, Q_\pi)] \in K_{n-1}(\mathcal{A}_\pi)$ (Lemma \ref{lem_K_class}). 
\item Assume $n$ is even. 
For a pair $(P'_M, P'_Y)$ of pre-$spin^c$ structures on $TM$ and $TY$ which satisfies $P'_M|_{\partial M} = \pi^*P'_Y \oplus P'_\pi$, we show that it naturally defines a class $[(P'_M, P'_Y)] \in KK(\mathcal{A}_\pi, \Sigma^{\mathring{M}}(G_\Phi))$ (Definition \ref{def_Khomology_K}). 
\item We show the equality 
\[
\mathrm{Ind}_\Phi(P'_M, P'_Y, E, Q_\pi) = [(E, Q_\pi)] \otimes_{\mathcal{A}_\pi} [(P'_M, P'_Y)] \otimes_{\Sigma^{\mathring{M}}(G_\Phi)} \mathrm{ind}^{\mathring{M}}(G_\Phi) \in \mathbb{Z}. 
\]
(Theorem \ref{main_thm_K}). 
This is the desired index pairing formula. 
\end{enumerate}
The most difficult point of the proof of Theorem \ref{main_thm_K} is to relate the invertible operators $\psi(\tilde{D}_\pi^E)$ and $\psi(\tilde{D}_\pi^E \hat{\otimes}1 + 1 \hat{\otimes} D_{TY \times \mathbb{R}})$, since they are not ``directly related'', for example by a $*$-homomorphism. 
In order to overcome this difficulty, we construct a $C^*$-algebra $\mathcal{D}_\pi$ which ``connects $\mathcal{A}_\pi$ and $\Sigma^{\mathring{M}}(G_\Phi)$'' using an asymptotic morphism giving the $KK$-equivalence between $C_0(TY \oplus \mathbb{R}; \mathbb{C}l(TY \oplus \mathbb{R}))$ and $C(Y)$, 
and construct an invertible element in $\mathcal{D}_\pi$ which, under suitable $*$-homomorphisms, maps to $[(E, [\tilde{D}_\pi^E])] \in K_1(\mathcal{A}_\pi)$ and $[(\sigma(D_\Phi), \psi(\tilde{D}_\pi^E \hat{\otimes}1 + 1 \hat{\otimes} D_{TY \times \mathbb{R}})] \in K_1(\Sigma^{\mathring{M}}(G_\Phi))$. 

Let $N$ be a compact space. 
Let $\pi : N \to Y$ be a fiber bundle whose fibers have closed manifold structure, and $\pi$ is equipped with a pre-$spin^c$-structure. 
Choose any differential $spin^c$-structure representing the given pre-$spin^c$ structure (choosing any other choice, we get canonically $KK$-equivalent $C^*$-algebras below). 
Let $S(T^VN)\to N$ denote the spinor bundle of vertical tangent bundle and $D_\pi$ denote the fiberwise $spin^c$-Dirac operators acting on $S(T^VN)$. 
Let $L_Y^2(N; S(T^VN))$ denote the Hilbert $C(Y)$-module which is obtained by the completion of $C_c^\infty(N; S(T^VN))$ with the natural $C(Y)$-valued inner product. 
Note that $L_Y^2(N; S(T^VN))$ is naturally $\mathbb{Z}_2$-graded if the typical fiber of $\pi$ is even dimensional. 
In this setting we define a $C^*$-algebra $\Psi(D_\pi)$. 
We separate the definition in two cases, depending on the parity of the dimension of the fiber of $\pi$. 
\begin{enumerate}
\item Assume that the typical fiber of $\pi$ is odd dimensional. 
Define $\chi \in C([-\infty, \infty])$ as $\chi(x) :=\frac{1}{2}(1 + x/\sqrt{1 + x^2})$.  
Let $\Psi(D_\pi)$ denote the $C^*$-subalgebra of $\mathcal{B}(L_Y^2(N; S(T^VN)))$ generated by $\{\chi(D_\pi)\}$, $C(N)$ and $\mathcal{K}(L_Y^2(N; S(T^VN)))$. 
\item Assume that the typical fiber of $\pi$ is even dimensional. 
Define the odd function $\psi\in C([-\infty, \infty])$ by $\psi(x) := x/\sqrt{1+x^2}$. 
Let $\Psi(D_\pi)$ denote the $\mathbb{Z}_2$-graded $C^*$-subalgebra of $\mathcal{B}(L_Y^2(N; S(T^VN)))$ generated by $\{\psi(D_\pi)\}$, $C(N)$ and $\mathcal{K}(L_Y^2(N; S(T^VN)))$.
\end{enumerate}

\begin{lem}\label{lemK_ext}
\begin{enumerate}
\item When the typical fiber of $\pi$ is odd dimensional, 
the algebra $\Psi(D_\pi)$ fits into the exact sequence
\[
0 \to \mathcal{K}(L_Y^2(N; S(T^VN))) \to \Psi(D_\pi) \to C(N) \to 0. 
\]
The connecting element of this extension coincides with the class $\pi_! \in KK^1(C(N), C(Y))$. 
\label{lemK_odd}
\item When the typical fiber of $\pi$ is even dimensional,  
the algebra $\Psi(D_\pi)$ fits into the exact sequence of graded $C^*$-algebras
\begin{equation}\label{lemK_exact}
0 \to \mathcal{K}(L_Y^2(N; S(T^VN))) \to \Psi(D_\pi) \to C(N)\otimes \mathbb{C}l_1 \to 0. 
\end{equation}
The connecting element of this extension coincides with the class $\pi_! \in KK(C(N), C(Y))$.
\label{lemK_even} 
\end{enumerate}
\end{lem}

\begin{proof}
We prove the case (\ref{lemK_even}). 
The case (\ref{lemK_odd}) can be proved analogously. 
Denote by $\Gamma$ the groupoid $N \times_\pi N \rightrightarrows N$. 
	Recall that we have a $\mathbb{Z}_2$-graded exact sequence
\[
0 \to C^*(\Gamma; S(\mathfrak{A}\Gamma)) \to \overline{\Psi^0_c(\Gamma; S(\mathfrak{A}\Gamma))} \xrightarrow{\sigma} C(\mathfrak{S}^*\Gamma; \mathrm{End}(S(\mathfrak{A}\Gamma)))\to 0. 
\] 
Of course we have $C^*(\Gamma; S(\mathfrak{A}\Gamma)) = \mathcal{K}(L_Y^2(N; S(T^VN)))$. 
Consider the restriction of the symbol map $\sigma$ to the $C^*$-subalgebra $\Psi(D_\pi) \subset\overline{\Psi^0_c(\Gamma; S(\mathfrak{A}\Gamma))} $. 
Its image is the $C^*$-subalgebra of $C(\mathfrak{S}^*\Gamma; \mathrm{End}(S(\mathfrak{A}\Gamma)))$, generated by $\{\sigma(\psi(D_\pi))\}$ and $C(N)$. 
Since $\sigma(\psi(D_\pi))$ is an odd self-adjoint unitary element commuting with elements in $C(N)$, we get the canonical isomorphism between this $C^*$-algebra and $C(N) \otimes \mathbb{C}l_1$, by mapping $\sigma(\psi(D_\pi))$ to the odd self-adjoint unitary generator $\epsilon \in \mathbb{C}l_1$. 
Thus we get the desired graded exact sequence (\ref{lemK_exact}).

Next we describe the connecting element of (\ref{lemK_exact}). 
We have the following commutative diagram, 
\begin{equation*}
\xymatrix{
0 \ar[r] &\mathcal{K}(H_Y)\ar[r] \ar@{=}[d] &\Psi(D_\pi) \ar[r]\ar@{^{(}->}[d]& C(N)\otimes \mathbb{C}l_1\ar[r]\ar@{^{(}->}[d]_\phi &0
\\
0 \ar[r]  &\mathcal{K}(H_Y)\ar[r]  & \overline{\Psi^0_c(\Gamma; S(\mathfrak{A}\Gamma))} \ar[r] &C(\mathfrak{S}^*\Gamma; \mathrm{End}(S(\mathfrak{A}\Gamma)))\ar[r] &0, 
}
\end{equation*}
where the rows are exact and the inclusion $\phi$ is explained above. 
The bottom row is Morita equivalent to the pseudodifferential extension for the groupoid $\Gamma$, so the connecting element is the element $\mathrm{ind}^N(\Gamma) \in KK(C(\mathfrak{S}^*\Gamma), C^*(\Gamma))$. 
By the naturality of connecting elements, the connecting element of (\ref{lemK_exact}) is equal to $[\phi]\otimes_{C(\mathfrak{S}^*\Gamma)} \mathrm{ind}^N(\Gamma)$. 

Let us consider the following $KK$-elements. 
\begin{itemize}
\item The element $[\sigma(D_\pi)] \in K^1(C(\mathfrak{S}^*\Gamma))$. 
This element coincides with the element in $KK(\mathbb{C}l_1, C(\mathfrak{S}^*\Gamma; \mathrm{End}(S(\mathfrak{A}\Gamma)))$ given by the unital $*$-homomorphism which maps $\epsilon \in \mathbb{C}l_1$ to $\sigma(\psi(D_\pi))$ (see Remark \ref{bdd_trans}). 
\item The element $[m] \in KK(C(N)\otimes C(\mathfrak{S}^*\Gamma), C(\mathfrak{S}^*\Gamma)))$ given by the $*$-homomorphism $f \otimes \xi \mapsto f\cdot \xi$. 
\end{itemize}
We see the equality $[\phi] = [\sigma(D_\pi)]\otimes_{C(\mathfrak{S}^*\Gamma)}[m]$. 
On the other hand, the element $\mathrm{ind}^N(\Gamma) \in KK^1(C(\mathfrak{S}^*\Gamma), \mathcal{K}(H_Y))\simeq KK(C_0(\mathfrak{S}^*\Gamma \times \mathbb{R}^*_+), C(Y))$ is the element giving the family index map.  
If we denote by $[p.d.]\in KK(C(N)\otimes C_0((T^VN)^*), C(Y) )$ the element which gives the fiberwise Poincar\'e duality and by $q : C_0(\mathfrak{S}^*\Gamma \times \mathbb{R}^*_+) \to C_0((T^VN)^*)$ the inclusion, we have the equation
\[
[m] \otimes_{C(\mathfrak{S}^*\Gamma)}\mathrm{ind}^N(\Gamma) = [q]\otimes_{C_0((T^VN)^*)}[p.d.] \in KK^1(C(N)\otimes C(\mathfrak{S}^*\Gamma), C(Y)).  
\]
(see \cite[pp.1159--1162]{CS}). 
Thus we see that the product $[\phi]\otimes_{C(\mathfrak{S}^*\Gamma)} \mathrm{ind}^N(\Gamma) = [\sigma(D_\pi)]\otimes_{C(\mathfrak{S}^*\Gamma)}[q]\otimes_{C_0((T^VN)^*)}[p.d.] $ is the element $\pi_!\in KK^1(C(N), C(Y))$, namely the element given by the Kasparov module
$(L_Y^2(N; S(T^VN)), \mathrm{multi}, \psi(D_\pi))$, where multi denotes the multiplication by $C(N)$. 
\end{proof}

\begin{rem}\label{rem_rigid}
As we work in $KK$-theory in this section, we only need $C^*$-algebras to be defined up to  $KK$-equivalence. 
As in Lemma \ref{lemK_ext}, in order to define $C^*$-algebras in terms of operators, we need to fix rigid structures, such as differential $spin^c$-structures. 
However, the $KK$-equivalence class is determined by homotopy equivalence class of those structures, such as pre-$spin^c$-structure (c.f. Remark \ref{rem_twisted_spinc}). 
In order to simplify the arguments, we often omit this procedure of ``choosing a rigid structure, defining algebras and forgetting the structure to get a $KK$-equivalence class'', but the reader should note that we always need such steps.  
\end{rem}

\begin{defn}[$\mathcal{A}_\pi$]\label{def_A_pi}
Let $(M,\pi : \partial M \to Y)$ be a compact manifold with fibered boundary.
Assume that $\pi$ is equipped with a pre-$spin^c$-structure. 
Denote $i : \partial M \to M$ the inclusion. 
\begin{enumerate}
\item Assume that the typical fiber of $\pi$ is odd dimensional. 
We define $\mathcal{A}_\pi$ to be the $C^*$-algebra defined by the pullback (c.f. Remark \ref{rem_rigid})
\[
\xymatrix{
\mathcal{A}_\pi \ar[r] \ar[d] \pullbackcorner & \Psi(D_\pi) \ar[d] \\
C(M) \ar[r]^{i^*} &  C(\partial M) 
}
\]

\item Assume that the typical fiber of $\pi$ is even dimensional.  
We define $\mathcal{A}_\pi$ to be the $\mathbb{Z}_2$-graded $C^*$-algebra defined by the pullback (c.f. Remark \ref{rem_rigid})
\[
\xymatrix{
\mathcal{A}_\pi \ar[r] \ar[d] \pullbackcorner & \Psi(D_\pi) \ar[d] \\
C(M)\otimes \mathbb{C}l_1 \ar[r]^{i^*} &  C(\partial M) \otimes \mathbb{C}l_1
}
\]
\end{enumerate}
\end{defn}

\begin{prop}\label{A_pi_longexact}
Let $(M,\pi : \partial M \to Y)$ be a compact manifold with fibered boundary.
Assume that $\pi$ is equipped with a pre-$spin^c$-structure. 
The $K$-groups of the $C^*$-algebra $\mathcal{A}_\pi$ naturally fits in the exact sequence
\[
\cdots  \to K^{*}(M) \xrightarrow{\pi_! \circ i^*} {K}^{*-n}(Y) \to {K}_{*-n}(\mathcal{A}_\pi) \to {K}^{*+1}(M) \xrightarrow{\pi_! \circ i^*} \cdots, 
\]
where $n$ is the dimension of the fiber of $\pi$. 
\end{prop}

\begin{proof}
We only prove the Proposition in the case $n$ is even. 
The odd case is similar. 
By Lemma \ref{lemK_ext} and the surjectivity of the restriction $i^* : C(M)\otimes \mathbb{C}l_1 \to C(\partial M) \otimes \mathbb{C}l_1$, we get the graded exact sequence
\begin{equation}\label{A_pi_ext_eq}
0 \to \mathcal{K}(L_Y^2(N; S(T^VN))) \to \mathcal{A}_\pi \to C(M) \otimes \mathbb{C}l_1 \to 0. 
\end{equation}
By Lemma \ref{lemK_ext}, the connecting element associated to the above exact sequence is equal to $[i] \otimes\pi_! \in KK(C(M), C(Y)) \simeq KK^1(C(M) \otimes \mathbb{C}l_1, \mathcal{K}(L_Y^2(N; S(T^VN))))$. 
Thus the long exact sequence of $K$-groups associated to the above short exact sequence gives the desired sequence. 
\end{proof}

\begin{lem}\label{lem_K_class}
Let $(M,\pi : \partial M \to Y)$ be a compact manifold with fibered boundary.
Denote $\Gamma$ the groupoid $\partial M \times_\pi \partial M \rightrightarrows \partial M$. 
Assume that $\pi$ is equipped with a pre-$spin^c$-structure $P'_\pi$. 
Let $E$ be a $\mathbb{Z}_2$-graded complex vector bundle over $M$. 
Let $Q_\pi \in \mathcal{I}(P'_\pi, E)$ (see Remark \ref{rem_twisted_spinc}). 
Then the pair $(E, Q_\pi)$ naturally defines a class $[(E, Q_\pi)] \in K_{n-1}(\mathcal{A}_\pi)$, where $n$ is the dimension of the fiber of $\pi$. 
\end{lem}

\begin{proof}
We only prove the Lemma in the case $n$ is even. 
Let us denote $D_\pi^E$ the fiberwise $spin^c$-Dirac operators twisted by $E$ (defined using any additional choice; see Remark \ref{rem_twisted_spinc}). 
Let $\Psi(D_\pi^E)$ denote the $\mathbb{Z}_2$-graded $C^*$-subalgebra of $\mathcal{B}(L_Y^2(N; S(T^VN)\hat{\otimes}E))$ generated by $\{\psi(D_\pi^E)\}$, $C(N; \mathrm{End}(E))$ and $\mathcal{K}(L_Y^2(N; S(T^VN)\hat{\otimes}E))$.
Consider the graded $C^*$-algebra $\mathcal{A}_\pi(E)$ defined by the pullback
\begin{equation}\label{def_A_pi_coeff}
\xymatrix{
\mathcal{A}_\pi(E) \ar[r] \ar[d] \pullbackcorner & \Psi(D_\pi^E) \ar[d] \\
C(M; \mathrm{End}(E))\hat{\otimes}\mathbb{C}l_1 \ar[r]^{i^*} &  C(\partial M; \mathrm{End}(E))\hat{\otimes}\mathbb{C}l_1 
}
\end{equation}
as in Definition \ref{def_A_pi}. 
There is a canonical Morita equivalence between $\mathcal{A}_\pi$ and $\mathcal{A}_\pi(E)$. 

Given a pair $(E,Q_\pi)$ where $Q_\pi \in \mathcal{I}(P'_\pi, E)$, we define an element in $K_1(\mathcal{A}_\pi(E))$ as follows. 
Let us choose a representative $\tilde{D}_\pi^E \in \tilde{\mathcal{I}}_{\mathrm{sm}}(D_\pi^E)$. 
Since $\tilde{D}_\pi^E$ is an invertible family which differs from $D_\pi^E$ by a lower order family, $\psi(\tilde{D}_\pi^E)$ is an invertible operator in $\Psi(D_\pi^E)$. 
Thus the element $(1_M \hat{\otimes}\epsilon, \psi(\tilde{D}_\pi^E)) \in \mathcal{A}_\pi(E)$ is invertible. 
The $K_1$ class defined by this element does not depend on the choice of $\tilde{D}_\pi^E$. 
By composing with the $KK$-equivalence between $\mathcal{A}_\pi$ and $\mathcal{A}_\pi(E)$, we get the element $[(E, Q_\pi)] := [(1_M \hat{\otimes}\epsilon, \psi(\tilde{D}_\pi^E))] \in K_1(\mathcal{A}_\pi)$. 
\end{proof}

Now we assume that $M$ is even dimensional and $Y$ is odd dimensional. 
We construct an element $[(P'_M, P'_Y)] \in KK^n(\mathcal{A}_\pi, \Sigma^{\mathring{M}}(G_\Phi))$
for given $P'_M$, $P'_Y$ pre-$spin^c$-structures on $TM$ and $TY$ which are compatible with the given fiberwise pre-$spin^c$ structure $P'_\pi$ at the boundary, i.e., $P'_M|_{\partial M} = \pi^*P'_Y \oplus P'_\pi$. 

The next lemma can be proved in the same way as Lemma \ref{lemK_ext}. 

\begin{lem}\label{lemK_ext_phi}
Let $(M^{\mathrm{ev}},\pi : \partial M \to Y^{\mathrm{odd}})$ be a compact manifold with fibered boundary, equipped with a pre-$spin^c$-structure $P'_\pi$ on $T^V\partial M$. 
Let $P'_M$, $P'_Y$ be pre-$spin^c$-structures on $TM$ and $TY$ respectively. 
We assume that the pre-$spin^c$-structures are compatible at the boundary. 

Choose differential $spin^c$ structures on $TY$ and $T^VM$ representing $P'_Y$ and $P'_\pi$, and denote the associated $spin^c$-Dirac operator on $G_\Phi|_{\partial M} = \partial M \times_\pi \partial M \times_\pi TY \times \mathbb{R} \rightrightarrows \partial M$ as, 
\[
D_{\Phi, \partial} := D_\pi \hat{\otimes} 1 + 1 \hat{\otimes} D_{TY \times \mathbb{R}}. 
\]
Let $\Psi(D_{\Phi, \partial})$ denote the $\mathbb{Z}_2$-graded $C^*$-subalgebra of $\overline{\Psi^0_c(G_\Phi|_{\partial M}; S(\mathfrak{A}G_\Phi|_{\partial M}))}$ generated by $\{\psi(D_{\Phi, \partial})\}$, $C(\partial M)$ and $C^*(G_\Phi|_{\partial M}; S(\mathfrak{A}G_\Phi|_{\partial M}))$. 
If we choose any other differential $spin^c$-structures on $TY$ and $T^VM$ representing $P'_Y$ and $P'_\pi$, the resulting $C^*$-algebras are canonically $KK$-equivalent. 

This $C^*$-algebra fits into the graded exact sequence 
\[
0 \to C^*(G_\Phi|_{\partial M}; S(\mathfrak{A}G_\Phi|_{\partial M})) \to \Psi(D_{\Phi, \partial}) \to C(\partial M)\otimes \mathbb{C}l_1 \to 0. 
\]
The connecting element of this extension coincides with the class $\pi_! \otimes_{C(Y)}\mathrm{Ind}^Y(D_{TY \times \mathbb{R}})\in KK(C(\partial M), C^*(G_\Phi|_{\partial M}))$ ($\mathrm{Ind}^Y$ is defined in (\ref{eq_ind_y})). 
\end{lem}

\begin{defn}[$\mathcal{B}_\pi$]\label{def_B_pi}
In the situations in Lemma \ref{lemK_ext_phi}, we define a graded $C^*$-algebra by the pullback (c.f. Remark \ref{rem_rigid})
\[
\xymatrix{
\mathcal{B}_\pi \ar[r] \ar[d] \pullbackcorner & \Psi(D_{\Phi, \partial}) \ar[d] \\
C(M)\otimes \mathbb{C}l_1 \ar[r]^{i^*} &  C(\partial M) \otimes \mathbb{C}l_1
}
\]
Choosing a differential $spin^c$ structure of $G_\Phi$ representing $P'_M$, we get a canonical injective $*$-homomorphism $\iota : \mathcal{B}_\pi \to \Sigma^{\mathring{M}}(G_\Phi; S(\mathfrak{A}G_\Phi))$ as follows. 
Denote the $spin^c$-Dirac operator on $G_\Phi$ by $D_\Phi \in \mathrm{Diff}^1(G_\Phi; S(\mathfrak{A}G_\Phi))$ and the principal symbol of the bounded transform of this operator by $\sigma(\psi(D_\Phi)) \in C(\mathfrak{S}^*G_\Phi; \mathrm{End}(S(\mathfrak{A}G_\Phi)))$, which is an odd self-adjoint unitary element. 
Thus the $*$-homomorphism 
\begin{align*}
C(M) \otimes\mathbb{C}l_1 &\to C(\mathfrak{S}^*G_\Phi; \mathrm{End}(S(\mathfrak{A}G_\Phi))) \\
f \otimes 1&\mapsto f \\
1 \otimes \epsilon &\mapsto \sigma(\psi(D_\Phi))
\end{align*} 
is well-defined. 
Recall we have $\Sigma^{\mathring{M}}(G_\Phi; S(\mathfrak{A}G_\Phi))  = C(\mathfrak{S}^*G_\Phi; \mathrm{End}(S(\mathfrak{A}G_\Phi)))\oplus_{\partial M} \overline{\Psi^0_c(G_\Phi|_{\partial M}; S(\mathfrak{A}G_\Phi|_{\partial M}))} $. 
It is easy to see that the inclusion $\Psi(D_{\Phi, \partial}) \to \overline{\Psi^0_c(G_\Phi|_{\partial M}; S(\mathfrak{A}G_\Phi|_{\partial M}))} $ is compatible with the above $*$-homomorphism at $\partial M$, so they combine to induce the desired $*$-homomorphism
$\iota : \mathcal{B}_\pi \to \Sigma^{\mathring{M}}(G_\Phi; S(\mathfrak{A}G_\Phi))  $. 
The $KK$-element $[\iota] \in KK(\mathcal{B}_\pi, \Sigma^{\mathring{M}}(G_\Phi))$ is independent of the differential $spin^c$-structure representing $P'_M$.  
\end{defn}

Next, in the settings in Lemma \ref{lemK_ext_phi}, we construct an element $\mu \in KK(\mathcal{A}_\pi, \mathcal{B}_\pi)$ which fits into a commutative diagram in $KK$-theory, 
\begin{equation*}
\xymatrix{
0 \ar[r] &C^*(\Gamma; S(\mathfrak{A}\Gamma)) \ar[r] \ar[d]&\mathcal{A}_\pi \ar[r]\ar[d]^{\mu}& C(M ) \otimes \mathbb{C}l_1\ar[r]\ar@{=}[d] &0
\\
0 \ar[r]  &C^*(G_\Phi|_{\partial M}; S(\mathfrak{A}G_\Phi|_{\partial M})) \ar[r]  &\mathcal{B}_\pi \ar[r] &C( M ) \otimes \mathbb{C}l_1\ar[r] &0.
}
\end{equation*}
Here we denote by $\Gamma$ the groupoid $N \times_\pi N \rightrightarrows N$. 

First we consider a general setting. 
Suppose we are given a compact manifold $Y$, and an oriented real vector bundle $V$ over $Y$. 
For simplicity we only consider the case where the rank $m$ of this vector bundle $V$ is odd. 
We consider the Clifford algebra bundle $\mathbb{C}l(V)$ over $V$ (trivial on each fiber of $V \to Y$) and the $\mathbb{Z}_2$-graded $C^*$-algebra $C_0(V; \mathbb{C}l(V))$, where the algebra structure comes from the pointwise operations and grading comes from the grading on $\mathbb{C}l(V)$. 
We consider a variant of the construction of an asymptotic morphism in \cite[Section 1]{GH}, which gives a $KK$-equivalence between $C_0(V; \mathbb{C}l(V))$ and $C(Y)$. 
We are going to apply the following general construction to the real vector bundle $V = T^*Y$ in our $\Phi$-$spin^c$-setting. 
The construction below is also used in the next subsection for signature operators. 

We fix a riemannian metric on $V$. 
On the Hilbert $C(Y)$-module $L_Y^2(V; \mathbb{C}l(V))$, we consider unbounded, odd essentially self-adjoint operators $C_V$, $D_V$, $B_V$ defined as follows. 
They are families of operators parametrized by $Y$, and for each $y \in Y$, the operator  acts on $L^2(V_y ; \mathbb{C}l(V_y))$ as 
\begin{align*}
D_{V, y} &:= \sum_i \hat{e}_i \frac{\partial}{\partial x_i} \\
C_{V, y} &:= \sum_i x_ie_i \\
B_V &:= D_V + C_V ,
\end{align*}
by choosing an oriented orthonormal basis $\{e_i\}_i$ of $V_y$ and denoting the corresponding orthonormal coodinates by $x_i$ on $V_y$. 
Here we denote the actions of elements in $V_y$ on $\mathbb{C}l(V)$ by
\begin{align*}
e(x) &= e \cdot x \\
\hat{f}(x) &= (-1)^{deg(x)}x \cdot f. 
\end{align*}
Consider the operator $\omega_y := (-1)^{(m+1)/4}\hat{e}_1\hat{e}_2\cdots\hat{e}_m \in \mathrm{End}(\mathbb{C}l(V_y))$. 
This element does not depend on the choice of an oriented orthonormal basis, and gives an odd element $\omega \in \mathrm{End}(Y; \mathbb{C}l(V))$. 
Recall that we have assumed that $m$ is odd. 
This operator satisfies
$\omega^2 = 1$, $\omega=\omega^*$, $\omega_y e_i = -e_i\omega_y$ and $\omega_y \hat{e}_i = \hat{e}_i \omega_y$. 
So we see that $D_V\omega$ is an odd  essentially self-adjoint operator, $D_V\omega = \omega D_V$ and $C_V \omega = -\omega C_V$. 

It is well-known that, the operator $B_V$, called the harmonic oscillator, is Fredholm and has rank $1$-kernel in the even part and zero kernel in the odd part. 
Moreover, this kernel bundle has a canonical trivialization, given by the global section $\{e^{-\| v\|^2/2}\in L^2(V_y ; \mathbb{C}l(V_y))\}_y$.  
We let $P \in \mathcal{K}(L^2_Y(V; \mathbb{C}l(V)))$ to be the projection to the kernel of $B_V$. 
We denote the asymptotic algebra of $\mathcal{K}(L^2_Y(V; \mathbb{C}l(V))$ by
\[
\mathfrak{A}_\mu(\mathcal{K}(L^2_Y(V; \mathbb{C}l(V))) := C_b([\mu, \infty);\mathcal{K}(L^2_Y(V; \mathbb{C}l(V)))/C_0([\mu, \infty);\mathcal{K}(L^2_Y(V; \mathbb{C}l(V)))
\]
for $\mu \in \mathbb{R}$. 
Of course for any $\mu \in \mathbb{R}$, the above algebra $\mathfrak{A}_\mu(\mathcal{K}(L^2_Y(V; \mathbb{C}l(V)))$ are canonically isomorphic. 

We define a $\mathbb{Z}_2$-graded $*$-homomorphism $\phi$ by
\begin{align*}
\phi :  C_0(\mathbb{R} ; \mathbb{C}l_1) \hat{\otimes}C_0(V; \mathbb{C}l(V)) &\to \mathfrak{A}_1(\mathcal{K}(L^2_Y(V; \mathbb{C}l(V))) \\
f \hat{\otimes}h & \mapsto [[1, \infty) \ni t \mapsto f(t^{-1}D_{V}\omega) \cdot M_{h_t}] \\
f\epsilon \hat{\otimes} h &\mapsto [[1, \infty) \ni t \mapsto f(t^{-1}D_{V}\omega)\omega \cdot M_{h_t}], 
\end{align*}
for $f \in C_0(\mathbb{R})$ and $h \in C_0(V; \mathbb{C}l(V)) $. 
Here we denote by $h_t \in C_0(V; \mathbb{C}l(V))$ the function $h_t(v) := h(t^{-1}v)$ and by $M_{h_t}$ the pointwise Clifford multiplication operator by $h_t$. 
The following lemma is an analogue of \cite[Proposition 1.5]{GH}. 
\begin{lem}\label{lem_asymp_mor}
The map $\phi$ defines a $*$-homomorphism.  
\end{lem}
\begin{proof}
This can be proved in an analogous way to the proof of \cite[Proposition 1.5]{GH}.
Instead of repeating the proof, we only point out that the essential point is that we have the following relations, 
\[
e_i D_V = -D_V e_i ,\ e_i\omega_y = -\omega_y e_i, \mbox{ and } D_V \omega = \omega D_V. 
\]
\end{proof}

\begin{rem}
In \cite{GH}, they use the $C^*$-algebra $\mathcal{S} :=(C_0(\mathbb{R}) \mbox{ with even-odd grading})$. 
For an Euclidean space $V$, they define a $\mathbb{Z}_2$-graded $*$-homomorphism
\begin{align*}
\tilde{\phi} :  \mathcal{S} \hat{\otimes}C_0(V; \mathbb{C}l(V)) &\to \mathfrak{A}_1(\mathcal{K}(L^2(V; \mathbb{C}l(V))) \\
f \hat{\otimes}h & \mapsto [t \mapsto f(t^{-1}D_{V}) \cdot M_{h_t}]. 
\end{align*}
This $*$-homomorphism defines an element in the $E$-theory group $[\tilde{\phi}] \in E(C_0(V; \mathbb{C}l(V)), \mathbb{C})$, since their definition of $E$-theory groups is $E(A, B) := [[\mathcal{S} \hat{\otimes}A \hat{\otimes}\mathcal{K}(\hat{H}), B\hat{\otimes}\mathcal{K}(\hat{H})]]$ (\cite[Sectiton 2.1]{GH}). 
Here we denoted the abelian group of equivalence classes of asymptotic morphisms from $A$ to $B$ by $[[A, B]]$. 
On the other hand, an asymptotic morphism from $A$ to $B$ which admits a completely positive lifting naturally defines an element in $KK(A, B)$. 
Thus their asymptotic morphism $\tilde{\phi}$ gives an element in $KK(\mathcal{S}\hat{\otimes}C_0(V; \mathbb{C}l(V)), \mathbb{C})$, which is not the desired element. 

We would like to treat $KK$-elements arising from asymptotic morphisms directly in our construction, so we do not use the $C^*$-algebra $\mathcal{S}$ and consider the above asymptotic morphism $\phi$, which indeed gives the $KK$-equivalence between $C_0(\mathbb{R} ; \mathbb{C}l_1) \hat{\otimes}C_0(V; \mathbb{C}l(V))$ and $C(Y)$ (see Proposition \ref{prop_mathcalC} below). 
\end{rem}

Define the $C^*$-algebra 
\[
T_P := \{f \in C_b([0, \infty); \mathcal{K}(L^2_Y(V; \mathbb{C}l(V)))) \ | \ f (0)= Pf(0) = f(0)P  \}. 
\]
Recall that $P$ is the orthogonal projection to $\ker B_V$, which is a rank one bundle on $Y$ with a canonical trivialization. 
So we can identify $C(Y)$ with $P\mathcal{K}(L^2_Y(V; \mathbb{C}l(V)))P$, to get $*$-homomorphism
$
ev_0 : T_P \to C(Y)
$. 
We define a $C^*$-algebra $\mathcal{C}$ by the pullback
\begin{equation}\label{def_mathcalC_var}
\xymatrix{
\mathcal{C} \ar[rr] \ar[d]^{ev^{\mathcal{C}}_\infty} \pullbackcorner & &T_P \ar[d] \\ 
C_0(\mathbb{R}; \mathbb{C}l_1)\hat{\otimes}C_0(V; \mathbb{C}l(V)) \ar[r]^-{\phi} &  \mathfrak{A}_1(\mathcal{K}(L^2_Y(V; \mathbb{C}l(V))) \ar[r]^\simeq & \mathfrak{A}_0(\mathcal{K}(L^2_Y(V; \mathbb{C}l(V)))
}
\end{equation}
We define the $*$-homomorphism $ev_0^{\mathcal{C}} : \mathcal{C} \to C(Y)$ by the composition of the top row of (\ref{def_mathcalC_var}) and $ev_0$. 
\begin{lem}\label{lem_asymp_var}
Let $g$ be a function in $C_0(\mathbb{R})$. 
For $t \in (0, \infty)$, consider the functional calculus $g(t^{-1}B_V) \in \mathcal{K}(L^2_Y(V; \mathbb{C}l(V)))$. 
This family canonically extends to an element in $T_P$ by setting $g(t^{-1}B_V)|_{t = 0} = g(0)P$. 
We abuse the notation and still denote this element by $g(t^{-1}B_V) \in T_P$. 

We define the function $X : \mathbb{R} \to \mathbb{R} , x \mapsto x$.  
Note that $X\epsilon$ and $C_V$ are odd unbounded multipliers on $C_0(\mathbb{R}; \mathbb{C}l_1)$ and $C_0(V; \mathbb{C}l(V))$, respectively. 
Then we have
\[
\phi(g(X\epsilon \hat{\otimes}1 + 1 \hat{\otimes}C_V)) = [g(t^{-1}B_V)] \in \mathfrak{A}(\mathcal{K}(L^2_Y(V; \mathbb{C}l(V))). 
\]
In other words, for a function $g \in C_0(\mathbb{R})$ we have an element $[g(X\epsilon \hat{\otimes} 1 + 1 \hat{\otimes} C_V), g(t^{-1}B_V)] \in \mathcal{C}$. 
Moreover, the operator $\mathbb{B}_V := [X\epsilon \hat{\otimes} 1 + 1 \hat{\otimes}C_V, t^{-1}B_V]$ (defined to be $0$ at $0 \in [0, \infty)$) is an unbounded multiplier (see \cite[pp. 168--169]{GH}) of $\mathcal{C}$.
\end{lem}

\begin{proof}
The essential point is that $B_V$ has discrete spectrum with finite multiplicity. 
The lemma is proved by checking on the generators $e^{-x^2}$ and $xe^{-x^2}$ of $C_0(\mathbb{R})$, and the computations are essentially the same as in \cite[Section 1.13]{GH}. 
\end{proof}

\begin{prop}\label{prop_mathcalC}
The $*$-homomorphisms
\begin{align*}
ev_0^{\mathcal{C}} &: \mathcal{C} \to C(Y), \mbox{ and } \\
ev_\infty^{\mathcal{C}} &: \mathcal{C} \to C_0(\mathbb{R}; \mathbb{C}l_1)\hat{\otimes}C_0(V; \mathbb{C}l(V))  \simeq C_0(V\oplus \mathbb{R}; \mathbb{C}l(V\oplus \mathbb{R}))
\end{align*}
induce $KK$-equivalences among $C(Y)$, $\mathcal{C}$ and $ C_0(V\oplus \mathbb{R}; \mathbb{C}l(V\oplus \mathbb{R}))$. 
Moreover, we have
\begin{equation}\label{prop_mathcalC_statement}
[ev_0^{\mathcal{C}}]^{-1} \otimes [ev_\infty^{\mathcal{C}}] = [\psi(X\epsilon \hat{\otimes}1 + 1 \hat{\otimes}C_V)] \in KK(C(Y), C_0(V\oplus \mathbb{R}; \mathbb{C}l(V\oplus \mathbb{R}))) . 
\end{equation}
\end{prop}

\begin{proof}
The fact that $ev_\infty^{\mathcal{C}}$ is a $KK$-equivalence is seen by checking that the kernel of this $*$-homomorphism, 
$\mathrm{ker}(ev_\infty^{\mathcal{C}}) = \{f \in C_0([0, \infty); \mathcal{K}(L^2_Y(V; \mathbb{C}l(V)))) \ | \ f (0)=Pf(0) = f(0)P \} $, 
is $KK$-contractible. 
Let us use the notation $H_Y :=L^2_Y(V; \mathbb{C}l(V))$ in this proof. 
We have the following commutative diagram, 
\begin{equation*}
\xymatrix{
0 \ar[r] & C_0((0, \infty); \mathcal{K}(H_Y)) \ar[r] \ar@{=}[d]& \ker(ev_\infty^{\mathcal{C}})\ar@{^{(}->}[d]\ar[r]^{ev_0}&C(Y)\ar[r]\ar@{^{(}->}[d] &0
\\
0 \ar[r]  &C_0((0, \infty); \mathcal{K}(H_Y)) \ar[r]  &C_0([0, \infty); \mathcal{K}(H_Y)) \ar[r]^-{ev_0} &\mathcal{K}(H_Y)\ar[r] &0, 
}
\end{equation*}
where the rows are exact. 
The right vertical inclusion is given by $C(Y) \simeq P\mathcal{K}(H_Y)P \hookrightarrow \mathcal{K}(H_Y)$ so this is a $KK$-equivalence. 
By the five lemma in $KK$-theory, we see that the middle vertical arrow is a $KK$-equivalence. 
Since $C_0([0, \infty); \mathcal{K}(H_Y)) $ is $KK$-contractible, so is $\ker (ev_\infty^{\mathcal{C}})$. 

On the other hand, 
by Lemma \ref{lem_asymp_var}, we see that $\psi(\mathbb{B}_V)$ is a self-adjoint multiplier of $\mathcal{C}$ which satisfies $\psi(\mathbb{B}_V)^2 - \mathrm{Id}_{\mathcal{C}} \in \mathcal{C}$, and by construction $\psi(\mathbb{B}_V)$ commutes with the action of $C(Y)$ on $\mathcal{C}$ by multiplication. 
So we get the element $[\psi(\mathbb{B}_V)] \in KK(C(Y), \mathcal{C})$. 
It satisfies $[\psi(\mathbb{B}_V)] \otimes[ev_\infty^{\mathcal{C}}] = [\psi(X\epsilon \hat{\otimes} 1 + 1 \hat{\otimes}C_V)]  \in KK(C(Y), C_0(\mathbb{R} \oplus V; \mathbb{C}l(\mathbb{R} \oplus V)))$ and $[\psi(\mathbb{B}_V)]\otimes [ev_0^{\mathcal{C}}] = id_{C(Y)} \in KK(C(Y), C(Y))$ (this is because $\psi(t^{-1}B_V)|_{t = 0} = \psi(0)P= 0$). 
Thus we have $ [\psi(X\epsilon \hat{\otimes} 1 + 1 \hat{\otimes}C_V)]  \otimes [ev_\infty^{\mathcal{C}}]^{-1} \otimes[ev_0^{\mathcal{C}}] = id_{C(Y)} \in KK(C(Y), C(Y))$. 
Since the element $ [\psi(X\epsilon \hat{\otimes} 1 + 1 \hat{\otimes}C_V)]  $ is the Thom element which is a $KK$-equivalence, we see that $ [ev_0^{\mathcal{C}}] \in KK(\mathcal{C}, C(Y))$ is a $KK$-equivalence, which proves (\ref{prop_mathcalC_statement}). 
\end{proof}

Now, we return to settings of manifolds with fibered boundaries $(M, \pi : \partial M\to Y)$ equipped with pre-$spin^c$-structures. 
We assume that $M$ is even dimensional and $Y$ is odd dimensional. 
We apply the above constructions for $V := T^*Y$ which is an oriented vector bundle over $Y$. 
Let us denote by $\Gamma$ the groupoid $\partial M \times_\pi \partial M \rightrightarrows \partial M$. 

Choosing differential $spin^c$-structures on $T^V\partial M$ and $TY$ representing the given pre-$spin^c$ structure, define a $C^*$-algebra $\Psi(\mathbb{B}_V \hat{\otimes} 1 + 1 \hat{\otimes} D_\pi)$ to be the $C^*$-subalgebra of $\mathcal{M}(\mathcal{C}\hat{\otimes}_{C(Y)} C^*(\Gamma; S(\mathfrak{A}\Gamma)))$ generated by $\{\psi (\mathbb{B}_V \hat{\otimes} 1 + 1 \hat{\otimes} D_\pi)\}$, $\mathcal{C} \hat{\otimes}_{C(Y)} C^*(\Gamma; S(\mathfrak{A}\Gamma))$ and $C(\partial M \times [0,\infty])$. 
Note that we have an exact sequence
\[
0 \to \mathcal{C} \hat{\otimes}_{C(Y)} C^*(\Gamma; S(\mathfrak{A}\Gamma)) \to \Psi(\mathbb{B}_V \hat{\otimes} 1 + 1 \hat{\otimes} D_\pi) \to C(\partial M \times [0, \infty]) \otimes \mathbb{C}l_1 \to 0
\]
analogous to (\ref{A_pi_ext_eq}). 

\begin{defn}[$\mathcal{D}_\pi$]
In the above settings, we define the $C^*$-algebra $\mathcal{D}_\pi$ by the pullback (c.f. Remark \ref{rem_rigid})
\[
\xymatrix{
\mathcal{D}_\pi \ar[r] \ar[d] \pullbackcorner & \Psi(\mathbb{B}_V \hat{\otimes} 1 + 1 \hat{\otimes} D_\pi) \ar[d] \\ 
C(M \times [0, \infty]) \otimes \mathbb{C}l_1 \ar[r]^{i^*} &  C(\partial M \times [0, \infty]) \otimes \mathbb{C}l_1
}
\] 
\end{defn}

We have a $*$-homomorphism $ev_0 := ev_0^{\mathcal{C}} \otimes_{C(Y)} id_{C^*(\Gamma)}: \mathcal{C}\hat{\otimes}C^*(\Gamma; S(\mathfrak{A}\Gamma)) \to C^*(\Gamma; S(\mathfrak{A}\Gamma)) $. 
This $*$-homomorphism extends to a $*$-homomorphism $\mathcal{D}_\pi \to \mathcal{A}_\pi$, by sending $ \psi(\mathbb{B}_V \hat{\otimes}1 + 1 \hat{\otimes}D_\pi)$ to $\psi(D_\pi)$ and by evaluating at $0 \in [0, \infty]$ on $C(M \times [0, \infty]) \otimes \mathbb{C}l_1$, 
since $\psi(0) = 0$ and $P$ is the projection to the kernel of $B_V$. 
We also denote this $*$-homomorphism by $ev_0$. 

On the other hand, using the isomorphism $C_0(\mathbb{R} \oplus T^*Y; \mathbb{C}l(\mathbb{R} \oplus T^*Y))\simeq C^*(TY \times \mathbb{R}; S(TY \times \mathbb{R}))$, we have a $*$-homomorphism $ev_\infty:= ev_\infty^{\mathcal{C}}\otimes_{C(Y)} id_{C^*(\Gamma)} : \mathcal{C}\hat{\otimes}_{C(Y)}C^*(\Gamma; S(\mathfrak{A}\Gamma)) \to C^*(TY \times \mathbb{R}; S(TY \times \mathbb{R})) \hat{\otimes}_{C(Y)} C^*(\Gamma, S(\mathfrak{A}\Gamma)) \simeq C^*(G_\Phi|_{\partial M}; S(\mathfrak{A}G_\Phi|_{\partial M}))$. 
This $*$-homomorphism extends to a $*$-homomorphism $\mathcal{D}_\pi \to \mathcal{B}_\pi$ by sending $\psi(\mathbb{B}_V \hat{\otimes}1 + 1 \hat{\otimes}D_\pi)$ to $\psi(D_{TY \times \mathbb{R}} \hat{\otimes}1 + 1 \hat{\otimes}D_\pi) = \psi(D_{\Phi, \partial})$ and evaluation at $\infty\in[0, \infty]$ on $C(M \times [0, \infty]) \otimes \mathbb{C}l_1$, 
since $X\epsilon \hat{\otimes} 1 + 1 \hat{\otimes} C_V$ corresponds to $D_{TY \times \mathbb{R}}$ under the Fourier transform. 
We also denote this $*$-homomorphism by $ev_\infty$. 

In the next proposition, we need to use the $\mathcal{R}KK$-theory for $C(Y)$-algebras (\cite[Section 2.19]{Kas}).  
Given a compact space $Y$ and two $C(Y)$-algebras $A$ and $B$, we get an abelian group $\mathcal{R}KK(Y; A, B)$, roughly by requiring Kasparov modules to be compatible with the action by $C(Y)$. 
We also have the Kasparov product in this theory, 
\begin{align*}
\mathcal{R}KK(Y; A_1, B_1\hat{\otimes}_{C(Y)}D) &\otimes \mathcal{R}KK(Y; D \hat{\otimes}_{C(Y)}A_2, B_2) \\
&\to \mathcal{R}KK(Y; A_1 \hat{\otimes}_{C(Y)}A_2, B_1 \hat{\otimes}_{C(Y)}B_2). 
\end{align*}
With an abuse of notation, we also denote this product by $\otimes_{D}$. 
Note that $\mathrm{Ind}^Y(D_{TY \times \mathbb{R}})\in KK(C(Y), C^*(TY \times \mathbb{R}))$ lifts canonically to an element in $\mathcal{R}KK(C(Y), C^*(TY\times \mathbb{R}))$, which is also denoted by $\mathrm{Ind}^Y(D_{TY \times \mathbb{R}})$. 

\begin{prop}\label{prop_asymp_var}
Assume that $M$ is even dimensional and $Y$ is odd dimensional. 
Consider the following commutative diagram, 
\begin{equation}\label{prop_asymp_var_diag1}
\xymatrix{
0 \ar[r] &C^*(\Gamma; S(\mathfrak{A}\Gamma)) \ar[r] &\mathcal{A}_\pi \ar[r]& C(M ) \otimes \mathbb{C}l_1\ar[r]&0
\\
0 \ar[r]  &\mathcal{C}\hat{\otimes}_{C(Y)}C^*(\Gamma; S(\mathfrak{A}\Gamma))  \ar[u]^{ev_0}\ar[d]^{ev_\infty} \ar[r] &\mathcal{D}_\pi \ar[r]\ar[u]^{ev_0}\ar[d]^{ev_\infty} & C(M \times [0, \infty] ) \otimes \mathbb{C}l_1\ar[r] \ar[u]^{ev_0}\ar[d]^{ev_\infty}&0
\\
0 \ar[r]  &C^*(G_\Phi|_{\partial M}; S(\mathfrak{A}G_\Phi|_{\partial M})) \ar[r] &\mathcal{B}_\pi \ar[r] &C( M ) \otimes \mathbb{C}l_1\ar[r] &0, 
}
\end{equation}
where the rows are exact. 
The vertical arrows are $KK$-equivalences, and we define
\begin{equation}\label{prop_asymp_statement}
\mu := [ev_0]^{-1} \otimes_{\mathcal{D}_\pi} [ev_\infty] \in KK(\mathcal{A}_\pi, \mathcal{B}_\pi). 
\end{equation}
Then this element fits into the commutative diagram in $KK$-theory, 
\begin{equation}\label{prop_asymp_var_diag2}
\xymatrix{
0 \ar[r] &C^*(\Gamma; S(\mathfrak{A}\Gamma)) \ar[r] \ar[d]_{ \mathrm{Ind}^Y(D_{TY \times \mathbb{R}}) \otimes id_{C^*(\Gamma)}} &\mathcal{A}_\pi \ar[r]\ar[d]^{\mu}& C(M ) \otimes \mathbb{C}l_1\ar[r]\ar@{=}[d] &0
\\
0 \ar[r]  &C^*(G_\Phi|_{\partial M}; S(\mathfrak{A}G_\Phi|_{\partial M})) \ar[r]  &\mathcal{B}_\pi \ar[r] &C( M ) \otimes \mathbb{C}l_1\ar[r] &0. 
}
\end{equation}
Here the left vertical arrow is defined by taking Kasparov product of elements $\mathrm{Ind}^Y(D_{TY \times \mathbb{R}}) \in \mathcal{R}KK(Y; C(Y), C^*(TY \times \mathbb{R}))$ and $id_{C^*(\Gamma)} \in \mathcal{R}KK(Y; C^*(\Gamma), C^*(\Gamma))$, and then forgetting the $C(Y)$-algebra structure. 
\end{prop}

\begin{proof}
For ease of notations, we drop the coefficient bundle in this proof and write, for example,  $C^*(\Gamma)$ for $C^*(\Gamma; S(\mathfrak{A}\Gamma))$.
The commutativity of the diagram (\ref{prop_asymp_var_diag1}) directly follows from the definition. 

That the arrows in the left column of (\ref{prop_asymp_var_diag1}) are $KK$-equivalences is the direct consequence of Proposition \ref{prop_mathcalC}, by noting that $ev_0 = ev_0^{\mathcal{C}}{\otimes}_{C(Y)} id_{C^*(\Gamma)}$ and $ev_\infty= ev_\infty^{\mathcal{C}}{\otimes}_{C(Y)} id_{C^*(\Gamma)}$. 
By (\ref{prop_mathcalC_statement}), we also have the equality
\begin{equation}\label{prop_asymp_var_eq1}
[ev_0]^{-1} \otimes [ev_\infty] = \mathrm{Ind}^Y(D_{TY \times \mathbb{R}}) {\otimes} id_{C^*(\Gamma)} \in \mathcal{R}KK(Y; C^*(\Gamma), C^*(G_\Phi|_{\partial M})), 
\end{equation}
by noting that the operator $D_{TY \times \mathbb{R}}$ corresponds to $X \epsilon \hat{\otimes}1 + 1 \hat{\otimes}C_{T^*Y}= C_{T^*Y \oplus \mathbb{R}}$ under the Fourier transform. 

Next let us look at the middle column of the diagram (\ref{prop_asymp_var_diag1}). 
By the commutativity of the diagram and the five lemma in $KK$-theory, the above $KK$-equivalence result on the left column implies that the middle vertical arrows are also $KK$-equivalences. 
Finally, the commutativity of the diagram (\ref{prop_asymp_var_diag2}) follows from (\ref{prop_asymp_var_eq1}). 
\end{proof}

\begin{defn}\label{def_Khomology_K}
Let $(M^{\mathrm{ev}},\pi : \partial M \to Y^{\mathrm{odd}})$ be a compact manifold with fibered boundary, equipped with a pre-$spin^c$-structures $P'_\pi$, $P'_M$ and $P'_Y$ on $T^V\partial M$, $TM$ and $TY$, respectively. 
We assume that these pre-$spin^c$ structures are compatible at the boundary. 
Then we define 
\[
[(P'_M, P'_Y)] := \mu \otimes [\iota] \in KK(\mathcal{A}_\pi, \Sigma^{\mathring{M}}(G_\Phi)). 
\]
Here the element $[\iota]\in KK( \mathcal{B}_\pi ,\Sigma^{\mathring{M}}(G_\Phi )) $ is defined in Definition \ref{def_B_pi} and $\mu \in KK(\mathcal{A}_\pi, \mathcal{B}_\pi)$ is defined in Proposition \ref{prop_asymp_var}. 
\end{defn}

\begin{thm}[The index pairing formula for $spin^c$-Dirac operators]\label{main_thm_K}
Let $(M^{\mathrm{ev}},\pi : \partial M \to Y^{\mathrm{odd}})$ be a compact manifold with fibered boundary. 
Let $P'_\pi$, $P'_M$ and $P'_Y$ be pre-$spin^c$-structures on $T^V\partial M$, $TM$ and $TY$ respectively. 
We assume that the pre-$spin^c$-structures are compatible at the boundary. 
Let $E$ be a complex vector bundle over $M$.  
Let $Q_\pi \in \mathcal{I}(P'_\pi, E)$. 
Then we have
\[
\mathrm{Ind}_\Phi(P'_M, P'_Y, E, Q_\pi) = [(E, Q_\pi)] \otimes_{\mathcal{A}_\pi} [(P'_M, P'_Y)] \otimes_{\Sigma^{\mathring{M}}(G_\Phi)} \mathrm{ind}^{\mathring{M}}(G_\Phi)\in \mathbb{Z}. 
\]
Here the element $\mathrm{ind}^{\mathring{M}}(G_\Phi) \in KK^1(\Sigma^{\mathring{M}}(G_\Phi), \mathbb{C})$ is defined in subsection \ref{subsubsec_ellipticity}. 
\end{thm}
\begin{proof}
We only prove the Theorem in the case where the dimension of the fiber of $\pi$ is even. 
As in Lemma \ref{lem_K_class}, a pair $(E, Q_{\Phi, \partial} )$, where $E \to M$ is a $\mathbb{Z}_2$-graded complex vector bundle and $Q_{\Phi, \partial} \in \mathcal{I}(D_{\Phi, \partial}^E)$, naturally defines a class $[(E, Q_{\Phi, \partial} )] \in K_1(\mathcal{B}_\pi)$. 
Here we denoted by $D_{\Phi, \partial}^E$ the $spin^c$ Dirac operator twisted by $E$ and $\mathcal{I}(D_{\Phi, \partial}^E)$ is the set of homotopy classes of $\mathbb{C}l_1$-invertible perturbations of the operator $D_{\Phi, \partial}^E$ on $G_\Phi|_{\partial }$. 

First we remark that, when we are given a twisting bundle $E$, it is convenient to use $C^*$-algebras $\mathcal{A}_\pi(E)$, $\mathcal{B}_\pi(E)$ and $\mathcal{D}_\pi(E)$ which are Morita equivalent to $\mathcal{A}_\pi$, $\mathcal{B}_\pi$ and $\mathcal{D}_\pi$, respectively; the first one appears in (\ref{def_A_pi_coeff}) and the definitions of the other algebras are self-explanatory. 
The corresponding element $\mu \in KK(\mathcal{A}_\pi(E), \mathcal{B}_\pi(E))$ is realized as $[ev_0]^{-1}\otimes_{\mathcal{D}_\pi(E)}[ev_\infty]$ as in Proposition \ref{prop_asymp_var}, and a $*$-homomorphism $\iota_E : \mathcal{B}_\pi(E) \to \Sigma^{\mathring{M}}(G_\Phi; S(\mathfrak{A}G_\Phi)\hat{\otimes}E)$ is constructed analogously to Definition \ref{def_B_pi}. 

It is enough to prove the following. 
Given an element $Q_\pi \in \mathcal{I}(P_\pi, E)$, take some representative $\tilde{D}_\pi^E$ for $Q_\pi$ and consider the class $Q_{\Phi, \partial} := [\tilde{D}_\pi^E \hat{\otimes} 1 +1 \hat{\otimes} D_{TY \times \mathbb{R}} ] \in \mathcal{I}({D}_\pi^E \hat{\otimes} 1 +1 \hat{\otimes} D_{TY \times \mathbb{R}} )$ (this class does not depend on the choice). 
Then we have
\begin{equation}\label{mainthm_eq1}
[(E, Q_\pi) ]\otimes_{\mathcal{A}_\pi} \mu = [(E, Q_{\Phi, \partial})] \in KK(\mathbb{C}l_1, \mathcal{B}_\pi). 
\end{equation}
For, in view of the definition of $\iota_E$, we have the equality
\begin{align*}
[(E, Q_{\Phi, \partial})]\otimes \iota \otimes \mathrm{ind}^{\mathring{M}}(G_\Phi) &= [(\sigma(D_\Phi^E), \psi(\tilde{D}_\pi^E \hat{\otimes}1 + 1 \hat{\otimes }D_{TY \times \mathbb{R}}))]\otimes \mathrm{ind}^{\mathring{M}}(G_\Phi) \\
&= \mathrm{Ind}_\Phi(P'_M, P'_Y, E, Q_\pi) . 
\end{align*}
So let us prove (\ref{mainthm_eq1}). 
For simplicity we assume $E$ is the trivial bundle. 
For a given representative $\tilde{D}_\pi$ for $Q_\pi \in \mathcal{I}(P_\pi)$, we can construct an invertible element in $\mathcal{D}_\pi$ defined as $(1 \hat{\otimes} \epsilon, \psi(\mathbb{B}_V \hat{\otimes} 1 + 1 \hat{\otimes} \tilde{D}_\pi)) \in \mathcal{D}_\pi$. 
By definition we have 
\begin{align*}
ev_0((1 \hat{\otimes} \epsilon, \psi(\mathbb{B}_V \hat{\otimes} 1 + 1 \hat{\otimes} \tilde{D}_\pi))) &=(1 \hat{\otimes}\epsilon, \psi(\tilde{D}_\pi)) &\in \mathcal{A}_\pi \\
ev_\infty((1 \hat{\otimes} \epsilon, \psi(\mathbb{B}_V \hat{\otimes} 1 + 1 \hat{\otimes} \tilde{D}_\pi))) &=(1 \hat{\otimes}\epsilon, \psi(\tilde{D}_{\Phi, \partial})) &\in \mathcal{B}_\pi. 
\end{align*}
Thus we see that $[(1 \hat{\otimes} \epsilon, \psi(\mathbb{B}_V \hat{\otimes} 1 + 1 \hat{\otimes} \tilde{D}_\pi)) ]\otimes[ev_0] = [(\underline{\mathbb{C}}, Q_\pi)] \in K_1(\mathcal{A}_\pi)$ and $[(1 \hat{\otimes} \epsilon, \psi(\mathbb{B}_V \hat{\otimes} 1 + 1 \hat{\otimes} \tilde{D}_\pi)) ] \otimes[ev_\infty]= [(\underline{\mathbb{C}}, Q_{\Phi, \partial})] \in K_1(\mathcal{B}_\pi)$. 
By Proposition \ref{prop_asymp_var}, we see that $[(\underline{\mathbb{C}}, Q_\pi)]  \otimes_{\mathcal{A}_\pi} \mu = [(\underline{\mathbb{C}}, Q_\pi)] \otimes_{\mathcal
{A}_\pi} [ev_0]^{-1} \otimes_{\mathcal{D}_\pi} [ev_\infty] = [(\underline{\mathbb{C}}, Q_{\Phi, \partial})] \in K_1(\mathcal{B}_\pi)$, so we get (\ref{mainthm_eq1}). 
\end{proof}

\subsection{The case of signature operators}

In this subsection, we consider the case of signature operators. 
The arguments are parallel to those in subsection \ref{sec_relativeK_spin}. 
Let $(M, \pi : \partial M \to Y)$ be a compact manifold with fibered boundaries, with fixed orientation on $TM$ and $TY$. 
For simplicity we only consider the case where the fibers of $\pi$ are even dimensional. 

We have the element $\mathrm{Ind}(D_\pi^{\mathrm{sign}}) \in KK(C(\partial M), C(Y))$ given by the fiberwise family of signature operators. 
This element gives the ``signature pushforward'' homomorphism, 
\[
\pi^{\mathrm{sign}}_! := \otimes_{C(\partial M)}\mathrm{Ind}(D_\pi^{\mathrm{sign}}) : K^*(\partial M) \to K^*(Y). 
\]
First, exactly in the analogous way to that in the last subsection, we define a $C^*$-algebra $\mathcal{A}^{\mathrm{sign}}_\pi$ whose $K$-group fits in the exact sequence
\[
\cdots   \to K^{*}(M) \xrightarrow{\pi^{\mathrm{sign}}_! \circ i^*} {K}^{*}(Y) \to {K}_{*}(\mathcal{A}^{\mathrm{sign}}_\pi) \to {K}^{*+1}(M) \xrightarrow{\pi^{\mathrm{sign}}_! \circ i^*}\cdots. 
\] 

Let $N$ be a compact space. 
Let $\pi : N \to Y$ be a fiber bundle whose fibers are equipped with even dimensional closed manifold structure, and an orientation of $\pi$ is fixed. 
Choose any fiberwise riemannian metric, and
denote the fiberwise signature operator by $D_\pi^{\mathrm{sign}}$. 
Let $L_Y^2(N; \wedge_{\mathbb{C}}(T^VN)^*)$ denote the $\mathbb{Z}_2$-graded Hilbert $C(Y)$-module which is obtained by the completion of $C_c^\infty(N; \wedge_{\mathbb{C}}(T^VN)^*)$ with the natural $C(Y)$-valued inner product. 

Denote the odd function $\psi(x) := x/\sqrt{1+x^2}$. 
Let $\Psi(D^{\mathrm{sign}}_\pi)$ denote the $\mathbb{Z}_2$-graded $C^*$-subalgebra of $\mathcal{B}(L_Y^2(N; \wedge_{\mathbb{C}}(T^VN)^*))$ generated by $\{\psi(D^{\mathrm{sign}}_\pi)\}$, $C(N)$ and $\mathcal{K}(L_Y^2(N; \wedge_{\mathbb{C}}(T^VN)^*))$.

\begin{lem}\label{lemK_ext_sign}
The algebra $\Psi(D^{\mathrm{sign}}_\pi)$ fits into the exact sequence of graded $C^*$-algebras
\begin{equation}\label{lemK_exact_sign}
0 \to \mathcal{K}(L_Y^2(N; \wedge_{\mathbb{C}}(T^VN)^*)) \to \Psi(D^{\mathrm{sign}}_\pi) \to C(N)\otimes \mathbb{C}l_1 \to 0. 
\end{equation}
The connecting element of this extension coincides with the class $\pi^{\mathrm{sign}}_! \in KK(C(N), C(Y))$.
\end{lem}

\begin{defn}[$\mathcal{A}^{\mathrm{sign}}_\pi$]\label{def_A_pi_sign}
Let $(M,\pi : \partial M \to Y)$ be a compact manifold with fibered boundary.
Assume that $\pi$ is oriented and the fibers are even dimensional. 
Denote $i : \partial M \to M$ the inclusion. 

We define $\mathcal{A}^{\mathrm{sign}}_\pi$ to be the $\mathbb{Z}_2$-graded $C^*$-algebra defined by the pullback (c.f. Remark \ref{rem_rigid})
\[
\xymatrix{
\mathcal{A}^{\mathrm{sign}}_\pi \ar[r] \ar[d] \pullbackcorner & \Psi(D^{\mathrm{sign}}_\pi) \ar[d] \\
C(M)\otimes \mathbb{C}l_1 \ar[r]^{i^*} &  C(\partial M) \otimes \mathbb{C}l_1
}
\]
\end{defn}

We can prove that this $C^*$-algebra induces the desired long exact sequence, analogously to Proposition \ref{A_pi_longexact}. 
\begin{prop}\label{A_pi_longexact_sign}
Let $(M,\pi : \partial M \to Y)$ be a compact manifold with fibered boundary.
Assume that $\pi$ is oriented and the dimension of fibers is even. 
The $K$-groups of the $C^*$-algebra $\mathcal{A}^{\mathrm{sign}}_\pi$ naturally fits in the exact sequence
\[
\cdots   \to K^{*}(M) \xrightarrow{\pi^{\mathrm{sign}}_! \circ i^*} {K}^{*}(Y) \to {K}_{*}(\mathcal{A}^{\mathrm{sign}}_\pi) \to {K}^{*+1}(M) \xrightarrow{\pi^{\mathrm{sign}}_! \circ i^*}  \cdots. 
\] 
\end{prop}

Then analogously to Lemma \ref{lem_K_class}, we have the following. 
\begin{lem}\label{lem_K_class_sign}
Let $(M,\pi : \partial M \to Y)$ be a compact manifold with fibered boundary.
Let us denote by $\Gamma$ the groupoid $\partial M \times_\pi \partial M \rightrightarrows \partial M$. 
Assume that $\pi$ is oriented and the fibers are even dimensional. 
Let $E$ be a $\mathbb{Z}_2$-graded complex vector bundle over $M$. 
Assume we are given an element $Q_\pi \in \mathcal{I}^{\mathrm{sign}}(\pi, E)$.  
Then the pair $(E, Q_\pi)$ naturally defines a class $[(E, Q_\pi)] \in K_1(\mathcal{A}^{\mathrm{sign}}_\pi)$. 
\end{lem}

Now we assume that $M$ is even dimensional and oriented. 
We construct an element $[M^{\mathrm{sign}}] \in KK(\mathcal{A}^{\mathrm{sign}}_\pi, \Sigma^{\mathring{M}}(G_\Phi))$. 

\begin{lem}\label{lemK_ext_phi_sign}
Let $(M^{\mathrm{ev}},\pi : \partial M \to Y^{\mathrm{odd}})$ be a compact manifold with fibered boundary, equipped with orientations on $TM$ and $T^V\partial M$. 

Choose any metric on $\mathfrak{A}G_\Phi$ which has a direct sum decomposition at the boundary, and denote the associated signature operator on $G_\Phi|_{\partial M} = \partial M \times_\pi \partial M \times_\pi TY \times \mathbb{R} \rightrightarrows \partial M$ as, 
\[
D^{\mathrm{sign}}_{\Phi, \partial} := D^{\mathrm{sign}}_\pi \hat{\otimes} 1 + 1 \hat{\otimes} D^{\mathrm{sign}}_{TY \times \mathbb{R}}. 
\]
	Let $\Psi(D^{\mathrm{sign}}_{\Phi, \partial})$ denote the $\mathbb{Z}_2$-graded $C^*$-subalgebra of $\overline{\Psi^0_c(G_\Phi|_{\partial M}; \wedge_{\mathbb{C}}(\mathfrak{A}G_\Phi|_{\partial M})^*)}$ generated by $\{\psi(D^{\mathrm{sign}}_{\Phi, \partial})\}$, $C(\partial M)$ and $C^*(G_\Phi|_{\partial M}; \wedge_{\mathbb{C}}(\mathfrak{A}G_\Phi|_{\partial M})^*)$. 
This $C^*$-algebra fits into the graded exact sequence 
\[
0 \to C^*(G_\Phi|_{\partial M}; \wedge_{\mathbb{C}}(\mathfrak{A}G_\Phi|_{\partial M})^*) \to \Psi(D^{\mathrm{sign}}_{\Phi, \partial}) \to C(\partial M)\otimes \mathbb{C}l_1 \to 0. 
\]
The connecting element of this extension coincides with the class $\pi^{\mathrm{sign}}_! \otimes_{C(Y)}\mathrm{Ind}^Y(D^{\mathrm{sign}}_{TY \times \mathbb{R}})\in KK(C(\partial M), C^*(G_\Phi|_{\partial M}))$. 
\end{lem}

\begin{defn}[$\mathcal{B}^{\mathrm{sign}}_\pi$]\label{def_B_pi_sign}
We define a graded $C^*$-algebra by the pullback (c.f. Remark \ref{rem_rigid})
\[
\xymatrix{
\mathcal{B}^{\mathrm{sign}}_\pi \ar[r] \ar[d] \pullbackcorner & \Psi(D^{\mathrm{sign}}_{\Phi, \partial}) \ar[d] \\
C(M)\otimes \mathbb{C}l_1 \ar[r]^{i^*} &  C(\partial M) \otimes \mathbb{C}l_1
}
\]
Choosing a metric of $G_\Phi$, we get a canonical injective $*$-homomorphism $\iota : \mathcal{B}^{\mathrm{sign}}_\pi \to \Sigma^{\mathring{M}}(G_\Phi; \wedge_{\mathbb{C}}(\mathfrak{A}G_\Phi)^*)$. 
The $KK$-element $[\iota] \in KK(\mathcal{B}^{\mathrm{sign}}_\pi, \Sigma^{\mathring{M}}(G_\Phi))$ is independent of the choice of a metric.  
\end{defn}

We are going to construct a $KK$-element $\mu^{\mathrm{sign}} \in KK(\mathcal{A}_\pi^{\mathrm{sign}}, \mathcal{B}_\pi^{\mathrm{sign}})$ analogous to $\mu$ in Proposition \ref{prop_asymp_var}. 
Note that in the signature case, this element is not a $KK$-equivalence.

We construct a $C^*$-algebra $\mathcal{D}_\pi^{\mathrm{sign}}$, using the $C^*$-algebra $\mathcal{C}$ constructed in the last subsection. 
Recall that the construction of $\mathcal{C}$ does not need any $spin^c$-structure on vector bundle $V \to Y$. 
We apply the constructions in the last subsection for $V = T^*Y$. 
Denote $\Gamma$ the groupoid $ \partial M \times_\pi \partial M \rightrightarrows \partial M$. 

Choosing any riemannian metrics on $T^V\partial M$ and $TY$, define the $C^*$-algebra $\Psi(\mathbb{B}_V \hat{\otimes} 1 + 1 \hat{\otimes} D^{\mathrm{sign}}_\pi)$ to be the $C^*$-subalgebra of $\mathcal{M}(\mathcal{C}\hat{\otimes} C^*(\Gamma; \wedge_{\mathbb{C}}(\mathfrak{A}\Gamma)^*))$ generated by $\{\psi (\mathbb{B}_V \hat{\otimes} 1 + 1 \hat{\otimes} D^{\mathrm{sign}}_\pi)\}$, $\mathcal{C} \hat{\otimes} C^*(\Gamma; \wedge_{\mathbb{C}}(\mathfrak{A}\Gamma)^*)$ and $C(\partial M \times [0,\infty])$. 
Note that we have an exact sequence
\[
0 \to \mathcal{C} \hat{\otimes} C^*(\Gamma; \wedge_{\mathbb{C}}(\mathfrak{A}\Gamma)^*) \to \Psi(\mathbb{B}_V \hat{\otimes} 1 + 1 \hat{\otimes} D^{\mathrm{sign}}_\pi) \to C(\partial M \times [0, \infty]) \otimes \mathbb{C}l_1 \to 0. 
\]
\begin{defn}[$\mathcal{D}_\pi^{\mathrm{sign}}$]
In the above settings, we define the $C^*$-algebra $\mathcal{D}^{\mathrm{sign}}_\pi$ by the pullback (c.f. Remark \ref{rem_rigid})
\[
\xymatrix{
\mathcal{D}^{\mathrm{sign}}_\pi \ar[r] \ar[d] \pullbackcorner & \Psi(\mathbb{B}_V \hat{\otimes} 1 + 1 \hat{\otimes} D^{\mathrm{sign}}_\pi) \ar[d] \\ 
C(M \times [0, \infty]) \otimes \mathbb{C}l_1 \ar[r]^{i^*} &  C(\partial M \times [0, \infty]) \otimes \mathbb{C}l_1
}
\]
\end{defn}
Analogously to the last subsection, 
the $*$-homomorphism $ev_0^{\mathcal{C}} : \mathcal{C} \to C(Y)$ induces a $*$-homomorphism $ev_0 : \mathcal{C}\hat{\otimes}C^*(\Gamma; \wedge_{\mathbb{C}}(\mathfrak{A}\Gamma)^*) \to C^*(\Gamma; \wedge_{\mathbb{C}}(\mathfrak{A}\Gamma)^*) $. 
This $*$-homomorphism extends to a $*$-homomorphism $\mathcal{D}^{\mathrm{sign}}_\pi \to \mathcal{A}^{\mathrm{sign}}_\pi$, by sending $ \psi(\mathbb{B}_V \hat{\otimes}1 + 1 \hat{\otimes}D^{\mathrm{sign}}_\pi)$ to $\psi(D^{\mathrm{sign}}_\pi)$ and evaluation at $0$ on $C(M \times [0, \infty]) \otimes \mathbb{C}l_1$. 
We denote this $*$-homomorphism by $ev_0$. 

On the other hand, contrary to the last subsection, the $*$-homomorphism $ev_\infty^{\mathcal{C}} : \mathcal{C} \to C_0(\mathbb{R}; \mathbb{C}l_1) \hat{\otimes} C_0(V ; \mathbb{C}l(V)) $ induces a $*$-homomorphism $ev_\infty : \mathcal{C}\hat{\otimes}_{C(Y)}C^*(\Gamma; \wedge_{\mathbb{C}}(\mathfrak{A}\Gamma)^*) \to C_0(V\oplus \mathbb{R}; \mathbb{C}l(V\oplus \mathbb{R})) \hat{\otimes}_{C(Y)} C^*(\Gamma, \wedge_{\mathbb{C}}(\mathfrak{A}\Gamma)^*)$, and the range of this homomorphism is not isomorphic to $C^*(G_\Phi|_{\partial M}; \wedge_{\mathbb{C}}(\mathfrak{A}G_\Phi|_{\partial M})^*)$. 
To overcome this difference, we need an intermediate $C^*$-algebra $\tilde{\mathcal{B}}_\pi^{\mathrm{sign}}$. 

\begin{defn}[$\tilde{\mathcal{B}}_\pi^{\mathrm{sign}}$]
Let us denote $V = T^*Y$. 
Consider the $C^*$-algebra $C_0(V\oplus \mathbb{R}; \mathbb{C}l(V\oplus \mathbb{R})) \hat{\otimes}_{C(Y)} C^*(\Gamma, \wedge_{\mathbb{C}}(\mathfrak{A}\Gamma)^*)$ and the unbounded multiplier $C_{V \oplus \mathbb{R}} \hat{\otimes} 1 + 1 \hat{\otimes} D_{\pi}^{\mathrm{sign}} $ of this $C^*$-algebra. 
Let $\Psi(C_{V \oplus \mathbb{R}} \hat{\otimes} 1 + 1 \hat{\otimes} D_{\pi}^{\mathrm{sign}} )$ denote the $\mathbb{Z}_2$-graded $C^*$-subalgebra of the multiplier algebra $\mathcal{M}(C_0(V\oplus \mathbb{R}; \mathbb{C}l(V\oplus \mathbb{R})) \hat{\otimes}_{C(Y)} C^*(\Gamma, \wedge_{\mathbb{C}}(\mathfrak{A}\Gamma)^*))$, generated by $\{\psi(C_{V \oplus \mathbb{R}} \hat{\otimes} 1 + 1 \hat{\otimes} D_{\pi}^{\mathrm{sign}})\}$, $C(\partial M)$ and $C_0(V\oplus \mathbb{R}; \mathbb{C}l(V\oplus \mathbb{R})) \hat{\otimes}_{C(Y)} C^*(\Gamma, \wedge_{\mathbb{C}}(\mathfrak{A}\Gamma)^*)$. 
This $C^*$-algebra fits into the graded exact sequence 
\[
0 \to C_0(V\oplus \mathbb{R}; \mathbb{C}l(V\oplus \mathbb{R})) \hat{\otimes}_{C(Y)} C^*(\Gamma, \wedge_{\mathbb{C}}(\mathfrak{A}\Gamma)^*) \to \Psi(C_{V \oplus \mathbb{R}} \hat{\otimes} 1 + 1 \hat{\otimes} D_{\pi}^{\mathrm{sign}}) \to C(\partial M)\otimes \mathbb{C}l_1 \to 0. 
\]
We define the $C^*$-algebra $\tilde{\mathcal{B}}_\pi^{\mathrm{sign}}$ by the pullback (c.f. Remark \ref{rem_rigid})
\[
\xymatrix{
\tilde{\mathcal{B}}^{\mathrm{sign}}_\pi \ar[r] \ar[d] \pullbackcorner & \Psi(C_{V \oplus \mathbb{R}} \hat{\otimes} 1 + 1 \hat{\otimes} D_{\pi}^{\mathrm{sign}} ) \ar[d] \\
C(M)\otimes \mathbb{C}l_1 \ar[r]^{i^*} &  C(\partial M) \otimes \mathbb{C}l_1
}
\]
\end{defn}

Next, we construct a $*$-homomorphism $b:\tilde{\mathcal{B}}_\pi^{\mathrm{sign}} \to {\mathcal{B}}_\pi^{\mathrm{sign}}$. 
Since the vector bundle $\wedge_{\mathbb{C}}(TY\oplus \mathbb{R})^* \to Y $ is a $\mathbb{C}l(TY \oplus \mathbb{R})$-module bundle, 
if we define a spin structure on $TY \oplus \mathbb{R}$ locally, we can write
\[
\wedge_{\mathbb{C}}(TY\oplus \mathbb{R})^* = S(TY \oplus \mathbb{R}) \hat{\otimes} W, 
\]
with some $\mathbb{Z}_2$-graded hermitian vector bundle $W$, and the Clifford multiplication can be written as $c \hat{\otimes} 1$. 
Note that the vector bundle $\mathrm{End}(W)$ is canonically defined independently on the chosen local spin structure, and extends to a vector bundle over the whole $Y$, still denoted by $\mathrm{End}(W)$. 
Under the Fourier transform, the operator $D_{TY \oplus \mathbb{R}}^{\mathrm{sign}}$ corresponds to the unbounded multiplier $C_{V \oplus \mathbb{R}}\hat{\otimes}1$ of $C_0(V \oplus \mathbb{R}; \mathrm{End}(\wedge_{\mathbb{C}}(TY\oplus \mathbb{R})^*)) \simeq C_0(V \oplus \mathbb{R}; \mathbb{C}l(V\oplus \mathbb{R})\hat{\otimes}\mathrm{End}(W))$. 

Thus the $*$-homomorphism
\[
id_{\mathbb{C}l(V\oplus \mathbb{R})}\hat{\otimes}1_{W} : C_0(V \oplus \mathbb{R}; \mathbb{C}l(V\oplus \mathbb{R})) \to C_0(V \oplus \mathbb{R}; \mathbb{C}l(V\oplus \mathbb{R})\hat{\otimes}\mathrm{End}(W)), 
\]
induces the $*$-homomorphism
\[
b' : \Psi(C_{V \oplus \mathbb{R}} \hat{\otimes} 1 + 1 \hat{\otimes} D_{\pi}^{\mathrm{sign}} ) \to \Psi(D^{\mathrm{sign}}_{\Phi, \partial})
\]
by sending $\psi(C_{V \oplus \mathbb{R}} \hat{\otimes} 1 + 1 \hat{\otimes} D_{\pi}^{\mathrm{sign}})$ to $\psi(D^{\mathrm{sign}}_{\Phi, \partial})$, and induces the desired $*$-homomorphism
\[
b:\tilde{\mathcal{B}}_\pi^{\mathrm{sign}} \to {\mathcal{B}}_\pi^{\mathrm{sign}}. 
\]
Using this intermediate algebra, we easily see the following proposition, which is the signature version of Proposition \ref{prop_asymp_var}. 
\begin{prop}\label{prop_asymp_var_sign}
Consider the following commutative diagram, 
\[
\xymatrix{
0 \ar[r] &C^*(\Gamma; \wedge_{\mathbb{C}}(\mathfrak{A}\Gamma)^*) \ar[r] &\mathcal{A}^{\mathrm{sign}}_\pi \ar[r]& C(M ) \otimes \mathbb{C}l_1\ar[r]&0
\\
0 \ar[r]  &\mathcal{C}\hat{\otimes}_{C(Y)}C^*(\Gamma; \wedge_{\mathbb{C}}(\mathfrak{A}\Gamma)^*)  \ar[u]^{ev_0}\ar[d]^{ev_\infty} \ar[r] &\mathcal{D}^{\mathrm{sign}}_\pi \ar[r]\ar[u]^{ev_0}\ar[d]^{ev_\infty} & C(M \times [0, \infty] ) \otimes \mathbb{C}l_1\ar[r] \ar[u]^{ev_0}\ar[d]^{ev_\infty}&0
\\
0 \ar[r]  &C_0(V\oplus \mathbb{R}; \mathbb{C}l(V\oplus \mathbb{R})) \hat{\otimes}_{C(Y)} C^*(\Gamma, \wedge_{\mathbb{C}}(\mathfrak{A}\Gamma)^*) \ar[r]\ar[d]^{\hat{\otimes}1_W\hat{\otimes}id_{C^*(\Gamma)}} &\tilde{\mathcal{B}}^{\mathrm{sign}}_\pi \ar[r] \ar[d]^b&C( M ) \otimes \mathbb{C}l_1\ar[r]\ar@{=}[d] &0 \\
0 \ar[r]  &C^*(G_\Phi|_{\partial M}; \wedge_{\mathbb{C}}(\mathfrak{A}G_\Phi|_{\partial M})^*) \ar[r] &\mathcal{B}^{\mathrm{sign}}_\pi \ar[r] &C( M ) \otimes \mathbb{C}l_1\ar[r] &0
}
\]
where the rows are exact. 
The arrows connecting the first, second and third rows, denoted by $ev_0$ and $ev_\infty$, are $KK$-equivalences, and we define
\begin{equation}\label{prop_asymp_statement}
\mu^{\mathrm{sign}} := [ev_0]^{-1} \otimes_{\mathcal{D}^{\mathrm{sign}}_\pi} [ev_\infty]\otimes_{\tilde{\mathcal{B}}^{\mathrm{sign}}_\pi}[b] \in KK(\mathcal{A}^{\mathrm{sign}}_\pi, \mathcal{B}^{\mathrm{sign}}_\pi). 
\end{equation}
Then this element fits into the commutative diagram in $KK$-theory, 
\[
\xymatrix{
0 \ar[r] &C^*(\Gamma; \wedge_{\mathbb{C}}(\mathfrak{A}\Gamma)^*) \ar[r] \ar[d]^{\mathrm{Ind}^Y(D^{\mathrm{sign}}_{TY \times \mathbb{R}})\otimes id_{C^*(\Gamma)}} &\mathcal{A}^{\mathrm{sign}}_\pi \ar[r] \ar[d]^{\mu^{\mathrm{sign}}}& C(M ) \otimes \mathbb{C}l_1\ar[r]\ar@{=}[d] &0
\\
0 \ar[r]  &C^*(G_\Phi|_{\partial M}; \wedge_{\mathbb{C}}(\mathfrak{A}G_\Phi|_{\partial M})^*) \ar[r]  &\mathcal{B}^{\mathrm{sign}}_\pi \ar[r] &C( M ) \otimes \mathbb{C}l_1\ar[r] &0.
}
\] 
Here the left vertical arrow is defined by taking Kasparov product of elements $\mathrm{Ind}^Y(D^{\mathrm{sign}}_{TY \times \mathbb{R}}) \in \mathcal{R}KK(Y; C(Y), C^*(TY \times \mathbb{R}))$ and $id_{C^*(\Gamma)} \in \mathcal{R}KK(Y; C^*(\Gamma), C^*(\Gamma))$, and then forgetting the $C(Y)$-algebra structure. 
\end{prop}

Finally we define the element $[M^{\mathrm{sign}}] \in KK(\mathcal{A}^{\mathrm{sign}}_\pi, \Sigma^{\mathring{M}}(G_\Phi))$ as follows. 
\begin{defn}
Let $(M^{\mathrm{ev}},\pi : \partial M \to Y^{\mathrm{odd}})$ be a compact manifold with fibered boundary, equipped with orientations on $TM$ and $T^V\partial M$. 
Then we define 
\[
[M^{\mathrm{sign}}] := \mu^{\mathrm{sign}} \otimes_{\mathcal{B}_\pi^{\mathrm{sign}}} [\iota] \in KK(\mathcal{A}^{\mathrm{sign}}_\pi, \Sigma^{\mathring{M}}(G_\Phi)). 
\]
Here the element $[\iota] \in KK( \mathcal{B}^{\mathrm{sign}}_\pi ,\Sigma^{\mathring{M}}(G_\Phi) )$ is defined in Definition \ref{def_B_pi_sign} and $\mu^{\mathrm{sign}}$ is defined in Proposition \ref{prop_asymp_var_sign}. 
\end{defn}

Then, we can describe the $\Phi$-signature as follows. 
\begin{thm}[The index pairing formula for signature operators]\label{main_thm_K_sign}
Let $(M,\pi : \partial M \to Y)$ be a compact even dimensional manifold with fibered boundary, equipped with orientations on $TM$ and $T^V\partial M$. 
Let $E$ be a complex vector bundle over $M$.  
Let $Q_\pi \in \mathcal{I}^{\mathrm{sign}}(\pi, E)$. 
Then we have
\[
\mathrm{Sign}_\Phi(M, E, Q_\pi)= [(E, Q_\pi)] \otimes_{\mathcal{A}_\pi^{\mathrm{sign}}} [M^{\mathrm{sign}}] \otimes_{\Sigma^{\mathring{M}}(G_\Phi)} \mathrm{ind}^{\mathring{M}}(G_\Phi) \in \mathbb{Z}. 
\]
\end{thm}

\begin{proof}
As in the proof of Theorem \ref{main_thm_K}, it is enough to prove the following. 
Given an element $Q_\pi \in \mathcal{I}^{\mathrm{sign}}(\pi, E)$, take some representative $\tilde{D}_\pi^{\mathrm{sign}, E}$ for $Q_\pi$ and consider the class $Q_{\Phi, \partial} := [\tilde{D}_\pi^{\mathrm{sign}, E} \hat{\otimes} 1 +1 \hat{\otimes} D^{\mathrm{sign}}_{TY \times \mathbb{R}} ] \in \mathcal{I}({D}_\pi^{\mathrm{sign}, E} \hat{\otimes} 1 +1 \hat{\otimes} D^{\mathrm{sign}}_{TY \times \mathbb{R}})$ (this class does not depend on the choice). 
Then we have
\[
[(E, Q_\pi) ]\otimes_{\mathcal{A}^{\mathrm{sign}}_\pi} \mu^{\mathrm{sign}} = [(E, Q_{\Phi, \partial})] \in KK(\mathbb{C}l_1, \mathcal{B}^{\mathrm{sign}}_\pi). 
\]

For simplicity we assume $E$ is the trivial bundle. 
For a given representative $\tilde{D}^{\mathrm{sign}}_\pi$ for $Q_\pi \in \mathcal{I}^{\mathrm{sign}}(\pi)$, we can construct an invertible element in $\mathcal{D}_\pi$ defined as $(1 \hat{\otimes} \epsilon, \psi(\mathbb{B}_V \hat{\otimes} 1 + 1 \hat{\otimes} \tilde{D}^{\mathrm{sign}}_\pi)) \in \mathcal{D}_\pi$. 
By definition we have 
\begin{align*}
ev_0((1 \hat{\otimes} \epsilon, \psi(\mathbb{B}_V \hat{\otimes} 1 + 1 \hat{\otimes} \tilde{D}^{\mathrm{sign}}_\pi))) &=(1 \hat{\otimes}\epsilon, \psi(\tilde{D}^{\mathrm{sign}}_\pi)) &\in \mathcal{A}^{\mathrm{sign}}_\pi \\ 
b \circ ev_\infty((1 \hat{\otimes} \epsilon, \psi(\mathbb{B}_V \hat{\otimes} 1 + 1 \hat{\otimes} \tilde{D}^{\mathrm{sign}}_\pi))) &=b((1 \hat{\otimes}\epsilon, \psi(C_{V \oplus \mathbb{R}}\hat{\otimes}1 + 1 \hat{\otimes}\tilde{D}^{\mathrm{sign}}_\pi))) & \\
& =(1 \hat{\otimes}\epsilon, \psi(D^{\mathrm{sign}}_{TY \oplus \mathbb{R}}\hat{\otimes}1 + 1 \hat{\otimes}\tilde{D}^{\mathrm{sign}}_\pi)) & \in {\mathcal{B}}^{\mathrm{sign}}_\pi. 
\end{align*}
Thus we see that $[(1 \hat{\otimes} \epsilon, \psi(\mathbb{B}_V \hat{\otimes} 1 + 1 \hat{\otimes} \tilde{D}^{\mathrm{sign}}_\pi)) ]\otimes[ev_0] = [(\underline{\mathbb{C}}, Q_\pi)] \in K_1(\mathcal{A}^{\mathrm{sign}}_\pi)$ and $[(1 \hat{\otimes} \epsilon, \psi(\mathbb{B}_V \hat{\otimes} 1 + 1 \hat{\otimes} \tilde{D}^{\mathrm{sign}}_\pi)) ] \otimes[ev_\infty] \otimes [b]= [(\underline{\mathbb{C}}, Q_{\Phi, \partial})] \in K_1(\mathcal{B}^{\mathrm{sign}}_\pi)$. 
By Proposition \ref{prop_asymp_var_sign}, we see that $[(\underline{\mathbb{C}}, Q_\pi)]  \otimes_{\mathcal{A}^{\mathrm{sign}}_\pi} \mu^{\mathrm{sign}} = [(\underline{\mathbb{C}}, Q_\pi)] \otimes_{\mathcal
{A}^{\mathrm{sign}}_\pi} [ev_0]^{-1} \otimes_{\mathcal{D}^{\mathrm{sign}}_\pi} [ev_\infty] \otimes_{\tilde{\mathcal{B}}_\pi}[b]= [(\underline{\mathbb{C}}, Q_{\Phi, \partial})] \in K_1(\mathcal{B}^{\mathrm{sign}}_\pi)$, so we get the result. 
\end{proof}
\section{The local Signature}\label{sec_localsignature}
\subsection{Settings}\label{settings}
The settings for the localization problem for signature are the following. 
\begin{itemize}
\item Let $F$ be an oriented closed even dimensional smooth manifold. 
\item Let $G$ be a subgroup of the orientation-preserving diffeomorphism group $\mbox{Diff}^+(F) $ of $F$. 
\item Let $Z \subset BG$ be a subspace of the classifying space of $G$. 
A particular case of interest is when $Z$ is the $k$-skeleton of a $CW$-complex model of $BG$ for some integer $k$. 
\end{itemize}
We can define the universal family signature class $\mathrm{Sign}(F_{\mathrm{univ}}) \in K^0(BG)$, where $K^0(BG)$ denotes the representable $K$-theory of $BG$ (see subsection \ref{subsec_repK}). 
This class is constructed by considering the fiberwise signature operators on the universal $F$-fiber bundle $\pi_{\mathrm{univ}}$ over $BG$. 
This construction is explained in detail in subsection \ref{subsec_univ_index_invertible_perturbation}. 
Moreover, if there exists a positive integer $n$ such that the restriction of $\mathrm{Sign}(F_{\mathrm{univ}})$ to $Z$ is $n$-torsion in $K^0(Z)$, the set of homotopy equivalence classes of $\mathbb{C}l_1$-invertible perturbations of $n$-direct sum of signature operators, $ \mathcal{I}^{\mathrm{sign}}(\pi_{\mathrm{univ}}|_Z, \underline{\mathbb{C}^n})$, is nonempty and has a canonical affine space structure modeled on $K^{-1}(Z)$. 

\begin{defn}
Let $S_{F, G, Z}$ denote the set of isomorphism classes of pairs $(M, \pi : \partial M \to Y)$ satisfying the following conditions: 
\begin{itemize}
\item The pair $(M, \pi : \partial M \to Y)$ is a compact oriented manifold with fibered boundaries, and assume that $M$ is even dimensional. 
\item Assume that $\pi$ is an $F$-fiber bundle structure with structure group $G$. 
\item Assume that $i_* : [Y, Z] \to [Y, BG]$, induced by the inclusion $i : Z \to BG$, is an isomorphism. 
\end{itemize}
\end{defn}

Our main theorem of this section is the following. 
\begin{thm}\label{mainthm}
Assume that a positive integer $n$ satisfies $n\cdot i^*\mathrm{Sign}(F_{\mathrm{univ}}) = 0 \in K^0(Z)$. 
For each element $Q_{Z} \in \mathcal{I}^{\mathrm{sign}}(\pi_{\mathrm{univ}}|_Z, \underline{\mathbb{C}^n})$, we  have a natural map
\[
\sigma_{Q_Z} : S_{F, G, Z} \to \frac{\mathbb{Z}}{n}
\]
such that the following holds. 
\begin{itemize}
\item (vanishing)

For $(M, \pi) \in S_{F, G, Z}$, if there exist a compact oriented manifold with boundary $(X, \partial X)$ with a fixed diffeomorphism $\partial X \simeq Y$, and an $F$-fiber bundle structure $\pi' : M \to X$ with structure group $G$ which satisfies $\pi'|_{\partial M} =\pi $ such that $i_* : [X, Z] \to [X, BG]$ is surjective, then we have 
\[
\sigma_{Q_Z}(M, \pi)=0. 
\]
\item (additivity) 

For $(M_0, \pi_0)$ and $(M_1, \pi_1)$ in $S_{F, G, Z}$, assume that there exists a decomposition $\partial M_i = \partial M_i^+ \sqcup -\partial M_i^-$ for $i = 0, 1$, 
and there exists an isomorphism of the fiber bundle $\phi : \pi_0|_{\partial M_0^+} \simeq \pi_1|_{- \partial M_1^-}$. 
We can form $(M, \pi) = (M_0, \pi_0)\cup_\phi (M_1, \pi_1) \in S_{F, G, Z}$. 
Then we have 
\[
\sigma_{Q_Z}(M, \pi) = \sigma_{Q_Z}(M_0, \pi_0) + \sigma_{Q_Z}(M_1, \pi_1). 
\]
\item (compatibility with signature) 

An oriented even dimensional closed manifold $M$ can be regarded as an element in $S_{F, G, Z}$. 
For this element, we have
\[
\sigma_{Q_Z}(M) = \mathrm{Sign}(M). 
\]
\end{itemize}
Moreover, if we have two elements $Q_Z^0, Q_Z^1 \in  \mathcal{I}^{\mathrm{sign}}(\pi_{\mathrm{univ}}|_Z, \underline{\mathbb{C}^n})$, the difference between $\sigma_{Q_Z^0}$ and $\sigma_{Q_Z^1}$ is described as follows. 
Let $(M, \pi : \partial M \to Y)$ be an element in $ S_{F, G, Z}$, and denote the classifying map for $\pi$ by $[f] \in [Y, Z] \simeq [Y, BG]$. 
Recall that we have the difference class $[Q_Z^1 - Q_Z^0]\in K^{-1}(Z)$. 
Denote the $K$-homology class of signature operator on $Y$ by $[D_{Y}^{\mathrm{sign}}] \in K_1(Y)$.  
We have, for each $(M, \pi ) \in S_{F, G, Z}$, 
\begin{equation}\label{mainthm_statement_eq}
\sigma_{Q_Z^1}(M, \pi) - \sigma_{Q_Z^0}(M, \pi) = \frac{2}{n}\langle f^*([Q_Z^1 - Q_Z^0]), [D_Y^{\mathrm{sign}}]\rangle. 
\end{equation}
Here we denoted the index pairing by $\langle{\cdot}, {\cdot}\rangle : K^1(Y)\otimes K_1(Y) \to \mathbb{Z}$. 
\end{thm}

\subsection{The universal index class and the pullback of $\mathbb{C}l_1$-invertible perturbations}\label{subsec_univ_index_invertible_perturbation}

Let $F$ be an oriented closed even dimensional smooth manifold and $G \subset \mathrm{Diff}^+(F)$ be a subgroup. 
We define the universal signature class $\mathrm{Sign}(F_{\mathrm{univ}}) \in K^0(BG)$ as follows. 
We have the universal $F$-fiber bundle over $BG$, 
\[
\pi_{\mathrm{univ}} : EG \times_G F \to BG. 
\]
Fix any continuous family of fiberwise smooth metrics $g_{\mathrm{univ}}$ over this fiber bundle. 
Then this defines a $\mathbb{Z}_2$-graded Hilbert bundle (see Definition \ref{def_hilbertbundle})
\[
\hat{\mathcal{H}}_{\mathrm{univ}} := \{L^2(\pi_{\mathrm{univ}}^{-1}(x); \wedge_{\mathbb{C}}T^*F_{g_{univ, x}} )\}_{x \in BG} \to BG. 
\]
Also the metric $g_{\mathrm{univ}}$ defines the fiberwise signature operator $D_{\mathrm{univ}}^{\mathrm{sign}}$ acting on $\hat{\mathcal{H}}_{\mathrm{univ}}$, and the bounded transform of this operator, $\psi(D_{\mathrm{univ}}^{\mathrm{sign}})$, gives an element $\psi(D_{\mathrm{univ}}^{\mathrm{sign}}) \in \Gamma(BG; \mathrm{Fred}^{(0)}(\hat{\mathcal{H}}_{\mathrm{univ}}))$ (see subsection \ref{subsec_invertibleperturbation}). 
So this operator defines a class 
\[
\mathrm{Sign}_{\mathrm{univ}}(F) := [\psi(D_{\mathrm{univ}}^{\mathrm{sign}})] \in K^0(BG),
\]
where the symbol $K^0$ denotes the representable $K$-theory. 
Since the space of choices of fiberwise metric $g_{\mathrm{univ}}$ is contractible, this class does not depend on the choice of $g_{\mathrm{univ}}$.  

Given a subset $Z \subset BG$, we can restrict this class and get the universal signature class over $Z$, 
\[
\mathrm{Sign}_{\mathrm{univ}}(F) := i^* \mathrm{Sign}_{\mathrm{univ}}(F) \in K^0(Z), 
\]
where $i : Z \to BG$ denotes the inclusion.  
We abuse the notation and use the same symbol for the universal signature class over $Z$ and $BG$ without any confusion. 

Next we give fundamental remarks on pullbacks of $\mathbb{C}l_1$-invertible perturbations for signature operators. 
Suppose that we are given a continuous map $f : X_0 \to X_1$ between topological spaces, and a continuous oriented fiber bundle structure $\pi : M \to X_1$ with fiber $F$. 
The fiberwise signature class, $\mathrm{Sign}(\pi) \in K^0(X_1)$, is defined as above. 
Consider the pullback bundle $f^*\pi : f^*(M) \to X_0$. 
The fiberwise signature class of this bundle satisfies $\mathrm{Sign}(f^*\pi)  = f^*\mathrm{Sign}(\pi) \in K^0(X_0)$.  

Suppose that we have $\mathrm{Sign}(\pi) = 0\in K^0(X_1)$. 
Then by Lemma \ref{perturbationexistence}, we see that the set $\mathcal{I}^{\mathrm{sign}}(\pi)$ is nonempty (see Remark \ref{rem_twisted_spinc}; it is easy to see the remark applies to the case where we do not assume smooth structure on the base $X_1$). 
By the contractibility of the space of fiberwise metrics, the pullback map, 
\[
f^* : \mathcal{I}^{\mathrm{sign}}(\pi) \to \mathcal{I}^{\mathrm{sign}}(f^*\pi), 
\]
is well-defined. 
Moreover we can easily see that this map is actually a homomorphism of affine spaces, with respect to the homomorphism $f^* : K^{-1}(X_1) \to K^{-1}(X_0)$. 

We can also generalize this construction to the signature operators twisted by the trivial rank $n$ bundle $\underline{\mathbb{C}^n}$, i.e., the $n$-direct sum of signature operators. 
Suppose that $n \cdot\mathrm{Sign}(f^*\pi)  = f^*\mathrm{Sign}(\pi) \in K^0(X_0) $. 
Then the set $\mathcal{I}^{\mathrm{sign}}(\pi, \underline{\mathbb{C}^n})$ is nonempty, and we have a well-defined affine space homomorphism
\begin{equation}\label{rigid_pullback_eq}
f^* : \mathcal{I}^{\mathrm{sign}}(\pi, \underline{\mathbb{C}^n}) \to \mathcal{I}^{\mathrm{sign}}(f^*\pi,\underline{\mathbb{C}^n}). 
\end{equation}

\subsection{The local signature}
In this subsection, we return to the settings of subsection \ref{settings}. 
First we explain the pullback of $\mathbb{C}l_1$-invertible perturbations by the classifying maps. 
We cannot apply the procedure explained in the last section directly, because the classifying map is defined {\it up to homotopy}. 

\begin{prop}
Suppose we are given a $F$-fiber bundle $\pi : M \to X$ over a topological space $X$ with structure group $G$. 
\begin{enumerate}
\item Suppose that the universal signature class satisfies $n \cdot \mathrm{Sign}(F_{\mathrm{univ}}) = 0 \in K^0(BG)$. 
Then the classifying map $[f] \in [X, BG]$ induces a well-defined affine space homomorphism
\[
[f]^* : \mathcal{I}^{\mathrm{sign}}(\pi_{\mathrm{univ}}, \underline{\mathbb{C}^n}) \to \mathcal{I}^{\mathrm{sign}}(\pi,\underline{\mathbb{C}^n}). 
\]
\label{pullback_prop_1}
\item Let $Z \subset BG$ be a subspace and assume that $i_* : [X, Z] \to [X, BG]$ induced by the inclusion $i : Z \to BG$ is an isomorphism. 
Also assume that $n \cdot i^* \mathrm{Sign}(F_{\mathrm{univ}}) = 0 \in K^0(Z)$. 
Then the classifying map $[f] \in [X, Z]$ induces a well-defined affine space homomorphism
\[
[f]^* : \mathcal{I}^{\mathrm{sign}}(\pi_{\mathrm{univ}}|_{Z}, \underline{\mathbb{C}^n}) \to \mathcal{I}^{\mathrm{sign}}(\pi,\underline{\mathbb{C}^n}). 
\]
\label{pullback_prop_2}
\end{enumerate}
\end{prop}

\begin{proof}
First we prove the case (\ref{pullback_prop_1}). 
We recall the construction of the classifying map for the fiber bundle $\pi$. 
Let $\tilde{\pi}: P \to X$ be the principal $G$-bundle such that $\pi =( P \times_G F \to X)$. 
We have a fiber bundle
\begin{equation}\label{contractiblebundle}
P \times_G EG \to X. 
\end{equation}
Since the fiber of the bundle (\ref{contractiblebundle}) is contractible, we can take a section $s : X \to P \times_G EG$, and any choice of section is homotopic to each other. 
If we fix a section $s $, we get the associated maps
\begin{align}\label{classifyingmap}
f_s : X \to BG ; & \  x \mapsto \bar{\pi}(s(x)) \\
\phi'_s : P \to EG ; & \ s(\tilde{\pi}(p)) = [p, \phi'_s(p)] \ \mbox{for }p \in P. 
\end{align}
Here we denoted by $\bar{\pi}$ the canonical map $P \times_G EG \to BG$. 
This defines a bundle map $(\phi'_s, f_s) : (P, X) \to (EG, BG)$. 
This induces a bundle map between the associated bundles $\pi : M = P \times_G F \to X$ and $\pi_{\mathrm{univ}} : M_{\mathrm{univ}} = EG \times_G F \to BG$, denoted by $(\phi_s, f_s) : (M, X) \to (M_{\mathrm{univ}}, BG )$.  
Note that fixing a bundle map as above is equivalent to fixing an identification $f_s^*M_{\mathrm{univ}} \simeq M$ as a fiber bundle over $X$. 
As in (\ref{rigid_pullback_eq}), this induces an affine space homomorphism
\[
(\phi_s, f_s)^* :  \mathcal{I}^{\mathrm{sign}}(\pi_{\mathrm{univ}}, \underline{\mathbb{C}^n}) \to \mathcal{I}^{\mathrm{sign}}(\pi,\underline{\mathbb{C}^n}).
\]
Since any two choices of the section $s$ are homotopic, we can easily see that this homomorphism does not depend on the choice of $s$. 

Next we prove the case (\ref{pullback_prop_2}). 
Denote $\pi_{\mathrm{univ}}^{-1}(Z) = \tilde{Z} \subset M_{\mathrm{univ}}$. 
In this case, by the next Lemma \ref{classifyingmap_small}, we see that we can take a section $s : X \to P \times_G \tilde{Z}$, and any choice of section is homotopic to each other. 
Thus we can apply exactly the same argument as in the case (\ref{pullback_prop_1}) and get the result. 
\end{proof}

\begin{lem}\label{classifyingmap_small}
If $i_* : [X, Z] \simeq [X, BG]$, the space $\Gamma(X; P \times_G \tilde{Z})$ is nonempty and path-connected. 
\end{lem}
\begin{proof}
First we prove the nonemptiness. 
Choose a section $s : X \to P \times_G EG$. 
As in (\ref{classifyingmap}), we get the associated maps $f : X \to BG$ and $\phi : P \to EG$. 
Since $i_* : [X, Z] \simeq [X, BG]$, we can find a continuous map
\[
F' : X \times [0, 1] \to BG
\]
such that $F'|_{X \times \{0\}} = f$ and $\mbox{Im}(F'|_{X \times \{1\}}) \subset Z$. 
We denote $F = id_X \times F' : X \times [0, 1] \to X \times BG$. 
Note that $s$ is a lift of $F|_{X \times \{0\}}$ for the fiber bundle $\Pi : P \times_G EG \to X \times BG$. 
By the homotopy lifting property applied to the diagram
\[
\xymatrix{
	& P \times_G EG \ar[d]_\Pi \\
	X \times [0, 1] \ar@{..>}[ru] \ar[r]_F & X \times BG
}
\]
we can lift $F$ to $\tilde{F} : X \times [0, 1] \to P \times_G EG$. 
$\tilde{F}|_{X \times \{1\}} : X \to P \times_G \tilde{Z}$ gives an element in $\Gamma(X; P \times_G \tilde{Z})$. 

Next we prove the path-connectedness. 
Let us denote the canonical projection by $\Pi = (\Pi_X, \Pi_{BG}) : P \times_G EG \to X \times BG$. 
Suppose we are given two sections $s_0, s_1 \in \Gamma(X; P \times_G \tilde{Z})$. 
Since the fiber of the fiber bundle $\Pi_X : P \times_G EG \to X$ is contractible, we can choose a path $\{s_t\}_{t \in [0, 1]}$ in $\Gamma(X; P \times_G EG)$ connecting $s_0$ and $s_1$. 
We denote
\begin{align*}
S : X \times[0,1] \to P \times_G EG &: \ (x, t ) \mapsto s_t(x). 
\end{align*}
$\Pi_X \circ S : X \times [0, 1] \to BG$ satisfies $\mathrm{Im}(\Pi_X \circ S |_{X \times \{0, 1\}}) \subset Z$. 
Since $i_* : [X, Z] \simeq [X, BG]$, we can take a continuous map
\[
\mathcal{F} : X \times [0, 1]_t \times [0, 1]_u \to BG
\]
satisfying $\mathcal{F}|_{X \times [0, 1] \times \{0\}} = \Pi_X \circ S$, 
$\mathcal{F}|_{X \times \{i\} \times \{u\}} = \mathcal{F}|_{X \times \{i\} \times \{0\}}$ for all $u \in [0, 1]$ and $i = 0, 1$, and 
$\mathrm{Im}(\mathcal{F}|_{X \times [0, 1] \times \{1\}}) \subset Z$. 
In the diagram
\[
\xymatrix{
	& P \times_G EG \ar[d]_\Pi \\
	(X \times [0, 1]_t) \times [0, 1]_u \ar@{..>}[ru] \ar[r]_-{id_X \times \mathcal{F}} & X \times BG
}
\]
we see that $S$ is a lift of $(id_X \times \mathcal{F})|_{X \times [0, 1] \times \{0\}}$. 
By the homotopy lifting property, we get a lift of $id_X \times \mathcal{F}$. 
Its restriction to $X \times [0, 1] \times \{1\}$ is a map $X \times [0, 1] \to P \times_G \tilde{Z}$, and by the construction, this gives a path in $\Gamma(X; P \times_G \tilde{Z})$ connecting $s_0$ and $s_1$. 
\end{proof}

We proceed to give a proof of Theorem \ref{mainthm}. 
Under the assumption $n \cdot i^*\mathrm{Sign}(F_{\mathrm{univ}}) = 0 \in K^0(Z)$, the set $\mathcal{I}^{\mathrm{sign}}(\pi_{\mathrm{univ}}|_Z, \underline{\mathbb{C}})$ is nonempty. 
For each $Q_{Z} \in \mathcal{I}^{\mathrm{sign}}(\pi_{\mathrm{univ}}|_Z, \underline{\mathbb{C}})$, we are going to construct a map 
\[
\sigma_{Q_Z} : S_{F, \Gamma, Z} \to \frac{\mathbb{Z}}{n}
\]
satisfying the conditions in the Theorem \ref{mainthm}. 

\begin{defn}
Assume that $n \cdot i^*\mathrm{Sign}(F_{\mathrm{univ}}) = 0 \in K^0(Z)$. 
For a given element $Q_{Z} \in \mathcal{I}^{\mathrm{sign}}(\pi_{\mathrm{univ}}|_Z, \underline{\mathbb{C}^n})$, we define the map
\[
\sigma_{Q_Z}  : S_{F, G, Z} \to \frac{\mathbb{Z}}{n}
\]
as, for $(M, \pi: \partial M \to Y)\in S_{F, G, Z}$, 
\[
\sigma_{Q_Z}(M, \pi : \partial M \to Y) = \frac{1}{n}\mathrm{Sign}_\Phi(M, \underline{\mathbb{C}^n}, [f]^*(Q_{Z})), 
\]
where $[f] \in [Y, Z]$ is the classifying map for the fiber bundle $\pi$. 
\end{defn}

We check that this map satisfies the conditions in Theorem \ref{mainthm}. 

\begin{proof}{(of Theorem \ref{mainthm})}
First we prove the vanishing condition. 
Suppose we are given an element $(M, \pi) \in S_{F, G, Z}$ such that $\pi$ extends to an $F$-fiber bundle structure $\pi' : M \to X$ with structure group $G$ and $i_* : [X, Z] \to [X, BG]$ is surjective. 
Denote the classifying map of $\pi'$ by $[f'] \in [X, BG]$. 
Take any lift of $[f']$ to an element of $[\tilde{f}'] \in [X, Z]$, and realize this map as a bundle map $(\tilde{\phi}', \tilde{f}') : (M, X) \to (\tilde{Z}, Z)$. 
Then as in (\ref{rigid_pullback_eq}), we can pullback the element $Q_{Z} \in \mathcal{I}^{\mathrm{sign}}(\pi_{\mathrm{univ}}|_Z, \underline{\mathbb{C}^n})$ by the bundle map $(\tilde{\phi}', \tilde{f}')$ to get an element $(\tilde{\phi}', \tilde{f}')^*Q_Z \in \mathcal{I}^{\mathrm{sign}}(\pi', \underline{\mathbb{C}^n})$. 
This element restricts to $[f]^*(Q_{Z}) \in \mathcal{I}^{\mathrm{sign}}(\pi, \underline{\mathbb{C}^n})$ at the boundary. 
Thus applying the vanishing proposition, the signature version of Proposition \ref{vanishing_perturbation}, we get the result. 

The additivity follows from the gluing formula, the twisted signature version of Proposition \ref{gluing_perturbation}. 

The compatibility with signature is obvious by definition. 

The equation (\ref{mainthm_statement_eq}) follows from the relative formula for $\Phi$-signature, Proposition \ref{relative_signature}. 

\end{proof}
\section{Examples}\label{sec_example}
In this section, as an application of Theorem \ref{mainthm}, we consider the following localization problem for singular surface bundles. 

Fix a positive integer $k$ and an integer $g \geq 0$. 
Let $M, X$ be $4k, (4k-2)$-dimensional closed oriented smooth manifolds, respectively. 
Let $\pi : M \to X$ be a smooth map and $X = U \cup \cup^m_{i=1} V_i$ be a partition into compact manifolds with closed boundaries, 
i.e., $U$ and $V_i$ are compact manifolds with closed boundaries, and each two of them intersect only on their boundaries. 
Assume $\{V_i\}_i$ are disjoint. 
Denote $M_U := \pi^{-1}(U)$ and $M_i := \pi^{-1}(V_i)$. 
Assume that $\pi|_{M_U} :M_U \to U$ defines a smooth fiber bundle with fiber $\Sigma_g$ (closed oriented surface with genus $g$). 
Then the localization problem is stated as follows: 

\begin{prob}[Localization problem for signature of singular surface bundles]\label{exprob}
Can we define a real number $\sigma (M_i, V_i, \pi|_{M_i}) \in \mathbb{R}$, which only depends on the data $(M_i, V_i, \pi|_{M_i})$, and write
\[
\mathrm{Sign}(M) = \sum_{i = 1}^m \sigma (M_i, V_i, \pi|_{M_i}) \ ?
\]
\end{prob}
We call the real number $\sigma (M_i, V_i, \pi|_{M_i}) $ the {\it local signature}. 
The answer to this problem is positive in the case $g = 0, 1, 2$. 
However the answer is negative for $g \geq 3$, since there exists a smooth $\Sigma_g$-fiber bundle $M \to X$ over a closed surface $X$ with $\mbox{Sign}(M) \neq 0$. 
However, if we assume some structure on the fiber bundle $\pi |_{M_U}$, the answer can be positive. 
There are some examples of ``structures'' for which the localization problem has a positive answer, and the local signatures are constructed and calculated in various areas of mathematics, including topology, algebraic geometry and complex analysis. 
See \cite{AK} and the introduction of \cite{S} for more detailed survey on this problem. 
In this paper, we consider the case $g \geq 2$ and {\it hyperellipticity} (Definition \ref{hypfib}) as the ``structure'' imposed on the regular part of the fibration. 
This is the analogue to the setting in \cite{E}, where the case $k=1$ is considered (note that in \cite{E} the case $g = 0, 1$ are also included).  
We consider the following variant of the Problem \ref{exprob}. 
\begin{prob}\label{localization_prob}
Let $g \geq 2$ and $k$ be fixed as above. 
Assume that $\pi|_{M_U} : M_U \to U$ defines a hyperelliptic $\Sigma_g$-fiber bundle structure (see Definition \ref{hypfib}). 
Can we define a real number $\sigma (M_i, V_i, \pi|_{M_i}) \in \mathbb{R}$, which only depends on the data $(M_i, V_i, \pi|_{M_i})$, and write
\[
\mathrm{Sign}(M) = \sum_{i = 1}^m \sigma (M_i, V_i, \pi|_{M_i}) \ ?
\]
\end{prob} 
We construct such a function $\sigma$ using Theorem \ref{mainthm}.  
We apply the notations in the last section for the following. 
\begin{itemize}
\item Let $F = \Sigma_g$, a closed oriented closed $2$-dimensional manifold of genus $g$. 
\item Denote by $\mathrm{MCG}^+(F)$ the orientation-preserving mapping class group of $F$. 
Let $G=\mathrm{Diff}^H_g := p^{-1}(H_g) \subset \mathrm{Diff}^+(\Sigma_g)$, where $H_g \subset \mathrm{MCG}^+(G)$ is the hyperelliptic mapping class group (Definition \ref{def_Hg}) and $p : \mathrm{Diff}^+(F) \to \mathrm{MCG}^+(F)$ is the quotient map. 
\end{itemize}

\begin{defn}\label{def_Hg}
Let $c \in \mathrm{MCG}^+(\Sigma_g)$ denote the class of hyperelliptic involution (\cite[p.240]{E}) on $\Sigma_g$. 
The hyperelliptic mapping class group, denoted by $H_g$, is defined as follows. 
\[
H_g := \{\gamma \in \mathrm{MCG}^+(\Sigma_g) | \gamma c = c \gamma \}. 
\]
For detailed descriptions of this group, see \cite{E}. 
\end{defn}

\begin{rem}
If $g = 0, 1, 2$, $\mathrm{MCG}^+(\Sigma_g) = H_g$. 
But in the case $g \geq 3$, $H_g$ is a subgroup of infinite index in $\mathrm{MCG}^+(\Sigma_g)$. 
\end{rem}

\begin{defn}\label{hypfib}
Let $X$ be a topological space. 
A $\Sigma_g$-fiber bundle $\pi : M \to X$ with structure group $\mathrm{Diff}^H_g$ is called a {\it hyperelliptic fiber bundle}. 
\end{defn}
We have the following facts about the groups $H_g$ and $\mathrm{Diff}_g^H$. 
\begin{fact}\label{exfact}
\begin{enumerate}
\item The rational group cohomology of $H_g$ satisfies $H^i(H_g; \mathbb{Q}) = 0$ for all $i \geq 1$ (\cite{Kaw}) . 
\label{exfact1}
\item For $g \geq 2$, the unit component of $\mathrm{Diff}_g^H$ is contractible. 
In particular we have a homotopy equivalence between $BH_g$ and $B\mathrm{Diff}_g^H$ (\cite{EE}). 
\label{exfact2}
\item For all $g \geq 0$, $H_g$ is of type $FP_{\infty}$. 
That is, $BH_g$ has a realization as a CW-complex whose $m$-skeleton are finite for all $m \geq 0$. 

It well-known that the mapping class group of an oriented compact surface of genus $g$ with $s$ punctures and $n$ boundary components is of type $FP_\infty$ (For example see \cite{Luc}). 
This case can be seen by noting that an extension of a type $FP_\infty$ group by a type $FP_\infty$ group is also of type $FP_\infty$, and that we have an extension by Birman-Hilden theorem (see \cite[equation (2.1)]{Kaw})
\[
0 \to \mathbb{Z}_2 \to H_g \to \pi_0\mathrm{Diff}^+(S^2, \{(2g+2)-\mbox{punctures}\}) \to 1. 
\]
\label{exfact5}
\end{enumerate}
\end{fact}

\begin{rem}
The reason for assuming $g \geq 2$ in Problem \ref{localization_prob} comes from Fact \ref{exfact} (\ref{exfact2}). 
For $g =0$ the inclusion $SO(3) \to \mathrm{Diff}^H_0(= \mathrm{Diff}^+(S^2))$ is a homotopy equivalence, and for $g = 1$ the unit component of $\mathrm{Diff}^H_1 = \mathrm{Diff}^+(T^2)$ is homotopy equivalent to $S^1 \times S^1$ (\cite{EE}). 
These groups have torsion in group cohomology, so the argument below does not work in these cases. 
\end{rem}

From now on, we fix an integer $g \geq 2$. 
From Fact \ref{exfact} (\ref{exfact2}) and (\ref{exfact5}), we see that $B\mbox{Diff}_g^H$ has a realization as a CW-complex whose $m$-skeleton are finite for all $m \geq 0$. 
We fix such a realization and denote its $m$-skeleton by $Z_{g,m}$. 

\begin{lem}\label{signaturetorsion}
For $g \geq 2$, the universal fiberwise signature class for hyperelliptic fiber bundle, $Sign(F_{\mathrm{univ}} ) \in K^0(B\mathrm{Diff}_g^H)$, maps to $0 \in K^0(B\mathrm{Diff}_g^H; \mathbb{Q})$ under the canonical homomorphism $K^0(B\mathrm{Diff}_g^H) \to K^0(B\mathrm{Diff}_g^H; \mathbb{Q})$. 
Here the symbol $K^0$ denotes the representable $K$-theory. 
\end{lem}
\begin{proof}
We have the rational Chern character isomorphism $Ch : K^{0} (BH_g; \mathbb{Q}) \simeq H^{\mathrm{ev}}(H_g ; \mathbb{Q})$. 
Using the Fact \ref{exfact} (\ref{exfact1}) and (\ref{exfact2}), we have $K^0(B\mathrm{Diff}^H_g; \mathbb{Q}) \simeq K^0(BH_g; \mathbb{Q}) \simeq \mathbb{Q}$ and this isomorphism is given by a map $* \to B\mathrm{Diff}^H_g$. 
Since the manifold $\Sigma_g$ is two dimensional, the virtual rank of the class $\mathrm{Sign}(F_{\mathrm{univ}}) \in K^0(B\mathrm{Diff}^H_g)$ is $0$. 
Thus we have $\mathrm{Sign}(F_{\mathrm{univ}}) = 0 \in K^0(B\mathrm{Diff}^H_g; \mathbb{Q})$ and the lemma follows. 
\end{proof}

For each $m$, we also denote  the restriction of the class $\mathrm{Sign}(F_{\mathrm{univ}})$ to $Z_{g, m}$ by $\mathrm{Sign}(F_{\mathrm{univ}}) \in K^0(Z_{g, m})$. 
By Lemma \ref{signaturetorsion}, we have $\mathrm{Sign}(F_{\mathrm{univ}}) = 0 \in K^0(Z_{g, m}; \mathbb{Q})$. 
Since $Z_{g, m}$ is compact, we have $K^0(Z_{g, m}; \mathbb{Q}) \simeq K^0(Z_{g, m}) \otimes\mathbb{Q}$, so the class $\mathrm{Sign}(F_{\mathrm{univ}}) $ is of finite order in $K^0(Z_{g, m}) $. 
\begin{defn}
For each positive integer $m$, let $n_{g, m}$ denote the order of the class $\mathrm{Sign}(F_{\mathrm{univ}}) \in K^0(Z_{g, m}) $. 
i.e., $n_{g, m}$ is the smallest positive integer satisfying $n_{g, m} \cdot \mathrm{Sign}(F_{\mathrm{univ}}) = 0 \in K^0(Z_{g, m}) $.  
\end{defn}

We are in the situation where Theorem \ref{mainthm} applies. 
\begin{defn}\label{domsigma}
Let $k$ be a positive integer and $g \geq 2$. 
Let $\mathcal{S}_{g, k}$ be the set of isomorphism classes of pairs $(M, \pi : \partial M \to Y)$ such that

\begin{itemize}
\item The pair $(M, \pi : \partial M \to Y)$ is a compact oriented manifold with fibered boundaries, and $M$ is $4k$-dimensional.  
\item The fiber bundle $\pi : \partial M \to Y$ is a hyperelliptic fiber bundle with fiber $\Sigma_g$. 
\end{itemize}
\end{defn}
For $(M, \pi : \partial M \to Y) \in \mathcal{S}_{g, k}$, $Y$ is a $(4k-3)$-dimensional manifold. 
Thus we have $[Y, Z_{g, 4k-2}] \simeq [Y, B\mathrm{Diff}_g^H]$. 
We see that $\mathcal{S}_{g, k} \subset S_{\Sigma_g,\mathrm{Diff}_g^H , Z_{g, 4k-2}}$. 
We apply Theorem \ref{mainthm} to the case $F = \Sigma_g$, $G = \mathrm{Diff}_g^H$, $Z=Z_{g, 4k-2}$, and $n= n_{g, 4k-2}$. 
Note that we have $K^{-1}(Z_{g, 4k-2}) \otimes \mathbb{Q} = 0$ because of Fact \ref{exfact}, (\ref{exfact1}) and (\ref{exfact2}). 
Thus, choosing any element $Q_{Z_{4k-2}} \in \mathcal{I}^{\mathrm{sign}}(\pi_{\mathrm{univ}}|_{Z_{4k-2}}, \underline{\mathbb{C}^n})$, we get the same map $\sigma_{Q_{Z_{4k-2}}}$ by (\ref{mainthm_statement_eq}) in Theorem \ref{mainthm}. 
So we set $\sigma := \sigma_{Q_{Z_{4k-2}}}$. 

\begin{cor}\label{example_cor}
Let $k$ be a positive integer and $g \geq 2$. 
We have a canonical map
\[
\sigma :\mathcal{S}_{g, k}  \to \frac{\mathbb{Z}}{n_{g, 4k-2}}
\]
satisfying the following. 
\begin{itemize}
\item (vanishing)

For $(M, \pi : \partial M \to Y) \in \mathcal{S}_{g, k}$, assume that $\pi$ extends to an oriented hyperelliptic $\Sigma_g$-fiber bundle structure $\pi' : (M, \partial M) \to ( X, \partial X)$ preserving boundaries. 
Here $X$ is an oriented smooth compact oriented $(4k-2)$-dimensional manifold and an orientation preserving diffeomorphism $\partial X \simeq Y$ is fixed. 
Then we have 
\[
\sigma(M, \pi)=0. 
\]
\item (additivity)

For $(M_0, \pi_0)$ and $(M_1, \pi_1)$ in $\mathcal{S}_{g, k}$, assume that there exists a decomposition $\partial M_i = \partial M_i^+ \sqcup -\partial M_i^-$ for $i = 0, 1$, 
and there exists an isomorphism of the fiber bundle $\phi : \pi_0|_{\partial M_0^+} \simeq \pi_1|_{- \partial M_1^-}$. 
We form the union $(M, \pi) = (M_0, \pi_0)\cup_\phi (M_1, \pi_1) \in \mathcal{S}_{g,k}$. 
Then we have 
\[
\sigma(M, \pi) = \sigma(M_0, \pi_0) + \sigma(M_1, \pi_1). 
\]
\item (compatibility with signature)

An oriented $4k$-dimensional closed manifold $M$ can be regarded as an element in $\mathcal{S}_{g, k}$. 
For this element we have
\[
\sigma(M) = \mathrm{Sign}(M). 
\]
\end{itemize}
\end{cor}
This solves Problem \ref{localization_prob} as follows. 
For each $i$, we have $(M_i, \pi|_{\partial M_i} : \partial M_i \to \partial V_i) \in \mathcal{S}_{g, k}$. 
We define $\sigma(M_i, V_i, \pi|_{M_i}) := \sigma(M_i, \pi|_{\partial M_i} : \partial M_i \to \partial V_i)$, where the right hand side is defined by Corollary \ref{example_cor}. 
We have, by the additivity property and the compatibility with signature proved in Corollary \ref{example_cor}, 
\[
\mathrm{Sign}(M) = \sigma(M_U, \pi|_{\partial M_U} : \partial M_U \to \partial U ) + \sum_{i = 1}^m \sigma (M_i, V_i, \pi|_{M_i}). 
\]
On the other hand, by the vanishing property, we have $\sigma(M_U, \pi|_{\partial M_U} : \partial M_U \to \partial U )=0$. 
Thus we get the equality
\[
\mbox{Sign}(M) = \sum_{i = 1}^m \sigma (M_i, V_i, \pi|_{M_i}). 
\] 

\begin{rem}
We remark that Corollary \ref{example_cor} ``solves'' the localization problem, Problem \ref{localization_prob}, in the sense that we have shown the existence of local signature function. 
However, this construction is abstract and does not give an explicit formula for the local signature. 
In contrast, in \cite{E} the author provides an explicit formula for the local signature in the case $k=1$. 
In order to find applications of the above results, we would definitely need to find an explicit formula. 
To proceed further, we need more geometric insight to signature class and their invertible perturbations on mapping class groups. 
In future works, the author hopes to investigate more on this aspect.  
\end{rem}
\section*{Acknowledgment}
This paper is written for master's thesis of the author. 
The author would like to thank her supervisor Yasuyuki Kawahigashi for his support and encouragement. 
She also would like to thank Georges Skandalis, Mikio Furuta and Yosuke Kubota for fruitful advice and discussions. 
This work is supported by Leading Graduate Course for Frontiers of Mathematical Sciences and Physics, MEXT, Japan. 

\end{document}